\documentclass{article}
\usepackage{amsmath}
\usepackage{amsthm}
\usepackage{amssymb}
\usepackage{amscd}
\usepackage{xspace}
\usepackage{verbatim}
\newtheorem{theorem}{Theorem}[section]

\newtheorem{lemma}{Lemma}[section]
\newtheorem{proposition}{Proposition}[section]
\theoremstyle{definition}
\newtheorem{definition}{Definition}[section]

\theoremstyle{remark}
\newtheorem{remark}{Remark}[section]
\numberwithin{equation}{section}
\newcommand{\ov}{\overline}

\newcommand{\e}{\varepsilon}

\newcommand{\G}{\Gamma}

\renewcommand{\O}{\Omega}

\renewcommand{\liminf}{\varliminf}
\renewcommand{\limsup}{\varlimsup}

\renewcommand{\vec}[1]{\mathbf{#1}}
\newcommand{\field}[1]{\mathbb{#1}}

\newcommand{\R}{\field{R}}

\newcommand{\Prod}{\mathop{\prod}\limits}
\newcommand{\er}{\eqref}

\DeclareMathOperator{\Div}{div}

\renewcommand{\O}{\Omega}

\newcommand{\f}{\varphi}
\renewcommand{\vec}[1]{\boldsymbol{#1}}

\date{}

\begin{document}

%%%%%%%%%%%%%%%%%%%%%%%%%%%%%%%%%%%%%%%%%%%%%%%%%%%%%%%%%%%%%%%%%%%%%%%%%%%%%%%%%%%%%%%%%%%%%%%%%%%%%%%%%%%%%%%%%%%%%%%%%%%%%%%%%%%%%%%%%%%%%%%%%%%%%%%%%%%%%%%%%%%%%%%%%%%%%%%%%%%%%%%%%%%%%%%%%%%%%%%%%%%%%%%%%%%%%%%%%%%%%%%%%%%%%%%%%%%%%%%%%%%%%%%%%%%%%%%%%%%%%%%%%%%%%%%%%%%%%%%%%%%%%%%%%%%%%%%%%%%%%%%%%%%%%%%%%%%%%%%%%%%%%%%%%%%%%%%%%%%%%%%%%%%%%%%%%%%%%%%%%%%%%%%%%%%%%%%%%%%%%%%%%%%%%%%%%%%%%
%%%\begin{document}
\title{On the $\Gamma$-limit of singular perturbation problems with optimal profiles which are not one-dimensional. Part II: The lower bound}
\maketitle
\begin{center}
\textsc{Arkady Poliakovsky \footnote{E-mail:
poliakov@math.bgu.ac.il}
}\\[3mm]
Department of Mathematics, Ben Gurion University of the Negev,\\
P.O.B. 653, Be'er Sheva 84105, Israel
\\[2mm]
%\today
%\date{}
\end{center}
\begin{abstract}
In part II we construct the lower bound, in the spirit of
$\Gamma$-$\liminf$ for some general classes of singular perturbation
problems, with or without the prescribed differential constraint,
taking the form $$E_\e(v):=\int_\Omega
\frac{1}{\e}F\Big(\e^n\nabla^n v,...,\e\nabla
v,v\Big)dx\quad\text{for}\;\;
v:\Omega\subset\R^N\to\R^k\;\;\text{such that}\;\; A\cdot\nabla
v=0,$$ where the function $F\geq 0$ and $A:\R^{k\times N}\to\R^m$ is
a prescribed linear operator (for example, $A:\equiv 0$,
$A\cdot\nabla v:=\text{curl}\, v$ and $A\cdot\nabla v=\text{div}\,
v$). Furthermore, we studied the cases where we can easy prove the
coinciding of this lower bound and the upper bound obtained in
\cite{PI}. In particular we find the formula for the $\Gamma$-limit
for the general class of anisotropic problems without a differential
constraint (i.e., in the case $A:\equiv 0$).
\end{abstract}

\section{Introduction}
\begin{definition}
%The asymptotic behavior, when $\varepsilon\to 0$ of the family
Consider a family $\{I_\varepsilon\}_{\varepsilon>0}$ of functionals
$I_\varepsilon(\phi):U\to[0,+\infty]$, where $U$ is a given metric
space. The $\Gamma$-limits of $I_\varepsilon$ are defined by:
\begin{align*}
%\label{fghfi}
(\Gamma-\liminf_{\varepsilon\to 0^+} I_\varepsilon)(\phi)
%=I_1(\phi)
:=\inf\left\{\liminf_{\varepsilon\to
0^+}I_\varepsilon(\phi_\varepsilon):\;\,\{\phi_\varepsilon\}_{\varepsilon>0}\subset
U,\; \phi_\varepsilon\to\phi\text{ in }U\;
\text{as}\;\varepsilon\to 0^+\right\},\\
%\label{sukduyf}
(\Gamma-\limsup_{\varepsilon\to 0^+} I_\varepsilon)(\phi)
%=I_2(\phi)
:=\inf\left\{\limsup_{\varepsilon\to
0^+}I_\varepsilon(\phi_\varepsilon):\;\,\{\phi_\varepsilon\}_{\varepsilon>0}\subset
U,\; \phi_\varepsilon\to\phi\text{ in }U\;
\text{as}\;\varepsilon\to 0^+\right\},\\
%\label{dfdgvfg}
(\Gamma-\lim_{\varepsilon\to 0^+}
I_\varepsilon)(\phi):=(\Gamma-\liminf_{\varepsilon\to 0^+}
I_\varepsilon\big)(\phi)=(\Gamma-\limsup_{\varepsilon\to 0^+}
I_\varepsilon)(\phi)\;\;\,\text{in the case they are equal}.
\end{align*}
\end{definition}
It is useful to know the $\Gamma$-limit of $I_\varepsilon$, because
it describes the asymptotic behavior as $\varepsilon\downarrow 0$ of
minimizers of $I_\varepsilon$, as it is clear from the following
simple statement:
\begin{proposition}[De-Giorgi]\label{propdj}
Assume that $\phi_\varepsilon$ is a minimizer of $I_\varepsilon$ for
every $\varepsilon>0$. Then:
\begin{itemize}
\item
If $I_0(\phi)=(\Gamma-\liminf_{\varepsilon\to 0^+}
I_\varepsilon)(\phi)$ and $\phi_\varepsilon\to\phi_0$ as
$\varepsilon\to 0^+$ then $\phi_0$ is a minimizer of $I_0$.

\item
If $I_0(\phi)=(\Gamma-\lim_{\varepsilon\to 0^+}
I_\varepsilon)(\phi)$ (i.e. it is a full $\Gamma$-limit of
$I_\varepsilon(\phi)$) and for some subsequence $\varepsilon_n\to
0^+$ as $n\to\infty$, we have $\phi_{\varepsilon_n}\to\phi_0$, then
$\phi_0$ is a minimizer of $I_0$.
\end{itemize}
\end{proposition}
Usually, for finding the $\Gamma$-limit of $I_\varepsilon(\phi)$, we
need to find two bounds.
\begin{itemize}
\item[{\bf(*)}] Firstly, we find a lower bound, i.e. a functional
$\underline{I}(\phi)$ such that for every family
$\{\phi_\varepsilon\}_{\varepsilon>0}$, satisfying
$\phi_\varepsilon\to \phi$ as $\varepsilon\to 0^+$, we have
$\liminf_{\varepsilon\to 0^+}I_\varepsilon(\phi_\varepsilon)\geq
\underline{I}(\phi)$.
\item[{\bf(**)}] Secondly, we find an upper
bound, i.e. a functional $\overline{I}(\phi)$, such that for every
$\phi\in U$ there exists a family
$\{\psi_\varepsilon\}_{\varepsilon>0}$, satisfying
$\psi_\varepsilon\to \phi$ as $\varepsilon\to 0^+$ and
$\limsup_{\varepsilon\to 0^+}I_\varepsilon(\psi_\varepsilon)\leq
\overline{I}(\phi)$.
\item[{\bf(***)}] If we find that
$\underline{I}(\phi)=\overline{I}(\phi):=I(\phi)$, then $I(\phi)$ is
the $\Gamma$-limit of $I_\varepsilon(\phi)$.
\end{itemize}

 In various applications we deal with the asymptotic behavior as $\e\to 0^+$ of a family of
 functionals $\{I_\e\}_{\e>0}$
of the following forms.
\begin{itemize}
\item
%We have
In the case of the first order problem the functional $I_\e$, which
acts on functions $\psi:\O\to\R^m$, has the form
\begin{equation}\label{b1..}
I_\e(\psi)=\int_\O
\e\big|\nabla\psi(x)\big|^2+\frac{1}{\e}W\Big(\psi(x),x\Big)dx\,,
\end{equation}
or more generally
\begin{equation}\label{b2..}
I_\e(\psi)=\int_\O\frac{1}{\e}G\Big(\e^n\nabla\psi^n,\ldots,\e\nabla\psi,\psi,x\Big)dx
+\int_\O\frac{1}{\e}W\big(\psi,x\big)dx\,,
\end{equation}
where $G(0,\ldots,0,\psi,x)\equiv 0$.
\item In the case of the second order problem the functional $I_\e$,
which acts on functions $v:\O\to\R^k$, has the form
\begin{equation}\label{b3..}
I_\e(v)=\int_\O \e\big|\nabla^2 v(x)\big|^2+\frac{1}{\e}W\Big(\nabla
v(x),v(x),x\Big)dx\,,
\end{equation}
or more generally
\begin{equation}\label{b4..}
I_\e(v)=\int_\O\frac{1}{\e}G\Big(\e^n\nabla^{n+1}
v,\ldots,\e\nabla^2 v,\nabla
v,v,x\Big)dx+\int_\O\frac{1}{\e}W\big(\nabla v,v,x\big)dx\,,
\end{equation}
where $G(0,\ldots,0,\nabla v,v,x)\equiv 0$.
\end{itemize}

 The functionals of the form \er{b1..} arise in the theories of phase
transitions and minimal surfaces. They were first studied by Modica
and Mortola \cite{mm1}, Modica \cite{modica}, Sterenberg
\cite{sternberg} and others. The $\Gamma$-limit of the functional in
\er{b1..}, where $W$ don't depend on $x$ explicitly, was obtained in
the general vectorial case by Ambrosio in \cite{ambrosio}. The
$\Gamma$-limit of the functional of the form \er{b2..}, where $n=1$
and there exist $\alpha,\beta\in\R^m$ such that $W(h,x)=0$ if and
only if $h\in\{\alpha,\beta\}$, under some restriction on the
explicit dependence on $x$ of $G$ and $W$, was obtained by Fonseca
and Popovici in \cite{FonP}. The $\Gamma$-limit of the functional of
the form \er{b2..}, with $n=2$,
$G(\cdot)/\e\equiv\e^3|\nabla^2\psi|^2$ and $W$ which doesn't depend
on $x$ explicitly, was found by I.~Fonseca and C.~Mantegazza in
\cite{FM}.

%
%
%
\begin{comment}
Note here that the particular cases of Theorems
\ref{dehgfrygfrgygenbgggggggggggggkgkgthtjtfnewbhjhjkgj2} and
\ref{dehgfrygfrgygenbgggggggggggggkgkgthtjtfnew}, where $n=1$ and
there exist $\alpha,\beta\in\R^m$ such that $W(h)=0$ if and only if
$h\in\{\alpha,\beta\}$, was obtained by Fonseca and Popovici in
\cite{FonP}.
\end{comment}
%
%
%

 The functionals of second order of the form \er{b3..} arise, for
example, in the gradient theory of solid-solid phase transitions,
where one considers energies of the form
\begin{equation}\label{b3..part}
I_\e(v)=\int_\O \e|\nabla^2 v(x)|^2+\frac{1}{\e}W\Big(\nabla
v(x)\Big)dx\,,
\end{equation}
where $v:\O\subset\R^N\to\R^N$ stands for the deformation, and the
free energy density $W(F)$ is nonnegative and satisfies
$$W(F)=0\quad\text{if and only if}\quad F\in K:=SO(N)A\cup SO(N)B\,.$$
Here $A$ and $B$ are two fixed, invertible matrices, such that
$rank(A-B)=1$ and $SO(N)$ is the set of rotations in $\R^N$. The
simpler case where $W(F)=0$ if and only if $F\in\{A,B\}$ was studied
by Conti, Fonseca and Leoni in \cite{contiFL}. The case of problem
\er{b3..part}, where $N=2$ and $W(QF)=W(F)$ for all $Q\in SO(2)$ was
investigated by Conti and Schweizer in \cite{contiS1} (see also
\cite{contiS} for a related problem). Another important example of
the second order energy is the so called Aviles-Giga functional,
defined on scalar valued functions $v$ by
\begin{equation}\label{b5..}
\int_\O\e|\nabla^2 v|^2+\frac{1}{\e}\big(1-|\nabla
v|^2\big)^2\quad\quad\text{(see \cite{adm},\cite{ag1},\cite{ag2})}.
\end{equation}

In this paper we deal with the asymptotic behavior as $\e\to 0^+$ of
a family of functionals of the following general form: Let
$\Omega\subset{\mathbb{R}}^N$ be an open set.
%%Consider the set
%%\begin{equation*}
%\label{fyyyugfjhgjffg}
%%\mathcal{A}=\mathcal{A}(\Omega,k,d,m):=\Big\{(\nabla
%%u,m,\psi):\;\,u:\Omega\to{\mathbb{R}}^{k},\;\, m:\Omega\to{\mathbb{R}}^{d\times N}\text{
%%s.t. }\text{div } m\equiv 0,\;\,\psi:\Omega\to{\mathbb{R}}^m\Big\}
%%\end{equation*}
For every $\varepsilon>0$ consider the general functional
%$I_\varepsilon:L^q\big(\Omega,{\mathbb{R}}^{k\times N}\times{\mathbb{R}}^{d\times
%N}\times{\mathbb{R}}^m\big)\to[0,+\infty]$, defined by
\begin{multline}\label{fhjvjhvjhv}
I_{\varepsilon}(v)=\big\{I_{\varepsilon}(\Omega)\big\}(v):=
\int_{\Omega}\frac{1}{\varepsilon}G\Big(\varepsilon^n\nabla^n
v,\ldots,\varepsilon\nabla
v,v,x\Big)+\frac{1}{\varepsilon}W\big(v,x\big)dx\quad \text{with }
%+\frac{1}{\varepsilon}W(\phi,x),
v:=(\nabla u,h,\psi),\\
%\in W^{1,n}_{loc}\big(\Omega,{\mathbb{R}}^{k\times N}\times{\mathbb{R}}^{d\times N}\times{\mathbb{R}}^m\big),
\text{ where } u\in W^{(n+1),1}_{loc}(\Omega,{\mathbb{R}}^k),\;\,
h\in W^{n,1}_{loc}(\Omega,{\mathbb{R}}^{d\times N})\text{ s.t.
}\text{div } h\equiv 0,\;\,\psi\in
W^{n,1}_{loc}(\Omega,{\mathbb{R}}^m).
\end{multline}
Here
$$G:\R^{\big(\{k\times N\}+\{d\times N\}+m\big)\times N^n}\times\ldots\times\R^{\big(\{k\times N\}+\{d\times N\}+m\big)\times N}
\times\R^{\{k\times N\}+\{d\times N\}+m}\times\R^N\,\to\,\R$$ and
$W:\R^{\{k\times N\}+\{d\times N\}+m}\times\R^N\,\to\,\R$ are
nonnegative continuous functions and $G$ satisfies
$G(0,\ldots,0,v,x)\equiv 0$. The functionals in \er{b1..},\er{b2..}
and \er{b3..},\er{b4..} are important particular cases of the
general energy $I_\e$ in \er{fhjvjhvjhv}. In the general form
\er{fhjvjhvjhv} we also include the dependence on $\Div$-free
function $h$, which can be useful in the study of problems with
non-local terms as the Rivi\`ere-Serfaty functional and other
functionals in Micromagnetics.

 In order to simplify the notations for every open
$\mathcal{U}\subset\R^N$ consider
\begin{multline}\label{fhjvjhvjhvholhiohiovhhjhvvjvf}
\mathcal{B}(\mathcal{U}):=\bigg\{v\in
L^1_{loc}\big(\mathcal{U},{\mathbb{R}}^{k\times
N}\times{\mathbb{R}}^{d\times
N}\times{\mathbb{R}}^m\big):\;\;v=(\nabla u,h,\psi),\\u\in
W^{1,1}_{loc}(\mathcal{U},{\mathbb{R}}^k),\;\, h\in
L^{1}_{loc}(\mathcal{U},{\mathbb{R}}^{d\times N})\text{ s.t.
}\text{div } h\equiv 0,\;\,\psi\in
L^{1}_{loc}(\mathcal{U},{\mathbb{R}}^m)\bigg\},
\end{multline}
and
\begin{equation}\label{huighuihuiohhhiuohoh}
F\Big(\nabla^n v,\ldots,\nabla v,v,x\Big)\,:=\,G\Big(\nabla^n
v,\ldots,\nabla v,v,x\Big)\,+\,W(v,x)
\end{equation}
Then
\begin{equation}\label{fhjvjhvjhvnlhiohoioiiy}
I_{\varepsilon}(v)
%=\big\{I_{\varepsilon}(\Omega)\big\}(v)
=\int_{\Omega}\frac{1}{\varepsilon}F\Big(\varepsilon^n\nabla^n
v,\ldots,\varepsilon\nabla v,v,x\Big)dx\quad \text{with }
%+\frac{1}{\varepsilon}W(\phi,x),
v\in \mathcal{B}(\O)\cap W^{n,1}_{loc}\big(\O,\R^{k\times
N}\times{\mathbb{R}}^{d\times N}\times{\mathbb{R}}^m\big).
\end{equation}
%
%
%
%
\begin{comment}
\begin{multline}\label{fhjvjhvjhvholhiohiovhhjhvvjvf}
\mathcal{B}(\Omega):=\bigg\{v\in
W^{1,n}_{loc}\big(\Omega,{\mathbb{R}}^{k\times
N}\times{\mathbb{R}}^{d\times
N}\times{\mathbb{R}}^m\big):\;v:=(\nabla u,m,\psi),\\u\in
W^{(n+1),1}_{loc}(\Omega,{\mathbb{R}}^k),\;\, m\in
W^{n,1}_{loc}(\Omega,{\mathbb{R}}^{d\times N})\text{ s.t. }\text{div
} m\equiv 0,\;\,\psi\in W^{n,1}_{loc}(\Omega,{\mathbb{R}}^m)\bigg\},
\end{multline}
\end{comment}
%
%
%
%
What can we expect as the $\Gamma$-limit or at least as an upper
bound of these general energies in $L^p$-topology for some $p\geq
1\,$? It is clear that if $G$ and $W$ are nonnegative and $W$ is a
continuous on the argument $v$ function, then the upper bound for
$I_\e(\cdot)$ will be finite only if
\begin{equation}\label{hghiohoijojjkhhhjhjjkjgg}
W\big(v(x),x\big)=0\quad\text{for a.e.}\;\;x\in\Omega\,,
\end{equation}
i.e. if we define
\begin{equation}\label{cuyfyugugghvjjhh}
\mathcal{A}_0:=\bigg\{v\in L^p\big(\Omega,{\mathbb{R}}^{k\times
N}\times{\mathbb{R}}^{d\times
N}\times{\mathbb{R}}^m\big)\cap\mathcal{B}(\Omega):\;\,
W\big(v(x),x\big)=0\;\,\text{for a.e.}\,\;x\in\Omega\bigg\}
\end{equation}
and
\begin{equation}\label{cuyfyugugghvjjhhggihug}
\mathcal{A}:=\bigg\{v\in L^p\big(\Omega,{\mathbb{R}}^{k\times
N}\times{\mathbb{R}}^{d\times N}\times{\mathbb{R}}^m\big):\;\,
(\Gamma-\limsup_{\varepsilon\to 0^+}
I_\varepsilon)(v)<+\infty\bigg\},
\end{equation}
then clearly $\mathcal{A}\subset\mathcal{A}_0$. In most interesting
applications the set $\mathcal{A}_0$ consists of discontinuous
functions. The natural space of discontinuous functions is $BV$
space. It turns out that in the general case if $G$ and $W$ are
$C^1$-functions and if we consider
\begin{equation}\label{hkghkgh}
\mathcal{A}_{BV}:=\mathcal{A}_0\cap\mathcal{B}(\mathbb{R}^N)\cap
BV\cap L^\infty,
\end{equation}
then
\begin{equation}\label{nnloilhyoih}
\mathcal{A}_{BV}\subset\mathcal{A}\subset\mathcal{A}_0.
\end{equation}
In many cases we have $\mathcal{A}_{BV}=\mathcal{A}$. For example
this is indeed the case if the energy $I_\varepsilon(v)$ has the
simplest form $I_\varepsilon(v)=\int_\Omega\varepsilon|\nabla
v|^2+\frac{1}{\varepsilon}W(v)\,dx$, and the set of zeros of $W$:
$\{h: W(h)=0\}$ is finite. However, this is in general not the case.
For example, as was shown by Ambrosio, De Lellis and Mantegazza in
\cite{adm}, $\mathcal{A}_{BV}\subsetneq\mathcal{A}$ in the
particular case of the energy defined by \eqref{b5..} with $N=2$. On
the other hand, there are many applications where the set
$\mathcal{A}$ still inherits some good properties of $BV$ space. For
example, it is indeed the case for the energy \eqref{b5..} with
$N=2$, as was shown by Camillo de Lellis and Felix Otto in
\cite{CDFO}.

 The main contribution of \cite{PI} was to improve our method (see
\cite{pol},\cite{polgen}) for finding upper bounds in the sense of
({\bf**}) for the general functional \er{fhjvjhvjhv} in the case
where the limiting function belongs to $BV$-space, i.e. for
$v=(\nabla u,h,\psi)\in\mathcal{A}_{BV}$. In order to formulate the
main results of \cite{PI} and of this paper we present the following
definitions.
\begin{definition}
For every $\vec\nu\in S^{N-1}$ define
$Q(\vec\nu):=\big\{y\in{\mathbb{R}}^N:\;
-1/2<y\cdot\vec\nu_j<1/2\quad\forall j\big\}$, where
$\{\vec\nu_1,\ldots,\vec\nu_N\}$ is an orthonormal base in
${\mathbb{R}}^N$ such that $\vec\nu_1=\vec\nu$. Then set
\begin{multline*}
%\label{L2009Ddef2slper}
\mathcal{D}_1(v^+,v^-,\vec\nu):=\bigg\{v\in
C^n\big(\mathbb{R}^N,{\mathbb{R}}^{k\times
N}\times{\mathbb{R}}^{d\times N}\times{\mathbb{R}}^m\big)\cap
\mathcal{B}(\mathbb{R}^N):\\ v(y)\equiv \theta(\vec\nu\cdot
y)\;\,\text{and}\;\, v(y)=v^-\;\text{ if }\;y\cdot\vec\nu\leq
-1/2,\;\; v(y)=v^+\;\text{ if }\; y\cdot\vec\nu\geq 1/2\bigg\},
\end{multline*}
where $\mathcal{B}(\cdot)$ is defined in
\eqref{fhjvjhvjhvholhiohiovhhjhvvjvf}, and
\begin{multline*}
%\label{L2009Ddef2slper}
\mathcal{D}_{per}(v^+,v^-,\vec\nu):=\bigg\{v\in
C^n\big(\mathbb{R}^N,{\mathbb{R}}^{k\times
N}\times{\mathbb{R}}^{d\times N}\times{\mathbb{R}}^m\big)\cap \mathcal{B}(\mathbb{R}^N):\\
v(y)=v^-\;\text{ if }\;y\cdot\vec\nu\leq -1/2,\;\; v(y)=v^+\;\text{
if }\; y\cdot\vec\nu\geq 1/2,\;\, v(y+\vec\nu_j)=v(y)\;\;\forall
j=2,\ldots, N\bigg\}.
\end{multline*}
Next define
\begin{align}
\label{Energia1} E_{1}(v^+,v^-,\vec\nu,x)=\inf\bigg\{
\int\limits_{Q(\vec\nu_v)}\frac{1}{L}F\Big(L^n\nabla^n\zeta,\ldots,L\nabla
\zeta,\zeta,x\Big)dy:\;\, L>0,\, \zeta(y)\in
\mathcal{D}_1(v^+,v^-,\vec\nu)\bigg\}\,,\\
\label{Energia2} E_{per}(v^+,v^-,\vec\nu,x)=\inf\bigg\{
\int\limits_{Q(\vec\nu_v)}\frac{1}{L}F\Big(L^n\nabla^n\zeta,\ldots,L\nabla
\zeta,\zeta,x\Big)dy:\;\, L>0,\, \zeta(y)\in
\mathcal{D}_{per}(v^+,v^-,\vec\nu)\bigg\}\,.\\
\label{Energia3} E_{abst}(v^+,v^-,x)=\Big(\Gamma-\liminf_{
\varepsilon \to 0^+}
I_\varepsilon\big(Q(\vec\nu)\big)\Big)\Big(\eta(v^+,v^-,\vec\nu)\Big),
\end{align}
where
\begin{equation}\eta(v^+,v^-,\vec\nu)(y):=
\begin{cases}
v^-\quad\text{if }\vec\nu\cdot y<0,\\
v^+\quad\text{if }\vec\nu\cdot y>0,
\end{cases}
\end{equation}
and we mean the $\Gamma-\liminf$ in $L^p$ topology for some $p\geq
1$.
\end{definition}
It is not difficult to deduce that
\begin{equation}\label{dghfihtihotj}
E_{abst}(v^+,v^-,\vec\nu,x)\leq E_{per}(v^+,v^-,\vec\nu,x)\leq
E_{1}(v^+,v^-,\vec\nu,x).
\end{equation}
Next define the functionals
$K_1(\cdot),K_{per}(\cdot),K^{*}(\cdot):\mathcal{B}(\O)\cap BV\cap
L^\infty\,\to\,\R$ by
\begin{equation}\label{hfighfighfih}
K_{1}(v):=
\begin{cases}
\int_{\O\cap
J_v}E_{1}\Big(v^+(x),v^-(x),\vec\nu_v(x),x\Big)\,d\mathcal{H}^{N-1}(x)\quad\text{if
}v\in \mathcal{A}_0,\\+\infty\quad\text{otherwise},
\end{cases}
\end{equation}
\begin{equation}\label{hfighfighfihgigiugi}
K_{per}(v):=
\begin{cases}
\int_{\O\cap
J_v}E_{per}\Big(v^+(x),v^-(x),\vec\nu_v(x),x\Big)\,d\mathcal{H}^{N-1}(x)\quad\text{if
}v\in \mathcal{A}_0,\\+\infty\quad\text{otherwise},
\end{cases}
\end{equation}
\begin{equation}\label{hfighfighfihhioh}
K^{*}(v):=
\begin{cases}
\int_{\Omega\cap
J_v}E_{abst}\Big(v^+(x),v^-(x),\vec\nu_v(x),x\Big)\,d\mathcal{H}^{N-1}(x)\quad\text{if
}v\in \mathcal{A}_0,\\+\infty\quad\text{otherwise},
\end{cases}
\end{equation}
where $J_v$ is the jump set of $v$, $\vec\nu_v$ is the jump vector
and $v^-,v^+$ are jumps of $v$.
Then, by \er{dghfihtihotj} trivially follows
\begin{equation}\label{nvhfighfrhyrtehu}
K^*\big(v\big)\leq K_{per}\big(v\big)\leq K_1\big(v\big)\,.
\end{equation}
We call $K_1(\cdot)$ by the bound, achieved by one dimensional
profiles, $K_{per}(\cdot)$ by the bound, achieved by
multidimensional periodic profiles and $K^*(\cdot)$ by the bound,
achieved by abstract profiles.

 Our general conjecture is that $K^*(\cdot)$ coincides with the $\Gamma$-limit for the
functionals $I_\e(\cdot)$ in \er{fhjvjhvjhvnlhiohoioiiy}, under
$L^{p}$ convergence, in the case where the limiting functions $v\in
BV\cap L^\infty$.
% and satisfy \er{hghiohoijojjkhhhjhjjkjgg}.
It is known that in the case of the problem \er{b1..}, where $W\in
C^1$ don't depend on $x$ explicitly, this is indeed the case and
moreover, in this case we have equalities in \er{nvhfighfrhyrtehu}
(see \cite{ambrosio}). The same result is also known for problem
\er{b5..} when $N=2$ (see \cite{adm} and \cite{CdL},\cite{pol}). It
is also the case for problem \er{b3..part} where $W(F)=0$ if and
only if $F\in\{A,B\}$, studied by Conti, Fonseca and Leoni, if $W$
satisfies the additional hypothesis ($H_3$) in \cite{contiFL}.
However, as was shown there by an example, if we don't assume
($H_3$)-hypothesis, then it is possible that $E_{per}\big(\nabla
v^+,\nabla v^-,\vec\nu\big)$ is strictly smaller than
$E_{1}\big(\nabla v^+,\nabla v^-,\vec\nu\big)$ and thus, in general,
$K_1(\cdot)$ can differ from the $\Gamma$-limit. In the same work it
was shown that if, instead of ($H_3$) we assume hypothesis ($H_5$),
then $K_{per}(\cdot)$ turns to be equal to $K^*(\cdot)$ and the
$\Gamma$-limit of \er{b3..part} equals to $K_{per}(\cdot)\equiv
K^*(\cdot)$. The similar result known also for problem \er{b2..},
where $n=1$ and there exist $\alpha,\beta\in\R^m$ such that
$W(h,x)=0$ if and only if $h\in\{\alpha,\beta\}$, under some
restriction on the explicit dependence on $x$ of $G$ and $W$. As was
obtained by Fonseca and Popovici in \cite{FonP} in this case we also
obtain that $K_{per}(\cdot)\equiv K^*(\cdot)$ is the $\Gamma$-limit
of \er{b2..}. In the case of problem \er{b3..part}, where $N=2$ and
$W(QF)=W(F)$ for all $Q\in SO(2)$, Conti and Schweizer in
\cite{contiS1} found that the $\Gamma$-limit equals to $K^*(\cdot)$
(see also \cite{contiS} for a related problem). However, by our
knowledge, it is not known, weather in general $K^*(\cdot)\equiv
K_{per}(\cdot)$.

 On \cite{polgen} we showed that for the general problems \er{b2..}
and \er{b4..}, $K_1(\cdot)$ is the upper bound in the sense of
{\bf(**)}, if the limiting function belongs to $BV$-class. However,
as we saw, this bound is not sharp in general. In \cite{PI} we
improved our method and obtained that for the general problem
\er{fhjvjhvjhvnlhiohoioiiy}, $K_{per}(\cdot)$ is always an upper
bound in the sense of {\bf(**)} in the case where the limiting
functions $v$ belong to $BV$-space and $G,W\in C^1$. More precisely,
we have the following Theorem:
\begin{theorem}\label{ffgvfgfhthjghgjhg}
Let $\O\subset\R^N$ be an open set and
$$F:\R^{\big(\{k\times N\}+\{d\times N\}+m\big)\times N^n}\times\ldots\times\R^{\big(\{k\times N\}+\{d\times N\}+m\big)\times N}
\times\R^{\{k\times N\}+\{d\times N\}+m}\times\R^N\,\to\,\R$$ be a
nonnegative $C^1$ function. Furthermore assume that $v:=(\nabla
u,h,\psi)\in\mathcal{B}(\R^N)\cap BV\big(\R^N,\R^{k\times
N}\times\R^{d\times N}\times\R^m\big)\cap
L^\infty\big(\R^N,\R^{k\times N}\times\R^{d\times N}\times\R^m\big)$
satisfies $\Div h\equiv 0$, $|Dv|(\partial\Omega)=0$ and
$$F\Big(0,\ldots,0,v(x),x\Big)=0\quad\text{for a.e.}\;\;x\in\O.$$
Then there exists a sequence $v_\e=\big(\nabla
u_\e,h_\e,\psi_\e\big)\in \mathcal{B}(\R^N)\cap
C^\infty\big(\R^N,\R^{k\times N}\times\R^{d\times N}\times\R^m\big)$
such that $\Div h_\e\equiv 0$, for every $p\geq 1$ we have
$v_\varepsilon\to v$ in $L^p$ and
$$\lim_{\varepsilon\to 0^+}\int_{\Omega}\frac{1}{\varepsilon}F\Big(\varepsilon^n\nabla^n
v_\e(x),\ldots,\varepsilon\nabla v_\e(x)\,,\,v(x)\,,\,x\Big)dx=
K_{per}(v).$$ Here $\mathcal{B}(\R^N)$ was defined by
\er{fhjvjhvjhvholhiohiovhhjhvvjvf} and $K_{per}(\cdot)$ was defined
by \er{hfighfighfihgigiugi}.
\end{theorem}

 The main result of this paper provides with that, for the general problem \er{fhjvjhvjhvnlhiohoioiiy},
when $G,W$ don't depend on $x$ explicitly, $K^*(\cdot)$ is a lower
bound in the sense of {\bf(*)}. More precisely, we have the
following Theorem:
\begin{theorem}\label{dehgfrygfrgygenjklhhjkghhjggjfjkh}
Let $\O\subset\R^N$ be an open set and
$$F:\R^{\big(\{k\times N\}+\{d\times N\}+m\big)\times N^n}\times\ldots\times\R^{\big(\{k\times N\}+\{d\times N\}+m\big)\times N}
\times\R^{\{k\times N\}+\{d\times N\}+m}\,\to\,\R$$ be a nonnegative
continuous function. Furthermore assume that $v:=(\nabla
u,h,\psi)\in\mathcal{B}(\O)\cap BV\big(\O,\R^{k\times
N}\times\R^{d\times N}\times\R^m\big)\cap
L^\infty\big(\O,\R^{k\times N}\times\R^{d\times N}\times\R^m\big)$
satisfies
$$F\Big(0,\ldots,0,v(x)\Big)=0\quad\text{for a.e.}\;\;x\in\O.$$
Then for every $\{v_\varepsilon\}_{\varepsilon>0}\subset
\mathcal{B}(\Omega)\cap W^{n,1}_{loc}\big(\O,\R^{k\times
N}\times\R^{d\times N}\times\R^m\big)$, such that $v_\varepsilon\to
v$ in $L^p$ as $\e\to 0^+$, we have $$\liminf_{\varepsilon\to
0^+}\int_{\Omega}\frac{1}{\varepsilon}F\Big(\varepsilon^n\nabla^n
v_\e(x),\ldots,\varepsilon\nabla v_\e(x)\,,\,v(x)\Big)dx\geq
K^{*}(v).$$ Here $K^{*}(\cdot)$ is defined by \er{hfighfighfihhioh}
with respect to $L^p$ topology.
\end{theorem}
For slightly generalized formulation and additional details see
Theorem \ref{dehgfrygfrgygenjklhhj}. See also Theorem
\ref{dehgfrygfrgygen} as an analogous result for more general
functionals than that defined by \er{fhjvjhvjhvnlhiohoioiiy}.

As we saw there is a natural question: weather in general
$K^*(\cdot)\equiv K_{per}(\cdot)\,$? The answer yes will mean that,
in the case when $G,W$ are $C^1$ functions which don't depend on $x$
explicitly, the upper bound in Theorem \ref{ffgvfgfhthjghgjhg} will
coincide with the lower bound of Theorem
\ref{dehgfrygfrgygenjklhhjkghhjggjfjkh} and therefore we will find
the full $\Gamma$-limit in the case of $BV$ limiting functions. The
equivalent question is weather
\begin{equation*}
E_{abst}(v^+,v^-,\vec\nu,x)= E_{per}(v^+,v^-,\vec\nu,x),
\end{equation*}
%\begin{equation}\label{hdiohdo}
%E^*\Big(\big\{\nabla v^+,h^+,\psi^+\big\}, \big\{\nabla
%v^-,h^-,\psi^-\big\},\vec\nu,x,p,q_1,q_2,q_3\Big)=
%E_{per}\Big(\big\{\nabla v^+,h^+,\psi^+\big\}, \big\{\nabla
%v^-,h^-,\psi^-\big\},\vec\nu,x\Big)\,,
%\end{equation}
where $E_{per}(\cdot)$ is defined in \er{Energia2} and
$E_{abst}(\cdot)$ is defined by \er{Energia3}. In section
\ref{vdhgvdfgbjfdhgf} we formulate and prove some partial results
that refer to this important question. In particular we prove that
this is indeed the case for the general problem \er{b2..} i.e. when
we have no prescribed differential constraint. More precisely, we
have the following Theorem:
\begin{theorem}\label{dehgfrygfrgygenbgggggggggggggkgkgthtjtfnewbhjhjkgj}
Let $G\in C^1\big(\R^{m\times N^n}\times\R^{m\times
N^{(n-1)}}\times\ldots\times\R^{m\times N}\times \R^m,\R\big)$ and
$W\in C^1(\R^m,\R)$ be nonnegative functions such that
$G\big(0,0,\ldots,0,b)= 0$ for every $b\in\R^m$ and there exist
$C>0$ and $p\geq 1$ satisfying
\begin{multline}\label{hgdfvdhvdhfvjjjjiiiuyyyjitghujtrnewkhjklhkl}
\frac{1}{C}|A|^p
%-C\Big(|b|^p+1\Big)
\leq F\Big(A,a_1,\ldots,a_{n-1},b\Big) \leq
C\bigg(|A|^p+\sum_{j=1}^{n-1}|a_j|^{p}+|b|^p+1\bigg)\quad \text{for
every}\;\;\big(A,a_1,a_2,\ldots,a_{n-1},b\big),
\end{multline}
where we denote
$$F\Big(A,a_1,\ldots,a_{n-1},b\Big):=G\Big(A,a_1,\ldots,a_{n-1},b\Big)+W(b)$$
Next let $\psi\in BV(\R^N,\R^{m})\cap L^\infty$ be such that $\|D
\psi\|(\partial\Omega)=0$ and $W\big(\psi(x)\big)=0$ for a.e.
$x\in\O$.
%Let $\O$, $D$, $\vec A$, $F$, $q$, $v$, $\vec n$, and $\{\vec
%A\cdot\nabla v\}^\pm$ as above.
Then $K^*(\psi)=K_{per}(\psi)$ and for every
%sequence $\e_n\to 0^+$ as $n\to 0^+$ and every sequence
$\{\varphi_\e\}_{\e>0}\subset W^{n,p}_{loc}(\O,\R^m)$ such that
$\varphi_\e\to \psi$ in $L^p_{loc}(\O,\R^m)$ as $\e\to 0^+$, we have
\begin{multline}\label{a1a2a3a4a5a6a7s8hhjhjjhjjjjjjkkkkgenhjhhhhjtjurtnewjgvjhgv}
\liminf_{\e\to 0^+}I_\e(\varphi_\e):=\liminf_{\e\to
0^+}\frac{1}{\e}\int_\O F\bigg(\,\e^n\nabla^n
\varphi_\e(x),\,\e^{n-1}\nabla^{n-1}\varphi_\e(x),\,\ldots,\,\e\nabla \varphi_\e(x),\, \varphi_\e(x)\bigg)dx\\
\geq K_{per}(\psi):= \int_{\O\cap J_\psi}\bar
E_{per}\Big(\psi^+(x),\psi^-(x),\vec \nu(x)\Big)d \mathcal
H^{N-1}(x)\,,
\end{multline}
where
\begin{multline}\label{L2009hhffff12kkkhjhjghghgvgvggcjhggghtgjutnewjgkjgjk}
\bar E_{per}\Big(\psi^+,\psi^-,\vec \nu\Big)\;:=\;\\
\inf\Bigg\{\int_{Q_{\vec \nu}}\frac{1}{L} F\bigg(L^n\,\nabla^n
\zeta,\,L^{n-1}\,\nabla^{n-1} \zeta,\,\ldots,\,L\,\nabla
\zeta,\,\zeta\bigg)\,dx:\;\; L\in(0,+\infty)\,,\;\zeta\in
\mathcal{\tilde D}_{per}(\psi^+,\psi^-,\vec \nu)\Bigg\}\,,
\end{multline}
with
\begin{multline}\label{L2009Ddef2hhhjjjj77788hhhkkkkllkjjjjkkkhhhhffggdddkkkgjhikhhhjjddddkkjkjlkjlkintuukgkggk}
\mathcal{\tilde D}_{per}(\psi^+,\psi^-,\vec \nu):=
%\mathcal{D}\Big(\vec A,(\vec A\cdot\nabla v)^-(x),(\vec A\cdot\nabla v)^+(x),\vec\nu(x),\vec k_2(x),\vec k_3(x),\ldots,\vec k_{N}(x)\Big):=\\
\bigg\{\zeta\in C^n(\R^N,\R^m):\;\;\zeta(y)=\psi^-\;\text{ if }\;y\cdot\vec\nu\leq-1/2,\\
\zeta(y)=\psi^+\;\text{ if }\; y\cdot\vec\nu(x)\geq 1/2\;\text{ and
}\;\zeta\big(y+\vec k_j\big)=\zeta(y)\;\;\forall j=2,3,\ldots,
N\bigg\}\,.
\end{multline}
Here $Q_{\vec \nu}:=\{y\in\R^N:\;|y\cdot \vec
\nu_j|<1/2\;\;\;\forall j=1,2\ldots N\}$ where $\{\vec \nu_1,\vec
\nu_2,\ldots,\vec \nu_N\}\subset\R^N$ is an orthonormal base in
$\R^N$ such that $\vec \nu_1:=\vec \nu$. Moreover, there exists e
sequence $\{\psi_\e\}_{\e>0}\subset C^\infty(\R^N,\R^m)$ such that
$\int_\O\psi_\e(x)dx=\int_\O \psi(x)dx$, for every $q\geq 1$ we have
$\psi_\e\to \psi$ in $L^q(\O,\R^m)$ as $\e\to 0^+$, and we have
\begin{multline}\label{a1a2a3a4a5a6a7s8hhjhjjhjjjjjjkkkkgenhjhhhhjtjurtgfhfhfjfjfjnewjkggujk}
\lim_{\e\to 0^+}I_\e(\psi_\e):=\lim_{\e\to 0^+}\frac{1}{\e}\int_\O
F\bigg(\,\e^n\nabla^n
\psi_\e(x),\,\e^{n-1}\nabla^{n-1}\psi_\e(x),\,\ldots,\,\e\nabla \psi_\e(x),\, \psi_\e(x)\bigg)dx\\
=K_{per}(\psi):= \int_{\O\cap J_\psi}\bar
E_{per}\Big(\psi^+(x),\psi^-(x),\vec \nu(x)\Big)d \mathcal
H^{N-1}(x)\,.
\end{multline}
\end{theorem}
See Theorem \ref{dehgfrygfrgygenbgggggggggggggkgkgthtjtfnew} as a
slightly generalized result.
%Note here that the particular cases of
%Theorems \ref{dehgfrygfrgygenbgggggggggggggkgkgthtjtfnewbhjhjkgj2}
%and \ref{dehgfrygfrgygenbgggggggggggggkgkgthtjtfnew}, where $n=1$
%and there exist $\alpha,\beta\in\R^m$ such that $W(h)=0$ if and only
%if $h\in\{\alpha,\beta\}$, was obtained by Fonseca and Popovici in
%\cite{FonP}.
%See Theorem \ref{dehgfrygfrgygenbgggggggggggggkgkgthtjtfnew} for the
%equivalent formulation.
\begin{remark}\label{vyuguigiugbuikkk}
In what follows we use some special notations and apply some basic
theorems about $BV$ functions. For the convenience of the reader we
put these notations and theorems in Appendix.
\end{remark}

\section{The abstract lower bound}
\begin{definition}\label{gdhgvdgjkdfgjkh}
Given an open set $G\subset\R^N$ and a \underline{vector} $\vec
q=(q_1,q_2,\ldots, q_m)\in\R^m$, such that $q_j\geq 1$ for every
$1\leq j\leq m$, define the Banach space $L^{\vec q}(G,\R^m)$ as the
space of all (equivalency classes of a.e. equal) functions
$f(x)=\big(f_1(x),f_2(x),\ldots, f_m(x)\big):G\to\R^m$, such that
$f_j(x)\in L^{q_j}(G,\R)$ for every $1\leq j\leq m$, endowed with
the norm $\|f\|_{L^{\vec
q}(G,\R^m)}:=\sum_{j=1}^{m}\|f_j\|_{L^{q_j}(G,\R)}$. Next define, as
usual, $L^{\vec q}_{loc}(G,\R^m)$ as a space of all functions
$f:G\to\R^m$, such that for every compactly embedded
$U\subset\subset G$ we have $f\in L^{\vec q}_{loc}(U,\R^m)$. Finally
in the case where $q\in[1,+\infty)$ is a \underline{scalar} we as,
usual, consider  $L^{q}(G,\R^m):=L^{\vec q}(G,\R^m)$ and
$L^{q}_{loc}(G,\R^m):=L^{\vec q}_{loc}(G,\R^m)$, where $\vec
q:=(q,q,\ldots,q)$.
\end{definition}
\begin{definition}\label{gdhgvdgjkdfgjkhdd}
Given a vector $x:=(x_1,x_2,\ldots, x_m)\in\R^m$ and a
\underline{vector} $\vec q=(q_1,q_2,\ldots, q_m)\in\R^m$, such that
$q_j\geq 1$, we define $|x|^{\vec q}:=\sum_{j=1}^{m}|x_j|^{q_j}$.
Note that for a \underline{scalar} $q$, $|x|^q$ and
$|x|^{(q,q,\ldots q)}$ are, in general, different quantities,
although they have the same order, i.e. $|x|^q/C\leq|x|^{(q,q,\ldots
q)}\leq C|x|^q$ for some constant $C>0$.
\end{definition}
\begin{theorem}\label{dehgfrygfrgy}
Let $\mathcal{M}$ be a subset of $\R^m$, $\O\subset\R^N$ be an open
set and $D\subset\O$ be a $\mathcal{H}^{N-1}$ $\sigma$-finite Borel
set.  Consider $F\in C(\R^{m\times N}\times \R^m\times\R^N,\R)$,
which satisfies $F\geq 0$ and the following property: For every
$x_0\in\O$ and every $\tau>0$ there exists $\alpha>0$ satisfying
%for every $a\in \R^{m\times N}$, $b\in\R^m$ and $x\in\R^N$
%satisfying $|x-x_0|<\delta$ we have
\begin{equation}\label{vcjhfjhgjkg}
F(a,b,x)-F(a,b,x_0)\geq -\tau F(a,b,x_0)\quad\forall\, a\in
\R^{m\times N}\;\forall\, b\in\R^m\;\forall\, x\in\R^N\;\;\text{such
that}\;\;|x-x_0|<\alpha\,.
\end{equation}
Furthermore, let $\vec A\in \mathcal{L}(\R^{d\times N};\R^m)$,
$q=(q_1,q_2,\ldots, q_m)\in\R^m$, $p\geq 1$ and
$v\in\mathcal{D}'(\O,\R^d)$ be such that $q_j\geq 1$, $\vec
A\cdot\nabla v\in L^q_{loc}(\O,\R^m)$ and $F\big(0,\{\vec
A\cdot\nabla v\}(x),x\big)=0$ a.e.~in $\O$. Assume also that there
exist three Borel mappings $\{\vec A\cdot\nabla v\}^+(x):D\to\R^m$,
$\{\vec A\cdot\nabla v\}^-(x):D\to\R^m$ and $\vec n(x):D\to S^{N-1}$
such that for every $x\in D$ we have
\begin{multline}\label{L2009surfhh8128odno888jjjjjkkkkkk}
\lim\limits_{\rho\to 0^+}\frac{\int_{B_\rho^+(x,\vec
n(x))}\big|\{\vec A\cdot\nabla v\}(y)-\{\vec A\cdot\nabla
v\}^+(x)\big|^q\,dy} {\mathcal{L}^N\big(B_\rho(x)\big)}=0,\\
\lim\limits_{\rho\to 0^+}\frac{\int_{B_\rho^-(x,\vec
n(x))}\big|\{\vec A\cdot\nabla v\}(y)-\{\vec A\cdot\nabla
v\}^-(x)\big|^q\,dy}
{\mathcal{L}^N\big(B_\rho(x)\big)}=0\quad\quad\quad\quad\text{(see
Definition \ref{gdhgvdgjkdfgjkhdd})}.
\end{multline}

%Let $\O$, $D$, $\vec A$, $F$, $q$, $v$, $\vec n$, and $\{\vec
%A\cdot\nabla v\}^\pm$ as above.
Then for every
%sequence $\e_n\to 0^+$ as $n\to 0^+$ and every sequence
$\{v_\e\}_{\e>0}\subset\mathcal{D}'(\O,\R^d)$, satisfying $\vec
A\cdot\nabla v_\e\in L^q_{loc}(\O,\R^m)\cap W^{1,p}_{loc}(\O,\R^m)$,
$\{\vec A\cdot\nabla v_\e\}(x)\in \mathcal{M}$ for a.e. $x\in\R^N$
and $\vec A\cdot\nabla v_\e\to \vec A\cdot\nabla v$ in
$L^q_{loc}(\O,\R^m)$ as $\e\to 0^+$, we have
\begin{multline}\label{a1a2a3a4a5a6a7s8hhjhjjhjjjjjjkkkk}
\liminf_{\e\to 0^+}\frac{1}{\e}\int_\O F\Big(\,\e\nabla\big\{\vec
A\cdot\nabla v_\e\big\}(x),\, \{\vec A\cdot\nabla v_\e\}(x),\,x\Big)dx\geq\\
\int_{D}E_0\Big(\{\vec A\cdot\nabla v\}^+(x),\{\vec A\cdot\nabla
v\}^-(x),\vec n(x),x\Big)d \mathcal H^{N-1}(x)\,,
\end{multline}
where for every $x\in\R^N$, $a,b\in\R^m$ and any unit vector $\vec
\nu\in\R^N$
\begin{multline}\label{a1a2a3a4a5a6a7s8hhjhjjhjjjjjjkkkkjgjgjgjhlllllkkk}
E_0\big(a,b,\vec \nu,x\big):=\inf\Bigg\{\liminf_{\e\to
0^+}\frac{1}{\e}\int_{I_{\vec \nu}} F\Big(\,\e\nabla\big\{\vec
A\cdot\nabla \varphi_\e\big\}(y),\, \{\vec A\cdot\nabla
\varphi_\e\}(y),\,x\Big)dy:\;\; \varphi_\e\in\mathcal{D}'(I_{\vec
\nu},\R^d)\quad\text{s.t.}
%\R^d
\\ \vec A\cdot\nabla \varphi_\e\in L^q(I_{\vec
\nu},\R^m)\cap W^{1,p}(I_{\vec \nu},\R^m),\;\;\vec A\cdot\nabla
\varphi_\e\in \mathcal{M}\; \text{a.e. in}\;I_{\vec
\nu}\;\;\text{and}\;\;\{\vec A\cdot\nabla \varphi_\e\}(y)\to
\xi(y,a,b,\vec \nu)\;\text{in}\; L^q(I_{\vec \nu},\R^m)\Bigg\}.
\end{multline}
Here $I_{\vec \nu}:=\{y\in\R^N:\;|y\cdot \vec\nu_j|<1/2\;\;\;\forall
j=1,2\ldots N\}$ where
$\{\vec\nu_1,\vec\nu_2,\ldots,\vec\nu_N\}\subset\R^N$ is an
orthonormal base in $\R^N$ such that $\vec\nu_1:=\vec \nu$ and
\begin{equation}\label{fhyffgfgfgfffgf}
\xi(y,a,b,\vec \nu):=\begin{cases}a\quad\text{if}\;y\cdot\vec
\nu>0\,,\\ b\quad\text{if}\;y\cdot\vec \nu<0\,.\end{cases}
\end{equation}
\end{theorem}
\begin{proof}
It is clear that we may assume that
\begin{equation}\label{a1a2a3a4a5a6a7s8hhjhjjhjjjjjjkkkkyiuouilokkk}
T_0:=\liminf_{\e\to 0^+}\frac{1}{\e}\int_\O
F\Big(\,\e\nabla\big\{\vec A\cdot\nabla v_\e\big\}(x),\, \{\vec
A\cdot\nabla v_\e\}(x),\,x\Big)dx<+\infty\,,
\end{equation}
otherwise it is trivial. Then, up to a subsequence
%still denoted by
$\e_n\to 0^+$ as $n\to+\infty$, we have
\begin{equation}\label{a1a2a3a4a5a6a7s8hhjhjjhjjjjjjkkkkyiuouilokkkmnhh}
\frac{1}{\e_n}F\Big(\,\e_n\nabla\big\{\vec A\cdot\nabla
v_n\big\}(x),\, \{\vec A\cdot\nabla
v_n\}(x),\,x\Big)dx\rightharpoonup\mu\,,
\end{equation}
and
\begin{equation}\label{a1a2a3a4a5a6a7s8hhjhjjhjjjjjjkkkkyiuouilokkkmnhhfghfhghgjkkkkk}
T_0\geq\mu(\O)\,,
\end{equation}
where $\mu$ is some positive finite Radon measure on $\O$ and the
convergence is in the sense of the weak$^*$ convergence of finite
Radon measures. Moreover, for every compact set $K\subset\O$ we have
\begin{equation}\label{a1a2a3a4a5a6a7s8hhjhjjhjjjjjjkkkkyiuouilokkkmnhhfghfhghgjkkkkkhghghffhfghghkkjjjkk}
\mu(K)\geq\liminf_{n\to +\infty}\frac{1}{\e_n}\int_{K}
F\Big(\,\e_n\nabla\big\{\vec A\cdot\nabla v_n\big\}(x),\, \{\vec
A\cdot\nabla v_n\}(x),\,x\Big)dx\,.
\end{equation}
Next by the Theorem about $k$-dimensinal densities (Theorem 2.56 in
\cite{amb}) we have
\begin{equation}\label{a1a2a3a4a5a6a7s8hhjhjjhjjjjjjkkkkyiuouilokkkmnhhfghfhghgjkkkkkhghghffhjkljlkkkkk}
\mu(D)\geq\int_D\sigma(x)\,d\mathcal{H}^{N-1}(x)\,,
\end{equation}
where
\begin{equation}\label{a1a2a3a4a5a6a7s8hhjhjjhjjjjjjkkkkyiuouilokkkmnhhfghfhghgjkkkkkhghghffhjkljlkkkkkilljkkkk}
\sigma(x):=\limsup_{\rho\to
0^+}\frac{\mu(B_\rho(x))}{\omega_{N-1}\rho^{N-1}}\,,
\end{equation}
with $\omega_{N-1}$ denoting the $\mathcal{L}^{N-1}$-measure of
$(N-1)$-dimensional unit ball. Fix now $\delta>1$. Then by
\er{a1a2a3a4a5a6a7s8hhjhjjhjjjjjjkkkkyiuouilokkkmnhhfghfhghgjkkkkkhghghffhfghghkkjjjkk},
for every $x\in\O$ and every $\rho>0$ sufficiently small we have
\begin{equation}\label{a1a2a3a4a5a6a7s8hhjhjjhjjjjjjkkkkyiuouilokkkmnhhfghfhghgjkkkkkhghghffhfghghkkjjjkkhjgjghghghkkkkk}
\mu\big(B_{(\delta\rho)}(x)\big)\geq\mu\big(\ov
B_\rho(x)\big)\geq\liminf_{n\to +\infty}\frac{1}{\e_n}\int_{\ov
B_\rho(x)} F\Big(\,\e_n\nabla\big\{\vec A\cdot\nabla v_n\big\}(y),\,
\{\vec A\cdot\nabla v_n\}(y),\,y\Big)dy\,.
\end{equation}
On the other hand by
\er{a1a2a3a4a5a6a7s8hhjhjjhjjjjjjkkkkyiuouilokkkmnhhfghfhghgjkkkkk},
\er{a1a2a3a4a5a6a7s8hhjhjjhjjjjjjkkkkyiuouilokkkmnhhfghfhghgjkkkkkhghghffhjkljlkkkkk}
and
\er{a1a2a3a4a5a6a7s8hhjhjjhjjjjjjkkkkyiuouilokkkmnhhfghfhghgjkkkkkhghghffhjkljlkkkkkilljkkkk}
we obtain
\begin{equation}\label{a1a2a3a4a5a6a7s8hhjhjjhjjjjjjkkkkyiuouilokkkmnhhfghfhghgjkkkkkhghghffhjkljlkkkkkiuok}
T_0\geq\mu(\O)\geq\mu(D)\geq\int_D\sigma(x)\,\mathcal{H}^{N-1}(x)=\int_D\bigg\{\limsup_{\rho\to
0^+}\frac{\mu(B_{(\delta\rho)}(x))}{\omega_{N-1}(\delta\rho)^{N-1}}\bigg\}\,d\mathcal{H}^{N-1}(x)\,.
\end{equation}
Thus plugging
\er{a1a2a3a4a5a6a7s8hhjhjjhjjjjjjkkkkyiuouilokkkmnhhfghfhghgjkkkkkhghghffhfghghkkjjjkkhjgjghghghkkkkk}
into
\er{a1a2a3a4a5a6a7s8hhjhjjhjjjjjjkkkkyiuouilokkkmnhhfghfhghgjkkkkkhghghffhjkljlkkkkkiuok}
we deduce
\begin{multline}\label{a1a2a3a4a5a6a7s8hhjhjjhjjjjjjkkkkyiuouilokkkmnhhfghfhghgjkkkkkhghghffhjkljlkkkkkiuokhjjkkkk}
T_0\geq\mu(\O)\geq \\
\frac{1}{\delta^{N-1}}\int_D\Bigg\{\limsup_{\rho\to
0^+}\Bigg(\frac{1}{\omega_{N-1}\rho^{N-1}}\liminf_{n\to
+\infty}\frac{1}{\e_n}\int_{B_\rho(x)} F\Big(\,\e_n\nabla\big\{\vec
A\cdot\nabla v_n\big\}(y),\, \{\vec A\cdot\nabla
v_n\}(y),\,y\Big)dy\Bigg)\Bigg\}\,d\mathcal{H}^{N-1}(x)\,.
\end{multline}
Therefore, since $\delta>1$ was chosen arbitrary we deduce
\begin{multline}\label{a1a2a3a4a5a6a7s8hhjhjjhjjjjjjkkkkyiuouilokkkmnhhfghfhghgjkkkkkhghghffhjkljlkkkkkiuokhjjkkkkjhjhhjhh}
T_0\geq\mu(\O)\geq \\
\int_D\Bigg\{\limsup_{\rho\to
0^+}\Bigg(\frac{1}{\omega_{N-1}\rho^{N-1}}\liminf_{n\to
+\infty}\frac{1}{\e_n}\int_{B_\rho(x)} F\Big(\,\e_n\nabla\big\{\vec
A\cdot\nabla v_n\big\}(y),\, \{\vec A\cdot\nabla
v_n\}(y),\,y\Big)dy\Bigg)\Bigg\}\,d\mathcal{H}^{N-1}(x)\,.
\end{multline}
Next set
\begin{equation}\label{a1a2a3a4a5a6a7s8hhjhjjhjjjjjjkkkkyiuouilokkkmnhhfghfhghgjkkkkkhghghffhjkljlkkkkkilljkkkipkppppp}
\varphi_{n,\rho,x}(z):=\frac{1}{\rho}v_n(x+\rho
z)\quad\text{and}\quad\varphi_{\rho,x}(z):=\frac{1}{\rho}v(x+\rho
z)\,.
\end{equation}
Then changing variables $y=x+\rho z$ in the interior integration in
\er{a1a2a3a4a5a6a7s8hhjhjjhjjjjjjkkkkyiuouilokkkmnhhfghfhghgjkkkkkhghghffhjkljlkkkkkiuokhjjkkkkjhjhhjhh}
we infer
\begin{multline}\label{a1a2a3a4a5a6a7s8hhjhjjhjjjjjjkkkkyiuouilokkkmnhhfghfhghgjkkkkkhghghffhjkljlkkkkkiuokhjjkkkkjhjhhjhhlllkkkkkkuku}
T_0\geq\mu(\O)\geq \int_D\Bigg\{\\ \limsup_{\rho\to
0^+}\Bigg(\frac{1}{\omega_{N-1}}\liminf_{n\to
+\infty}\frac{1}{(\e_n/\rho)}\int_{B_1(0)}
F\Big(\,(\e_n/\rho)\nabla\big\{\vec A\cdot\nabla
\varphi_{n,\rho,x}\big\}(z),\, \{\vec A\cdot\nabla
\varphi_{n,\rho,x}\}(z),\,x+\rho z\Big)dz\Bigg)
\Bigg\}\,d\mathcal{H}^{N-1}(x)\,.
\end{multline}
However, by \er{vcjhfjhgjkg} for every $x\in D$ and every $\tau>0$
we obtain
\begin{multline}\label{gjdjgdjghdfjghfjhkfg}
\limsup_{\rho\to 0^+}\Bigg(\frac{1}{\omega_{N-1}}\liminf_{n\to
+\infty}\frac{1}{(\e_n/\rho)}\int_{B_1(0)}
F\Big(\,(\e_n/\rho)\nabla\big\{\vec A\cdot\nabla
\varphi_{n,\rho,x}\big\}(z),\, \{\vec A\cdot\nabla
\varphi_{n,\rho,x}\}(z),\,x+\rho z\Big)dz\Bigg)\geq\\
\limsup_{\rho\to 0^+}\Bigg(\frac{1}{\omega_{N-1}}\liminf_{n\to
+\infty}\frac{1}{(\e_n/\rho)}\int_{B_1(0)}
(1-\tau)F\Big(\,(\e_n/\rho)\nabla\big\{\vec A\cdot\nabla
\varphi_{n,\rho,x}\big\}(z),\, \{\vec A\cdot\nabla
\varphi_{n,\rho,x}\}(z),\,x\Big)dz\Bigg)\,.
\end{multline}
Thus since $\tau>0$ is arbitrary
\begin{multline}\label{gjdjgdjghdfjghfjhkfgjjkgjk}
\limsup_{\rho\to 0^+}\Bigg(\frac{1}{\omega_{N-1}}\liminf_{n\to
+\infty}\frac{1}{(\e_n/\rho)}\int_{B_1(0)}
F\Big(\,(\e_n/\rho)\nabla\big\{\vec A\cdot\nabla
\varphi_{n,\rho,x}\big\}(z),\, \{\vec A\cdot\nabla
\varphi_{n,\rho,x}\}(z),\,x+\rho z\Big)dz\Bigg)\geq\\
\limsup_{\rho\to 0^+}\Bigg(\frac{1}{\omega_{N-1}}\liminf_{n\to
+\infty}\frac{1}{(\e_n/\rho)}\int_{B_1(0)}
F\Big(\,(\e_n/\rho)\nabla\big\{\vec A\cdot\nabla
\varphi_{n,\rho,x}\big\}(z),\, \{\vec A\cdot\nabla
\varphi_{n,\rho,x}\}(z),\,x\Big)dz\Bigg)\,.
\end{multline}
Plugging \er{gjdjgdjghdfjghfjhkfgjjkgjk} into
\er{a1a2a3a4a5a6a7s8hhjhjjhjjjjjjkkkkyiuouilokkkmnhhfghfhghgjkkkkkhghghffhjkljlkkkkkiuokhjjkkkkjhjhhjhhlllkkkkkkuku}
we deduce
\begin{multline}\label{a1a2a3a4a5a6a7s8hhjhjjhjjjjjjkkkkyiuouilokkkmnhhfghfhghgjkkkkkhghghffhjkljlkkkkkiuokhjjkkkkjhjhhjhhlllkkkkk}
T_0\geq\mu(\O)\geq \int_D\Bigg\{\\ \limsup_{\rho\to
0^+}\Bigg(\frac{1}{\omega_{N-1}}\liminf_{n\to
+\infty}\frac{1}{(\e_n/\rho)}\int_{B_1(0)}
F\Big(\,(\e_n/\rho)\nabla\big\{\vec A\cdot\nabla
\varphi_{n,\rho,x}\big\}(z),\, \{\vec A\cdot\nabla
\varphi_{n,\rho,x}\}(z),\,x\Big)dz\Bigg)
\Bigg\}\,d\mathcal{H}^{N-1}(x)\,.
\end{multline}
Furthermore, for every $x\in D$ for every small $\rho>0$ we have
\begin{equation}\label{a1a2a3a4a5a6a7s8hhjhjjhjjjjjjkkkkyiuouilokkkmnhhfghfhghgjkkkkkhghghffhjkljlkkkkkilljkkkipkpppppfbhfjkgjgkkkhghghfhfgggg}
\vec A\cdot\nabla \varphi_{n,\rho,x}\to \vec A\cdot\nabla
\varphi_{\rho,x}\quad{as}\;n\to +\infty\quad\text{in}\quad
L^q\big(B_1(0),\R^m\big)\,.
\end{equation}
On the other hand by \er{L2009surfhh8128odno888jjjjjkkkkkk} for
every $x\in D$ we have
\begin{equation}\label{a1a2a3a4a5a6a7s8hhjhjjhjjjjjjkkkkyiuouilokkkmnhhfghfhghgjkkkkkhghghffhjkljlkkkkkilljkkkipkpppppfbhfjkgjgkkkhghghfhfgggglokjjjjhhh}
\{\vec A\cdot\nabla \varphi_{\rho,x}\}(z)\to \xi\Big(z,\{\vec
A\cdot\nabla v\}^+(x),\{\vec A\cdot\nabla v\}^-(x),\vec
n(x)\Big)\quad{as}\;\rho\to 0^+\quad\text{in}\quad
L^q\big(B_1(0),\R^m\big)\,.
\end{equation}
Thus, for every $x\in D$, we can extract appropriate diagonal
subsequences of $\{\varphi_{n,\rho,x}\}_{\{n,\rho\}}$, and
$\{\e_n/\rho\}_{\{n,\rho\}}$ which we denote by
$\{\vartheta_{j}\}_{j=1}^{+\infty}$, and $\{\e'_j\}_{j=1}^{+\infty}$
respectively, so that $\e'_j\to 0^+$ as $j\to +\infty$,
$\vartheta_{j}(z)\to \xi\Big(z,\{\vec A\cdot\nabla v\}^+(x),\{\vec
A\cdot\nabla v\}^-(x),\vec n(x)\Big)$ in $L^q\big(B_1(0),\R^m\big)$
as $j\to +\infty$, and
\begin{multline}\label{a1a2a3a4a5a6a7s8hhjhjjhjjjjjjkkkkyiuouilokkkmnhhfghfhghgjkkkkkhghghffhjkljlkkkkkiuokhjjkkkkjhjhhjhhlllkkkkkffgfgfgffgfg}
\limsup_{\rho\to 0^+}\Bigg(\frac{1}{\omega_{N-1}}\liminf_{n\to
+\infty}\frac{1}{(\e_n/\rho)}\int_{B_1(0)}
F\Big(\,(\e_n/\rho)\nabla\big\{\vec A\cdot\nabla
\varphi_{n,\rho,x}\big\}(z),\, \{\vec A\cdot\nabla
\varphi_{n,\rho,x}\}(z),\,x\Big)dz\Bigg)\\ \geq
\frac{1}{\omega_{N-1}}\liminf_{j\to
+\infty}\frac{1}{\e'_j}\int_{B_1(0)} F\Big(\,\e'_j\nabla\big\{\vec
A\cdot\nabla \vartheta_{j}\big\}(z),\, \{\vec A\cdot\nabla
\vartheta_{j}\}(z),\,x\Big)dz\,.
\end{multline}
Plugging this fact into
\er{a1a2a3a4a5a6a7s8hhjhjjhjjjjjjkkkkyiuouilokkkmnhhfghfhghgjkkkkkhghghffhjkljlkkkkkiuokhjjkkkkjhjhhjhhlllkkkkk}
we obtain
\begin{multline}\label{a1a2a3a4a5a6a7s8hhjhjjhjjjjjjkkkkjjjkjkkkkkk}
\liminf_{\e\to 0^+}\frac{1}{\e}\int_\O F\Big(\,\e\nabla\big\{\vec
A\cdot\nabla v_\e(x)\big\},\, \{\vec A\cdot\nabla v_\e\}(x),\,x\Big)dx=T_0\geq\\
\int_{D}E_1\Big(\{\vec A\cdot\nabla v\}^+(x),\{\vec A\cdot\nabla
v\}^-(x),\vec n(x),x\Big)d \mathcal H^{N-1}(x)\,,
\end{multline}
where for every $a,b\in\R^m$ and any unit vector $\vec \nu\in\R^N$
\begin{multline}\label{a1a2a3a4a5a6a7s8hhjhjjhjjjjjjkkkkjgjgjgjhlllllkkkggghhhjjjkkkk}
E_1\big(a,b,\vec \nu,x\big):=\inf\Bigg\{\liminf_{\e\to
0^+}\frac{1}{\omega_{N-1}\e}\int_{B_1(0)} F\Big(\,\e\nabla\big\{\vec
A\cdot\nabla \varphi_\e(y)\big\},\, \{\vec A\cdot\nabla
\varphi_\e\}(y),\,x\Big)dy:\;\; \varphi_\e\in
\mathcal{D}'\big(B_1(0),\R^d\big)
%\R^d
\quad\text{s.t.}\\ \vec A\cdot\nabla \varphi_\e\in
L^q\big(B_1(0),\R^m\big)\cap W^{1,p}\big(B_1(0),\R^m\big),\;\;\vec
A\cdot\nabla \varphi_\e\in\mathcal{M}\;\text{a.e. in}\;B_1(0)\\
\text{and}\;\;\{\vec A\cdot\nabla \varphi_\e\}(y)\to \xi(y,a,b,\vec
\nu)\;\text{in}\; L^q\big(B_1(0),\R^m\big)\Bigg\}\,.
\end{multline}
So it is sufficient to prove that for every $a,b\in\R^m$ and any
unit vector $\vec \nu\in\R^N$ we have
\begin{equation}\label{dhfhdghfhfgkkkhhh7788jjkk}
E_1\big(a,b,\vec \nu,x\big)\geq E_0\big(a,b,\vec \nu,x\big)\,,
\end{equation}
where $E_0\big(a,b,\vec \nu,x\big)$ is defined by
\er{a1a2a3a4a5a6a7s8hhjhjjhjjjjjjkkkkjgjgjgjhlllllkkk}. Without loss
of generality it is sufficient to prove
\er{dhfhdghfhfgkkkhhh7788jjkk} in the particular case where
$\vec\nu=\vec e_1$ and $I_{\vec \nu}:=\{y\in\R^N:\;|y\cdot \vec
e_j|<1/2\;\;\;\forall j=1,2\ldots N\}$ where $\{\vec e_1,\vec
e_2,\ldots,\vec e_N\}\subset\R^N$ is the standard orthonormal base
in $\R^N$. Choose a natural number $n\in\mathbb{N}$. Then changing
variables of integration $z=ny$ in
\er{a1a2a3a4a5a6a7s8hhjhjjhjjjjjjkkkkjgjgjgjhlllllkkkggghhhjjjkkkk}
we obtain
\begin{multline}\label{a1a2a3a4a5a6a7s8hhjhjjhjjjjjjkkkkjgjgjgjhlllllkkkggghhhjjjkkkkkkkkkkllllkkkk}
E_1\big(a,b,\vec \nu,x\big)=\inf\Bigg\{\liminf_{\e\to
0^+}\frac{1}{\omega_{N-1}n^{N-1}\e}\int\limits_{B_n(0)}
F\Big(\,\e\nabla\big\{\vec A\cdot\nabla \varphi_\e(y)\big\},\,
\{\vec A\cdot\nabla \varphi_\e\}(y),\,x\Big)dy:\; \varphi_\e\in
\mathcal{D}'\big(B_n(0),\R^d\big)\quad
%\R^d
\text{s.t.}\\ \vec A\cdot\nabla \varphi_\e\in
L^q\big(B_n(0),\R^m\big)\cap W^{1,p}\big(B_n(0),\R^m\big),\;\;
\vec A\cdot\nabla \varphi_\e\in\mathcal{M}\;\text{a.e. in}\;B_n(0)\\
\text{and}\;\{\vec A\cdot\nabla \varphi_\e\}(y)\to \xi(y,a,b,\vec
\nu)\;\text{in}\; L^q\big(B_n(0),\R^m\big)\Bigg\}\,.
\end{multline}
Next for every integers $i_1,i_2,\ldots,i_{N-1}\in\mathbb{Z}$
consider the set
\begin{equation}\label{a1a2a3a4a5a6a7s8hhjhjjhjjjjjjkkkkyiuouilokkkmnhhfghfhghgjkkkkkhghghffhjkljlkkkkkilljkkkipkpppppfbhfjk}
I_{(i_1,i_2,\ldots,i_{N-1})}:=\bigg\{z\in\R^N:\;|z\cdot\vec
e_1|<1/2\;\;\text{and}\;\;\big|z\cdot \vec
e_j-i_{j-1}\big|<1/2\;\;\;\forall j=2,3,\ldots, N\bigg\}.
\end{equation}
and set $I_0:=I_{(0,0,\ldots,0)}$. Then by
\er{a1a2a3a4a5a6a7s8hhjhjjhjjjjjjkkkkjgjgjgjhlllllkkkggghhhjjjkkkkkkkkkkllllkkkk}
\begin{multline}\label{a1a2a3a4a5a6a7s8hhjhjjhjjjjjjkkkkjgjgjgjhlllllkkkggghhhjjjkkkkkkkkkkllllkkkkkkkkklklklklkllkklkkkkkjjkkjjkjkjkjk}
E_1\big(a,b,\vec e_1,x\big)\geq
\frac{1}{\omega_{N-1}n^{N-1}}\,Card\bigg(\Big\{(i_1,i_2,\ldots,i_{N-1})\in\mathbb{Z}^{N-1}:\,I_{(i_1,i_2,\ldots,i_{N-1})}\in
B_n(0)\Big\}\bigg)\times\\ \times\inf\Bigg\{\liminf_{\e\to
0^+}\frac{1}{\e}\int_{I_0} F\Big(\,\e\nabla\big\{\vec A\cdot\nabla
\varphi_\e(y)\big\},\, \{\vec A\cdot\nabla
\varphi_\e\}(y),\,x\Big)dy:\;\; \varphi_\e\in
\mathcal{D}'(I_0,\R^d)\;\;
%\R^d
\text{s.t.}\\
\;\vec A\cdot\nabla \varphi_\e\in L^q\big(I_0,\R^m\big)\cap
W^{1,p}\big(I_0,\R^m\big),\;\;\vec A\cdot\nabla
\varphi_\e\in\mathcal{M}\;\text{a.e. in}\;I_0 \;\;
\text{and}\;\{\vec A\cdot\nabla \varphi_\e\}(y)\to \xi(y,a,b,\vec
\nu)\;\text{in}\;
L^q\big(I_0,\R^m\big)\Bigg\}\\
=\frac{1}{\omega_{N-1}n^{N-1}}\,Card\bigg(\Big\{(i_1,i_2,\ldots,i_{N-1})\in\mathbb{Z}^{N-1}:\,I_{(i_1,i_2,\ldots,i_{N-1})}\in
B_n(0)\Big\}\bigg)\,E_0(a,b,\vec e_1,x)\,.
\end{multline}
On the other hand clearly
\begin{equation}\label{a1a2a3a4a5a6a7s8hhjhjjhjjjjjjkkkkyiuouilokkkmnhhfghfhghgjkkkkkhghghffhjkljlkkkkkilljkkkipkpppppfbhfjkgjgkkk}
\lim\limits_{n\to
+\infty}\frac{1}{\omega_{N-1}n^{N-1}}\,Card\bigg(\Big\{(i_1,i_2,\ldots,i_{N-1})\in\mathbb{Z}^{N-1}:\,I_{(i_1,i_2,\ldots,i_{N-1})}\in
B_n(0)\Big\}\bigg)=1\,.
\end{equation}
Therefore, since $n\in\mathbb{N}$ was chosen arbitrary we deduce
\er{dhfhdghfhfgkkkhhh7788jjkk}. Plugging it into
\er{a1a2a3a4a5a6a7s8hhjhjjhjjjjjjkkkkjjjkjkkkkkk} completes the
proof.
\end{proof}
By the same method we can prove the following more general Theorem.
\begin{theorem}\label{dehgfrygfrgygen}
Let $\mathcal{M}$ be a subset of $\R^m$, $\O\subset\R^N$ be an open
set and $D\subset\O$ be a $\mathcal{H}^{N-1}$ $\sigma$-finite Borel
set. Consider $F\in C\big(\R^{m\times N^n}\times\R^{m\times
N^{(n-1)}}\times\ldots\times\R^{m\times N}\times
\R^m\times\R^N,\R\big)$, which satisfies $F\geq 0$ and the following
property: For every $x_0\in\O$ and every $\tau>0$ there exists
$\alpha>0$ satisfying
%for every $a\in \R^{m\times N}$, $b\in\R^m$ and $x\in\R^N$
%satisfying $|x-x_0|<\delta$ we have
\begin{multline}\label{vcjhfjhgjkgkgjgghj}
F\big(a_1,a_2,\ldots, a_n,b,x\big)-F\big(a_1,a_2,\ldots,
a_n,b,x_0\big)\geq -\tau F\big(a_1,a_2,\ldots, a_n,b,x_0\big)\\
\forall\, a_1\in \R^{m\times N^n}\;\forall\,a_2\in \R^{m\times
N^{n-1}}\ldots\forall\, a_n\in\R^{m\times N}\;\forall\,
b\in\R^m\;\forall\, x\in\R^N\;\;\text{such
that}\;\;|x-x_0|<\alpha\,.
\end{multline}
Furthermore, let $\vec A\in \mathcal{L}(\R^{d\times N};\R^m)$,
$q=(q_1,q_2,\ldots, q_m)\in\R^m$, $p\geq 1$ and
$v\in\mathcal{D}'(\O,\R^d)$ be such that $q_j\geq 1$,
%Let $q\geq 1$, $p\geq 1$ and $v\in\mathcal{D}'(\O,\R^d)$ be such that
$\vec A\cdot\nabla v\in L^q_{loc}(\O,\R^m)$ and
$$F\Big(0,0,\ldots,0,\{\vec A\cdot\nabla
v\}(x),x\Big)=0\quad\quad\text{for a.e.}\;\,x\in\O\,.$$ Assume also
that there exist three Borel mappings $\{\vec A\cdot\nabla
v\}^+(x):D\to\R^m$, $\{\vec A\cdot\nabla v\}^-(x):D\to\R^m$ and
$\vec n(x):D\to S^{N-1}$ such that for every $x\in D$ we have
\begin{multline}\label{L2009surfhh8128odno888jjjjjkkkkkkgen}
\lim\limits_{\rho\to 0^+}\frac{\int_{B_\rho^+(x,\vec
n(x))}\big|\{\vec A\cdot\nabla v\}(y)-\{\vec A\cdot\nabla
v\}^+(x)\big|^q\,dy} {\mathcal{L}^N\big(B_\rho(x)\big)}=0,\\
\lim\limits_{\rho\to 0^+}\frac{\int_{B_\rho^-(x,\vec
n(x))}\big|\{\vec A\cdot\nabla v\}(y)-\{\vec A\cdot\nabla
v\}^-(x)\big|^q\,dy}
{\mathcal{L}^N\big(B_\rho(x)\big)}=0\quad\quad\quad\quad\text{(see
Definition \ref{gdhgvdgjkdfgjkhdd})}.
\end{multline}
%Let $\O$, $D$, $\vec A$, $F$, $q$, $v$, $\vec n$, and $\{\vec
%A\cdot\nabla v\}^\pm$ as above.
Then for every
%sequence $\e_n\to 0^+$ as $n\to 0^+$ and every sequence
$\{v_\e\}_{\e>0}\subset\mathcal{D}'(\O,\R^d)$, satisfying $\vec
A\cdot\nabla v_\e\in L^q_{loc}(\O,\R^m)\cap W^{n,p}_{loc}(\O,\R^m)$,
$\{\vec A\cdot\nabla v_\e\}(x)\in\mathcal{M}$ for a.e. $x\in\O$ and
$\vec A\cdot\nabla v_\e\to \vec A\cdot\nabla v$ in
$L^q_{loc}(\O,\R^m)$ as $\e\to 0^+$, we have
\begin{multline}\label{a1a2a3a4a5a6a7s8hhjhjjhjjjjjjkkkkgen}
\liminf_{\e\to 0^+}\frac{1}{\e}\int_\O
F\bigg(\,\e^n\nabla^n\big\{\vec A\cdot\nabla
v_\e\big\}(x),\,\e^{n-1}\nabla^{n-1}\big\{\vec A\cdot\nabla
v_\e\big\}(x),\,\ldots,\,\e\nabla\big\{\vec
A\cdot\nabla v_\e(x)\big\},\, \{\vec A\cdot\nabla v_\e\}(x),\,x\bigg)dx\\
\geq \int_{D}E^{(n)}_0\Big(\{\vec A\cdot\nabla v\}^+(x),\{\vec
A\cdot\nabla v\}^-(x),\vec n(x),x\Big)d \mathcal H^{N-1}(x)\,,
\end{multline}
where for every $a,b\in\R^m$ and any unit vector $\vec \nu\in\R^N$
\begin{multline}\label{a1a2a3a4a5a6a7s8hhjhjjhjjjjjjkkkkjgjgjgjhlllllkkkgen}
E^{(n)}_0\big(a,b,\vec \nu,x\big):=\\ \inf\Bigg\{\liminf_{\e\to
0^+}\frac{1}{\e}\int_{I_{\vec \nu}} F\bigg(\,\e^n\nabla^n\big\{\vec
A\cdot\nabla \varphi_\e\big\}(y),\,\e^{n-1}\nabla^{n-1}\big\{\vec
A\cdot\nabla \varphi_\e\big\}(y),\,\ldots,\,\e\nabla\big\{\vec
A\cdot\nabla \varphi_\e\big\}(y),\, \{\vec A\cdot\nabla
\varphi_\e\}(y),\,x\bigg)dy:\\ \varphi_\e\in\mathcal{D}'(I_{\vec
\nu},\R^d)
%\R^d
\;\; \text{s.t.}\;\vec A\cdot\nabla \varphi_\e\in L^q(I_{\vec
\nu},\R^m)\cap W^{n,p}(I_{\vec \nu},\R^m),\;\;\vec A\cdot\nabla
\varphi_\e\in\mathcal{M}\;\text{a.e. in}\;I_{\vec \nu} \\
\text{and}\;\{\vec A\cdot\nabla \varphi_\e\}(y)\to \xi(y,a,b,\vec
\nu)\;\text{in}\; L^q(I_{\vec \nu},\R^m)\Bigg\}\,.
\end{multline}
Here $I_{\vec \nu}:=\{y\in\R^N:\;|y\cdot \vec\nu_j|<1/2\;\;\;\forall
j=1,2\ldots N\}$ where
$\{\vec\nu_1,\vec\nu_2,\ldots,\vec\nu_N\}\subset\R^N$ is an
orthonormal base in $\R^N$ such that $\vec\nu_1:=\vec \nu$ and
\begin{equation}\label{fhyffgfgfgfffgfgen}
\xi(y,a,b,\vec \nu):=\begin{cases}a\quad\text{if}\;y\cdot\vec
\nu>0\,,\\ b\quad\text{if}\;y\cdot\vec \nu<0\,.\end{cases}
\end{equation}
\end{theorem}
We have the following particular case of Theorem
\ref{dehgfrygfrgygen}.
\begin{theorem}\label{dehgfrygfrgygenjklhhj}
Let $\mathcal{M}$ be a subset of $\R^m$, $\O\subset\R^N$ be an open
set and $D\subset\O$ be a $\mathcal{H}^{N-1}$ $\sigma$-finite Borel
set. Furthermore, let $q_1,q_2,q_3\geq 1$, $p\geq 1$ and let $F$ be
a continuous function defined on
$$
\big\{\R^{k\times N^{n+1}}\times\R^{d\times
N^{n+1}}\times\R^{m\times
N^n}\big\}\times\ldots\times\big\{\R^{k\times N\times
N}\times\R^{d\times N\times N}\times\R^{m\times
N}\big\}\times\big\{\R^{k\times N}\times\R^{d\times
N}\times\R^{m}\big\}\times\R^N,
$$
taking values in $\R$ and satisfying $F\geq 0$ and the following
property: For every $x_0\in\O$ and every $\tau>0$ there exists
$\alpha>0$ satisfying
%for every $a\in \R^{m\times N}$, $b\in\R^m$ and $x\in\R^N$
%satisfying $|x-x_0|<\delta$ we have
\begin{multline}\label{vcjhfjhgjkgkgjgghjfhfhf}
F\big(a_1,a_2,\ldots, a_n,b,x\big)-F\big(a_1,a_2,\ldots,
a_n,b,x_0\big)\geq -\tau F\big(a_1,a_2,\ldots, a_n,b,x_0\big)\quad
\forall\, a_1\in \R^{k\times N^{n+1}}\times\R^{d\times
N^{n+1}}\times\R^{m\times N^n}\\ \ldots\forall\, a_n\in\R^{k\times
N\times N}\times\R^{d\times N\times N}\times\R^{m\times
N}\;\;\forall\, b\in\R^{k\times N}\times\R^{d\times
N}\times\R^{m}\;\;\forall\, x\in\R^N\;\;\text{such
that}\;\;|x-x_0|<\alpha\,.
\end{multline}
Let $v(x)\in W^{1,q_1}_{loc}(\O,\R^k)$, $\bar m(x)\in
L^{q_2}_{loc}(\O,\R^{d\times N})$ and $\varphi\in
L^{q_3}_{loc}(\O,\R^{m})$ be such that  $div_x \bar m(x)\equiv 0$ in
$\O$ and
$$F\Big(0,0,\ldots,0,\{\nabla v,\bar m,\varphi\},x\Big)=0
\quad\text{a.e.~in}\; \Omega\,.$$ Assume also that there exist Borel
mappings $\{\nabla v\}^+(x):D\to\R^{k\times N}$, $\{\nabla
v\}^-(x):D\to\R^{k\times N}$, $\bar m^+(x):D\to\R^{d\times N}$,
$\bar m^-(x):D\to\R^{d\times N}$, $\varphi^+(x):D\to\R^{m}$,
$\varphi^-(x):D\to\R^{m}$ and $\vec n(x):D\to S^{N-1}$ such that for
every $x\in D$ we have
\begin{multline}\label{L2009surfhh8128odno888jjjjjkkkkkkgenhjjhjkj}
\lim\limits_{\rho\to 0^+}\frac{1}
{\mathcal{L}^N\big(B_\rho(x)\big)}\int_{B_\rho^+(x,\vec
n(x))}\Bigg(\Big|\nabla v(y)-\{\nabla v\}^+(x)\Big|^{q_1}+\Big|\bar
m(y)-\bar m^+(x)\Big|^{q_2}+\Big|\varphi(y)-\varphi^+(x)\Big|^{q_3}\Bigg)\,dy=0\,,\\
\lim\limits_{\rho\to 0^+}\frac{1}
{\mathcal{L}^N\big(B_\rho(x)\big)}\int_{B_\rho^-(x,\vec
n(x))}\Bigg(\Big|\nabla v(y)-\{\nabla v\}^-(x)\Big|^{q_1}+\Big|\bar
m(y)-\bar
m^-(x)\Big|^{q_2}+\Big|\varphi(y)-\varphi^-(x)\Big|^{q_3}\Bigg)\,dy=0\,.
\end{multline}
%Let $\O$, $D$, $\vec A$, $F$, $q$, $v$, $\vec n$, and $\{\vec
%A\cdot\nabla v\}^\pm$ as above.
Then for every
%sequence $\e_n\to 0^+$ as $n\to 0^+$ and every sequence
$\{v_\e\}_{\e>0}\subset W^{1,q_1}_{loc}(\O,\R^k)\cap
W^{(n+1),p}_{loc}(\O,\R^k)$, $\{m_\e\}_{\e>0}\subset
L^{q_2}_{loc}(\O,\R^{d\times N})\cap W^{n,p}_{loc}(\O,\R^{d\times
N})$ and $\{\psi_\e\}_{\e>0}\subset L^{q_3}_{loc}(\O,\R^{m})\cap
W^{n,p}_{loc}(\O,\R^m)$ satisfying $div_x m_\e(x)\equiv 0$ in $\O$,
$\psi_\e(x)\in\mathcal{M}$ for a.e. $x\in\O$, $v_\e\to v$ in
$W^{1,q_1}_{loc}(\O,\R^k)$ as $\e\to 0^+$, $m_\e\to \bar m$ in
$L^{q_2}_{loc}(\O,\R^{d\times N})$ as $\e\to 0^+$ and $\psi_\e\to
\varphi$ in $L^{q_3}_{loc}(\O,\R^{m})$, we have
\begin{multline}\label{a1a2a3a4a5a6a7s8hhjhjjhjjjjjjkkkkgenjbhghgh}
\liminf_{\e\to 0^+}\int\limits_{\O}\frac{1}{\e} F\bigg(
\big\{\e^n\nabla^{n+1}v_{\e},\,\e^n\nabla^n
m_\e,\,\e^n\nabla^n\psi_\e\big\},\,\ldots\,,\big\{\e\nabla^2v_{\e},\,\e\nabla
m_\e,\,\e\nabla\psi_\e\big\},\,\big\{\nabla
v_{\e},\,m_\e,\,\psi_\e\big\},\,x\bigg)\,dx\\
\geq \int_{D}\bar E^{(n)}_0\bigg(\{\nabla v\}^+(x),\bar
m^+(x),\varphi^+(x),\{\nabla v\}^-(x),\bar m^-(x),\varphi^-(x),\vec
n(x),x\bigg)d \mathcal H^{N-1}(x)\,,
\end{multline}
where for
%every $a,b\in\R^m$ and
any unit vector $\vec \nu\in\R^N$
\begin{multline}\label{a1a2a3a4a5a6a7s8hhjhjjhjjjjjjkkkkjgjgjgjhlllllkkkgenhgjkggg}
\bar E^{(n)}_0\bigg(\{\nabla v\}^+,\bar m^+,\varphi^+,\{\nabla v\}^-,\bar m^-,\varphi^-,\vec \nu,x\bigg):=\\
\inf\Bigg\{\liminf_{\e\to 0^+}\int_{I_{\vec \nu}}
\frac{1}{\e}F\bigg(\Big\{\e^n\nabla^{n+1}\sigma_\e(y),\e^n\nabla^n\theta_\e(y),
\e^n\nabla^n\gamma_\e(y)\Big\},\ldots,\Big\{\nabla\sigma_\e(y),\theta_\e(y),\gamma_\e(y)\Big\}
,x\bigg)dy:\\
\sigma_\e\in W^{1,q_1}(I_{\vec \nu},\R^k)\cap W^{(n+1),p}(I_{\vec
\nu},\R^k),\;\theta_\e\in L^{q_2}(I_{\vec \nu},\R^{d\times N})\cap
W^{n,p}(I_{\vec \nu},\R^{d\times N}),\\ \big\{\gamma_\e:I_{\vec
\nu}\to\mathcal{M}\big\}\in L^{q_3}(I_{\vec \nu},\R^{m})\cap
W^{n,p}(I_{\vec \nu},\R^m)\;\\ \text{s.t.}\;
\Div_y\theta_\e(y)\equiv
0,\;\nabla\sigma_\e(y)\to\sigma\big(y,\{\nabla v\}^+,\{\nabla
v\}^-,\vec\nu\big)\;\text{in}\;L^{q_1}(I_{\vec \nu},\R^{k\times N}),\\
\theta_\e(y)\to\theta(y,\bar m^+,\bar
m^-,\vec\nu)\;\text{in}\;L^{q_2}(I_{\vec \nu},\R^{d\times
N}),\;\gamma_\e(y)\to\gamma(y,\varphi^+,\varphi^-,\vec\nu)\;\text{in}\;L^{q_3}(I_{\vec
\nu},\R^{m})
%\\\varphi_\e:I_{\vec \nu}\to\R^d\;\; \text{s.t.}\;\vec A\cdot\nabla
%\varphi_\e\in L^q(I_{\vec \nu},\R^m)\cap W^{n,p}(I_{\vec
%\nu},\R^m)\;\text{and}\;\{\vec A\cdot\nabla \varphi_\e\}(y)\to
%\xi(y,a,b,\vec \nu)\;
%\text{in}\; L^q(I_{\vec \nu},\R^m)
\Bigg\}\,.
\end{multline}
Here $I_{\vec \nu}:=\{y\in\R^N:\;|y\cdot \vec\nu_j|<1/2\;\;\;\forall
j=1,2\ldots N\}$ where
$\{\vec\nu_1,\vec\nu_2,\ldots,\vec\nu_N\}\subset\R^N$ is an
orthonormal base in $\R^N$ such that $\vec\nu_1:=\vec \nu$ and
\begin{multline}\label{fhyffgfgfgfffgfgenkjgjgkgkg}
\sigma\big(y,\{\nabla v\}^+,\{\nabla
v\}^-,\vec\nu\big):=\begin{cases}\{\nabla
v\}^+\quad\text{if}\;\,y\cdot\vec \nu>0\,,\\
\{\nabla v\}^-\quad\text{if}\;\,y\cdot\vec
\nu<0\,,\end{cases}\quad\theta\big(y,\bar m^+,\bar
m^-,\vec\nu\big):=\begin{cases}\bar m^+\quad\text{if}\;\,y\cdot\vec \nu>0\,,\\
\bar m^-\quad\text{if}\;\,y\cdot\vec
\nu<0\,,\end{cases}\\ \text{and}\quad\gamma\big(y,\varphi^+,\varphi^-,\vec\nu\big):=\begin{cases}\varphi^+\quad\text{if}\;\,y\cdot\vec \nu>0\,,\\
\varphi^-\quad\text{if}\;\,y\cdot\vec \nu<0\,.\end{cases}
\end{multline}
\end{theorem}
%%%%%%%%%%%%%%%%%%%%%%%%%%%%%%%%%%%%%%%%%%%%%%%%%%%%%%%%%%%%%%%%%%%%%%%%%%%%%%%%%%%%%%%%%%%%%%%%%%%%%%%%%%%%%%%%%%%%%%%%%%%%%%%%%%%%%%%%%%%%%%%%%%%%%%%%%%%%%%%%%%%%%%%%%%%%%%%%%%%%%%%%%%%%%%%%%%%%%%%%%%%%%%%%%%%%%%%%%%%%%%%%%%%%%%%%%%%%%%%%%%%%%%%%%%%%%%%%%%%%%%%%%%%%%%%%%%%%%%%%%
\begin{proof}
Without any loss of generality we may assume that
$\O=\big\{x=(x_1,x_2,\ldots, x_N)\in\R^N:\;|x_j|<c_0\;\forall
j\big\}$. for some $c_0>0$. Let $\{v_\e\}_{\e>0}\subset
W^{1,q_1}_{loc}(\O,\R^k)\cap W^{(n+1),p}_{loc}(\O,\R^k)$,
$\{m_\e\}_{\e>0}\subset L^{q_2}_{loc}(\O,\R^{d\times N})\cap
W^{n,p}_{loc}(\O,\R^{d\times N})$ and
$\big\{\psi_\e:\O\to\mathcal{M}\big\}_{\e>0}\subset
L^{q_3}_{loc}(\O,\R^{m})\cap W^{n,p}_{loc}(\O,\R^m)$ be such that
$div_x m_\e(x)\equiv 0$ in $\O$, $v_\e\to v$ in
$W^{1,q_1}_{loc}(\O,\R^k)$ as $\e\to 0^+$, $m_\e\to \bar m$ in
$L^{q_2}_{loc}(\O,\R^{d\times N})$ as $\e\to 0^+$ and $\psi_\e\to
\varphi$ in $L^{q_3}_{loc}(\O,\R^{m})$.
%Without loss of generality we may assume that $\psi_\e$ and $m_\e$ are $C^1$ functions.
Clearly there exist
$L(x'),L_\e(x'):\big\{x'\in\R^{N-1}:\;|x'_j|<c_0\;\forall
j\big\}\to\R^{d\times(N-1)}$ such that $\Div_{x'}L(x')\equiv \bar
m_{1}(0,x')$ and $\Div_{x'}L_\e(x')\equiv m_{\e,1}(0,x')$, where we
denote by $\bar m_{1}(x):\O\to\R^d$ and $\bar m'(x):\O\to
\R^{d\times (N-1)}$ the first column and the rest of the matrix
valued function $\bar m(x):\O\to\R^{d\times N}$, so that $\big(\bar
m_{1}(x),\bar m'(x)\big):=\bar m(x):\O\to\R^{d\times N}$, and we
denote by $m_{\e,1}(x):\O\to\R^d$ and $m'_\e(x):\O\to \R^{d\times
(N-1)}$ the first column and the rest of the matrix valued function
$m_\e(x):\O\to\R^{d\times N}$, so that
$\big(m_{\e,1}(x),m'_\e(x)\big):=m_\e(x):\O\to\R^{d\times N}$. Then
define $\Psi_\e,\Psi:\R^N\to\R^m$ and $M_\e,M:\R^N\to\R^{d\times
(N-1)}$ by
\begin{multline}\label{vhgvtguyiiuijjkjkkjggjkjjhjkkllgvvjhkjhk}
\Psi_\e(x):=\int_{0}^{x_1}\psi_\e(s,x')ds\,,\quad\Psi(x):=\int_{0}^{x_1}\varphi(s,x')ds\,,\quad M(x):=-L(x')+\int_{0}^{x_1}\bar m'(s,x')ds\quad\text{and}\\
M_\e(x):=-L_\e(x')+\int_{0}^{x_1}m'_\e(s,x')ds\quad\quad\forall
x=(x_1,x'):=(x_1,x_2,\ldots x_N)\in\O\,,
\end{multline}
Then, since $div_x \bar m\equiv 0$ and $div_x m_\e\equiv 0$, by
\er{vhgvtguyiiuijjkjkkjggjkjjhjkkllgvvjhkjhk} we obtain
\begin{multline}\label{vhgvtguyiiuijjkjkkjggjkjjhjkkllghjjhjhhkjhkljljhlk}
\frac{\partial\Psi}{\partial x_1}(x)=\varphi(x)\,,\quad
\frac{\partial M}{\partial x_1}(x)=\bar m'(x)\,,\quad
-div_{x'}M(x)=\bar m_{1}(x)\quad\text{for a.e.}\;\;
x=(x_1,x')\in\O\,,\quad\text{and}\\
\frac{\partial\Psi_\e}{\partial x_1}(x)=\psi_\e(x)\,,\quad
\frac{\partial M_\e}{\partial x_1}(x)=m'_\e(x)\,,\quad
-div_{x'}M_\e(x)=m_{\e,1}(x)\quad\quad\text{for a.e.}\;\;
x=(x_1,x')\in\O\,.
\end{multline}
Therefore,
%in the case $q_1=q_2=q_3$
the result follows by applying
Theorem \ref{dehgfrygfrgygen} to the functions $\{v,M,\Psi\}$ and to
the sequence $\{v_\e,M_\e,\Psi_\e\}$.
%In the general case we just
%repeat to this case the arguments of the proof of Theorem
%\ref{dehgfrygfrgygen} without any significant change.
\end{proof}

\section{Further estimates for the lower
bound}\label{vdhgvdfgbjfdhgf}
\begin{lemma}\label{L2009.02new}
Let
%$\vec A\in\mathcal{L}(\R^{d\times N},\R^k)$,
%$\vec B\in\mathcal{L}(\R^{m},\R^k)$
%, $\vec P\in\mathcal{L}(\R^{k\times N},\R^k)$
%and
$\vec Q\in\mathcal{L}(\R^{m\times N},\R^d)$ be linear operator and
%$\Omega$ be a bounded domain
%order %(in particular a bounded $BVG$-domain)
%$F\in C^1(\R^{k\times N\times N}\times\R^{m\times
%N}\times\R^{k\times N}\times\R^m\,,\R)$
let $F\in C^0(\R^{k}\times\R^d\times\R^m,\R)$ be such that $F\geq 0$
and there exist $C>0$, $q\geq 1$ and $p=(p_1,p_2,\ldots,
p_k)\in\R^k$ such that $p_j\geq 1$ for every $j$ and
\begin{equation}\label{RsTT}
0\leq F(a,b,c)\leq
C\Big(|a|^{p}+|b|^q+|c|^q+1\Big)\quad\quad\forall\,
(a,b,c)\in\R^{k}\times\R^d\times\R^m
\end{equation}
(see Definition \ref{gdhgvdgjkdfgjkhdd}). Next let $\vec\nu\in
S^{N-1}$ and let $\big\{m_\e(x)\big\}_{0<\e<1}\subset
L^{p}_{loc}(I_{\vec \nu},\R^k)$ and $m_0(x)\in L^{p}(I_{\vec
\nu},\R^k)$ be such that $F\big(m_0(x),0,0\big)= 0$ for a.e. $x\in
I_{\vec \nu}$ and $\,\lim_{\e\to 0^+}m_\e=m_0$ in
$L^{p}_{loc}(I_{\vec \nu},\R^k)$, where, as before, $I_{\vec
\nu}:=\{y\in\R^N:\;|y\cdot \vec\nu_j|<1/2\;\;\;\forall j=1,
%2,
\ldots
N\}$ and $\{\vec\nu_1,
%\vec\nu_2,
\ldots,\vec\nu_N\}\subset\R^N$ is an
orthonormal base in $\R^N$, such that $\vec\nu_1:=\vec \nu$.
%assume $p\geq\max{(p_1,p_2)}$ and
Furthermore, let
%%$\bar p:=\max\limits_{1\leq j\leq k}p_j$
%, $\big\{v_\e(x)\big\}_{0<\e<1}\subset L^{\bar p}_{loc}(I_{\vec\nu},\R^d)$
%%and
$\big\{\varphi_\e(x)\big\}_{0<\e<1}\subset L^{q}_{loc}(I_{\vec
\nu},\R^m)$ be such that
%$\lim_{\e\to 0^+}v_\e/\e=0$ in $L^{\bar p}_{loc}(I_{\vec \nu},\R^d)$,
$\,\lim_{\e\to 0^+}\varphi_\e=0$ in $L^{q}_{loc}(I_{\vec
\nu},\R^m)$,
%$\lim_{\e\to 0^+}\vec A\cdot \nabla v_\e=0$ in $L^{p}_{loc}(I_{\vec\nu},\R^k)$
%, $\lim_{\e\to 0^+}\big(\e\vec P\cdot\nabla\{\vec A\cdot\nabla v_\e\}\big)=0$ in $L^p_{loc}(I_{\vec \nu},\R^{k})$
and $\lim_{\e\to 0^+}\big(\e\vec Q\cdot\nabla\varphi_\e\big)=0$ in
$L^q_{loc}(I_{\vec \nu},\R^{d})$.
%We suppose also that
%%%either $\vec P\equiv 0$ or
%$\lim_{\e\to 0^+}\|P\|\nabla v_\e=0$ in $L^{\bar p}_{loc}(I_{\vec\nu},\R^{d\times N})$.
%%%Finally
%%%assume that
%%%\begin{equation}
%\begin{multline}
%%%\label{fvyjhfyffhjfgh}
%%%\lim_{\rho\to 0^+}\Bigg\{\limsup_{\e\to
%%%0^+}\int_{I_{\rho,\vec \nu}}\frac{1}{\e}
%\times\\ \times
%%%F\big(m_\e(x)
%+\e\big(\vec P\cdot\nabla \{\vec A\cdot \nabla
%v_{\e}\}\big)(x)+\e\big(\vec Q\cdot\nabla\{\vec
%B\cdot\varphi_{\e}\}\big)(x)+\big(\vec A\cdot\nabla
%v_{\e}\big)(x)+\big(\vec
%B\cdot\varphi_{\e}\big)(x)
%%%\big)dx\Bigg\}=0,
%\end{multline}
%%%\end{equation}
%%%where $I_{\rho,\vec \nu}:=\{y\in\R^N:\;<(1-\rho)/2<|y\cdot
%%%\vec\nu_j|<1/2\;\;\;\forall j=1,2\ldots N\}$.
Then there exist $\big\{r_\e\big\}_{0<\e<1}\subset(0,1)$
%, $\big\{u_\e(x)\big\}_{0<\e<1}\subset C^{\infty}_{c}(I_{\vec\nu},\R^d)$
and $\big\{\psi_\e(x)\big\}_{0<\e<1}\subset
C^{\infty}_{c}(I_{\vec \nu},\R^m)$ such that $\lim_{\e\to
0^+}r_\e=1$,
%$\lim_{\e\to 0^+}u_\e/\e=0$ in $L^{\bar p}(I_{\vec\nu},\R^d)$,
$\,\lim_{\e\to 0^+}\psi_\e=0$ in $L^{q}(I_{\vec \nu},\R^m)$,
%$\lim_{\e\to 0^+}\vec A\cdot \nabla u_\e=0$ in $L^{p}(I_{\vec \nu},\R^k)$,
%$\lim_{\e\to 0^+}\big(\e\vec P\cdot\nabla\{\vec A\cdot\nabla u_\e\}\big)=0$ in $L^p(I_{\vec\nu},\R^{k})$
$\lim_{\e\to 0^+}\big(\e\vec Q\cdot\nabla \psi_\e\big)=0$ in
$L^q(I_{\vec \nu},\R^{d})$,
%$\lim_{\e\to0^+}\|\vec P\|\,\nabla u_\e=0$ in $L^{\bar p}(I_{\vec\nu},\R^{d\times N})$,
$\,\lim_{\e\to 0^+}m_{(r_\e\e)}\big(r_\e
x\big)=m_0(x)$ in $L^{p}(I_{\vec \nu},\R^k)$ and
\begin{multline}\label{fvyjhfyffhjfghgjkghfff}
\liminf_{\e\to 0^+}\int\limits_{I_{\vec \nu}}\frac{1}{\e}
F\bigg(m_\e(x)\,,\,
%\e\big(\vec P\cdot\nabla \{\vec A\cdot \nabla v_{\e}\}\big)(x)+
\e\vec Q\cdot\nabla\varphi_{\e}(x)\,,\,
%\big(\vec A\cdot\nabla v_{\e}\big)(x)+
\varphi_{\e}(x)\bigg)dx \geq \liminf_{\e\to 0^+}\int\limits_{I_{\vec
\nu}}\frac{1}{\e}F\bigg(m_{(r_\e\e)}\big(r_\e x\big)\,,\,
%\e\big(\vec P\cdot\nabla \{\vec A\cdot \nabla u_{\e}\}\big)(x)+
\e\vec Q\cdot\nabla\psi_{\e}(x)\,,\,
%\big(\vec A\cdot\nabla u_{\e}\big)(x)+
\psi_{\e}(x)\bigg)dx\,.
\end{multline}
\end{lemma}
\begin{proof}
Clearly we may assume
\begin{equation}\label{ggfghjjhfhfjfhjhjhjhjfjkdghfdnew}
\liminf_{\e\to 0^+}\int\limits_{I_{\vec \nu}}\frac{1}{\e}
F\bigg(m_\e(x)\,,\,
%\e\big(\vec P\cdot\nabla \{\vec A\cdot \nabla v_{\e}\}\big)(x)+
\e\vec Q\cdot\nabla\varphi_{\e}(x)\,,\,
%\big(\vec A\cdot\nabla v_{\e}\big)(x)+
\varphi_{\e}(x)\bigg)dx<+\infty\,,
\end{equation}
otherwise it is trivial. Moreover, without any loss of generality we
may assume that $\vec\nu=\vec e_1:=(1,
%0,
0,\ldots,0)$ and $I_{\vec
\nu}=I:=\big\{y=(y_1,
%y_2
\ldots, y_N)\in\R^N:\;|y_j|<1/2\;\;\;\forall j=1,
%2,
\ldots, N\big\}$.
Furthermore, since, by mollification, we always can approximate
%$v_\e$ and
$\varphi_\e$ by smooth functions, there exist
%$\big\{\bar v_\e(x)\big\}_{0<\e<1}\subset C^\infty$ and
$\big\{\bar\varphi_\e(x)\big\}_{0<\e<1}\subset C^\infty(I_{\vec
\nu},\R^m)$, such that
%$\lim_{\e\to 0^+}\bar v_\e/\e=0$ in $L^{\bar p}_{loc}(I_{\vec \nu},\R^d)$,
$\,\lim_{\e\to 0^+}\bar\varphi_\e=0$ in $L^{q}_{loc}(I_{\vec
\nu},\R^m)$,
%$\lim_{\e\to 0^+}\vec A\cdot \nabla \bar v_\e=0$ in $L^{p}_{loc}(I_{\vec \nu},\R^k)$,
%$\lim_{\e\to 0^+}\big(\e\vec P\cdot\nabla\{\vec A\cdot\nabla \bar v_\e\}\big)=0$ in $L^p_{loc}(I_{\vec \nu},\R^{k})$,
$\lim_{\e\to 0^+}\big(\e\vec Q\cdot\nabla\bar\varphi_\e\big)=0$ in
$L^q_{loc}(I_{\vec \nu},\R^{d})$ and
\begin{multline}\label{fughfighdfighfihklhh}
\frac{1}{\e}F\bigg(m_\e(x)\,,\,
%\e\big(\vec P\cdot\nabla \{\vec A\cdot \nabla \bar v_{\e}\}\big)(x)+
\e\vec Q\cdot\nabla\bar \varphi_{\e}(x)\,,\,
%\big(\vec A\cdot\nabla \bar v_{\e}\big)(x)+
\bar \varphi_{\e}(x)\bigg)- \frac{1}{\e}F\bigg(m_\e(x)\,,\,
%\e\big(\vec P\cdot\nabla \{\vec A\cdot \nabla v_{\e}\}\big)(x)+
\e\vec Q\cdot\nabla\varphi_{\e}(x)\,,\,
%\big(\vec A\cdot\nabla v_{\e}\big)(x)+
\varphi_{\e}(x)\bigg)\to 0 \;\;\;\text{in}\;\,L^1_{loc}(I_{\vec
\nu},\R)\;\,\text{as}\;\,\e\to 0^+.
\end{multline}
%and either $\vec P\equiv 0$ or $\lim_{\e\to 0^+}\nabla \bar v_\e=0$
%in $L^{\bar p}_{loc}(I_{\vec \nu},\R^{d\times N})$.

 Next consider $l(t)\in C^\infty(\R,\R)$ with the properties
%$\int_0^1l(s)ds=1/2$ and
\begin{equation}\label{L2009smooth1new}
\begin{cases}l(t)=1 \quad\quad\quad\text{for every }t\in(-\infty,\delta)\,,\\
l(t)\in[0,1] \quad\;\;\text{for every }t\in[\delta,1-\delta]\,,\\
l(t)=0 \quad\quad\quad\text{for every }
t\in(1-\delta,+\infty)\,,\end{cases}
\end{equation}
where $\delta\in(0,1/2)$. Clearly such a function exists.
Then for every $0\leq t<1/2$ define
%\begin{multline}\label{L2009deftvu1}
%\psi^{(1)}_{\e,t}(x):=\varphi_\e(x)+\bigg(l\Big((x_1-t)/\e\Big)+l\Big(-(t+x_1)/\e\Big)\bigg)\cdot\Big(\psi_\e(x)-\varphi_\e(x)\Big)\quad\forall
%x\in\R^N\,,\\
%u^{(1)}_{\e,t}(x):=v_\e(x)+\bigg(l\Big((x_1-t)/\e\Big)+l\Big(-(t+x_1)/\e\Big)\bigg)\cdot\Big(u_\e(x)-v_\e(x)\Big)\quad\forall
%x\in\R^N\,,
%\end{multline}
%\begin{equation}\label{L2009deftvu1new}
%\psi^{(0)}_{\e,t}(x):=\varphi_\e(x)\quad\forall
%x\in\R^N\,,\quad\text{and}\quad
%u^{(0)}_{\e,t}(x):=v_\e(x)\quad\forall x\in\R^N\,,
%\end{equation}
%and by induction
\begin{equation}\label{L2009deftvu1hhhjnew}
%\begin{split}
\psi_{\e,t}(x):=\bar\varphi_\e(x)\times\Prod\limits_{j=1}^{N}\bigg(l\big((x_j-t)/\e\big)\cdot
l\big(-(t+x_j)/\e\big)\bigg)\quad \quad\forall x\in\R\,.
%,\\u_{\e,t}(x):=\bar v_\e(x)\times\Prod\limits_{j=1}^{N}\bigg(l\big((x_j-t)/\e\big)\cdot
%l\big(-(t+x_j)/\e\big)\bigg)\quad\quad\forall x\in\R \,.
%\end{split}
\end{equation}
Then for every $t\in[0,1/2)$ clearly $\psi_{\e,t}\in
C^\infty(I^1,\R^m)$
%, $u_{\e,t}\in C^\infty(I^1,\R^d)$,
where
\begin{equation}\label{vhjfyhjgjkgjkghjk}
I^s=\Big\{y=(y_1,
%y_2,
\ldots, y_N)\in\R^N:\;|y_j|<s/2\;\;\;\forall j=1,
%2,
\ldots,
N\Big\}\quad\forall s>0\,.
\end{equation}
Moreover, for each such $t\in[0,1/2)$ we have
\begin{equation}\label{L2009eqgl1new}
\begin{cases}
\psi_{\e,t}(x)=\bar\varphi_\e(x)
%\;\;\text{and}\;\; u_{\e,t}(x)=\bar v_\e(x)
\quad\text{if for every}\;\; j\in\{1,
%2,
\ldots,N\}\;\;\text{we have}\;\; |x_j|<t\,,\\
\psi_{\e,t}(x)=0
%\;\;\text{and}\;\; u_{\e,t}(x)=0
\quad\text{if
}|x_j|>t+(1-\delta)\e\;\;\text{for some}\;\; j\in\{1,
%2,
\ldots,N\}\,.
\end{cases}
\end{equation}
So, by \er{L2009eqgl1new} for small $\e>0$ we have $\psi_{\e,t}\in
C^\infty_c(I^{t+\e},\R^m)$.
% and $u_{\e,t}\in C^\infty_c(I^{t+\e},\R^d)$.
%%%%%%%%%%%%%%%%%%%%%%%%%%%%%%%%%%%%%%%%%%%%%%%%%%%%%%%%%%%%%%%%%%%%%%%%%%%%%%%%%%%%%%%%%%%%%%%%%%%%%%%%%%%%%%%%%%%%%%%%%%%%%%%%%%%%%%%%%%%%%%%%%%%%%%%%%%%%%%%%%%%%%%%%%%%%%%%%%%%%%%%%%%%%%%%%%%%%%%%%%%%%%%%%%%%%%%%%%%%%%%%%%%%%%%%%%%%%%%%%%%%%%%%%%%%%%%%%%
Next we will prove that for every $\tau\in(0,1/2)$ we have
%\begin{multline}
\begin{equation}
\label{L2009small1new} \lim_{\e\to
0}\int_{0}^{\tau}\,\int_{\cup_{j=1}^{N}\{ x\in
I^{t+\e}:\,|x_j|>t\}}\frac{1}{\e}
%\times\\ \times
F\bigg(m_\e(x)\,,\,
%\e\big(\vec P\cdot\nabla \{\vec A\cdot \nabla u_{\e,t}\}\big)(x)+
\e\vec Q\cdot\nabla\psi_{\e,t}(x)\,,\,
%\big(\vec A\cdot\nabla u_{\e,t}\big)(x)+
\psi_{\e,t}(x)\bigg)dx\,dt=0\,.
\end{equation}
%\end{multline}
Indeed, fix $\tau_0\in(\tau,1/2)$. Then for $0<\e<(\tau_0-\tau)/2$
we have
\begin{multline}\label{L2009small1hjhjjhhjhhjjioiiounew}
\int_{0}^{\tau}\,\int_{\cup_{j=1}^{N}\{ x\in
I^{t+\e}:\,|x_j|>t\}}\frac{1}{\e}
%\times\\ \times
F\bigg(m_\e(x)\,,\,
%\e\big(\vec P\cdot\nabla \{\vec A\cdot \nabla u_{\e,t}\}\big)(x)+
\e\vec Q\cdot\nabla\psi_{\e,t}(x)\,,\,
%\big(\vec A\cdot\nabla u_{\e,t}\big)(x)+
\psi_{\e,t}(x)\bigg)dx\,dt\leq\\
\sum\limits_{j=1}^{N}\int_{0}^{\tau}\,\int_{\{ x\in
I^{t+\e}:\,|x_j|>t\}}\frac{1}{\e}
%\times\\ \times
F\bigg(m_\e(x)\,,\,
%\e\big(\vec P\cdot\nabla \{\vec A\cdot \nabla u_{\e,t}\}\big)(x)+
\e\vec Q\cdot\nabla\psi_{\e,t}(x)\,,\,
%\big(\vec A\cdot\nabla u_{\e,t}\big)(x)+
\psi_{\e,t}(x)\bigg)dxdt\leq\\
\sum\limits_{j=1}^{N}\frac{1}{\e}\int\limits_{-\tau}^{\tau}\Bigg\{
\int\limits_{\{x: x\in I^{|t|+\e},-\e<x_j<\e\}}F\bigg(m_\e(x+t\vec
e_j)\,,\,
%\e\big(\vec P\cdot\nabla \{\vec A\cdot \nabla u_{\e,|t|}\}\big)(x+t\vec e_j)+
%\\
\e\vec Q\cdot\nabla\psi_{\e,|t|}(x+t\vec e_j)\,,\,
%\big(\vec A\cdot\nabla u_{\e,|t|}\big)(x+t\vec e_j)+
\psi_{\e,|t|}(x+t\vec e_j)\bigg)dx\Bigg\}dt\\
\leq \sum\limits_{j=1}^{N}\frac{1}{\e}\int\limits_{-\e}^{\e}\Bigg\{
\int\limits_{I^{\tau+\e}}F\bigg(m_\e(x+s\vec e_j)\,,\,
%\e\big(\vec P\cdot\nabla \{\vec A\cdot \nabla u_{\e,|x_j|}\}\big)(x+s\vec e_j)+
\e\vec Q\cdot\nabla\psi_{\e,|x_j|}(x+s\vec e_j)\,,\,
%\big(\vec A\cdot\nabla u_{\e,|x_j|}\big)(x+s\vec e_j)+
\psi_{\e,|x_j|}(x+s\vec e_j)\bigg)dx\Bigg\}ds\\
\leq \sum\limits_{j=1}^{N}\frac{1}{\e}\int\limits_{-\e}^{\e}\Bigg\{
\int\limits_{I^{\tau_0}}F\bigg(m_\e(x)\,,\,
%\e\big(\vec P\cdot\nabla
%\{\vec A\cdot \nabla u_{\e,|x_j-s|}\}\big)(x)\\+
\e\vec Q\cdot\nabla\psi_{\e,|x_j-s|}(x)\,,\,
%\big(\vec A\cdot\nabla u_{\e,|x_j-s|}\big)(x)+
\psi_{\e,|x_j-s|}(x)\bigg)dx\Bigg\}ds\,.
\end{multline}
Thus changing variables in \er{L2009small1hjhjjhhjhhjjioiiounew}
gives
\begin{multline}\label{L2009small1hjhjjhhjhhjjjkljkljklnew}
\int_{0}^{\tau}\,\int_{\cup_{j=1}^{N}\{ x\in
I^{t+\e}:\,|x_j|>t\}}\frac{1}{\e}
%\times\\ \times
F\bigg(m_\e(x)\,,\,
%\e\big(\vec P\cdot\nabla \{\vec A\cdot \nabla u_{\e,t}\}\big)(x)+
\e\vec Q\cdot\nabla\psi_{\e,t}(x)\,,\,
%\big(\vec A\cdot\nabla u_{\e,t}\big)(x)+
\psi_{\e,t}(x)\bigg)dx\,dt\\
\leq\sum\limits_{j=1}^{N}\int\limits_{-1}^{1}\Bigg\{
\int\limits_{I^{\tau_0}}F\bigg(m_\e(x)\,,\,
%\e\big(\vec P\cdot\nabla\{\vec A\cdot \nabla u_{\e,|x_j-\e s|}\}\big)(x)+\\
\e\vec Q\cdot\nabla\psi_{\e,|x_j-\e s|}(x)\,,\,
%\big(\vec A\cdot\nabla u_{\e,|x_j-\e s|}\big)(x)+
\psi_{\e,|x_j-\e s|}(x)\bigg)dx\Bigg\}ds\,.
\end{multline}
Next,
%since either $\vec P\equiv 0$ or $\lim_{\e\to 0^+}\nabla \bar v_\e=0$ in $L^{\bar p}_{loc}(I_{\vec \nu},\R^{d\times N})$,
clearly by \er{L2009deftvu1hhhjnew} there exists a constant $C_0>0$
such that for every $j\in\{1,
%2,
\ldots, N\}$ every $s\in[0,1/2)$ and
every $\e\in(0,1)$ we have
\begin{equation}\label{ffyfyguihihiuiolkkkkjjjkjkjknew}
\int_{I^{\tau_0}}\bigg(
%\bigg|\e\big(\vec P\cdot\nabla \{\vec A\cdot\nabla u_{\e,s}\}\big)(x)\bigg|^p+
\Big|\e\vec Q\cdot\nabla\psi_{\e,s}(x)\Big|^q+
%\bigg|\big(\vec A\cdot\nabla u_{\e,s}\big)(x)\bigg|^p+
\Big|\psi_{\e,s}(x)\Big|^{q}
%\\+\bigg|\frac{1}{\e}u_{\e,s}(x)\bigg|^{\bar p}+\|\vec P\|\cdot\Big|\nabla u_{\e,s}(x)\Big|^{\bar p}
\bigg)dx \leq C_0\int_{I^{\tau_0}}\Bigg(
%\Big|\e\big(\vec P\cdot\nabla\{\vec A\cdot\nabla \bar v_\e\}\big)(x)\Big|^p+
\Big|\e\vec Q\cdot\nabla\bar\varphi_\e(x)\Big|^q+
%\\ \Big|\big(\vec A\cdot\nabla \bar v_\e\big)(x)\Big|^p+
\big|\bar\varphi_\e(x)\big|^{q}
%+\bigg|\frac{1}{\e}\bar v_\e(x)\bigg|^{\bar p}+\|\vec P\|\cdot\Big|\nabla \bar v_\e(x)\Big|^{\bar p}
\Bigg)dx.
\end{equation}
In particular since $0<\tau<\tau_0<1/2$ were chosen arbitrary we
deduce that for every $s\in[0,1/2)$ we have
%$\lim_{\e\to 0^+}u_{\e,s}/\e=0$ in $L^{\bar p}_{loc}(I_{\vec \nu},\R^d)$,
$\,\lim_{\e\to 0^+}\psi_{\e,s}=0$ in $L^{q}_{loc}(I_{\vec
\nu},\R^m)$ and
%$\lim_{\e\to 0^+}\vec A\cdot \nabla u_{\e,s}=0$ in $L^{p}_{loc}(I_{\vec \nu},\R^k)$,
%$\lim_{\e\to 0^+}\big(\e\vec P\cdot\nabla\{\vec A\cdot\nabla u_{\e,s}\}\big)=0$ in $L^p_{loc}(I_{\vec \nu},\R^{k})$,
$\lim_{\e\to 0^+}\big(\e\vec Q\cdot\nabla\psi_{\e,s}\big)=0$  in
$L^{q}_{loc}(I_{\vec \nu},\R^d)$.
% and $\lim_{\e\to 0^+}\|\vec P\|\,\nabla u_{\e,s}=0$ in $L^{\bar p}_{loc}(I_{\vec\nu},\R^{d\times N})$.
Next, by \er{ffyfyguihihiuiolkkkkjjjkjkjknew} we deduce
\begin{multline}\label{ffyfyguihihiuiolkkknew}
\lim_{\e\to 0^+}\int\limits_{I^{\tau_0}}\bigg(
%\bigg|\e\big(\vec P\cdot\nabla \{\vec A\cdot\nabla u_{\e,|x_j-\e s|}\}\big)(x)\bigg|^p+
\Big|\e\vec Q\cdot\nabla\psi_{\e,|x_j-\e s|}(x)\Big|^q+
%\bigg|\big(\vec A\cdot\nabla u_{\e,|x_j-\e s|}\big)(x)\bigg|^p+\\
\Big|\psi_{\e,|x_j-\e s|}(x)\Big|^{q}
%+\bigg|\frac{1}{\e}u_{\e,|x_j-\e s|}(x)\bigg|^{\bar p}
\bigg)dx=0 \;\;\text{uniformly by}\;s\in(-1,1)\;\;\forall j=1,\ldots
,N.
\end{multline}
On the other hand we have $\lim_{\e\to 0^+}m_\e=m_0$ in
$L^{p}(I_{\tau_0},\R^k)$, $F\big(m_0(x),0,0\big)=0$ for a.e. $x\in
I_{\tau_0}$ and \er{RsTT}. Therefore, by \er{ffyfyguihihiuiolkkknew}
for every $j=1,
%2,
\ldots, N$ we deduce
%\begin{multline}
\begin{equation}
\label{L2009small1hjhjjhhjhhjjjkljkljklhjhjhihjnew} \lim_{\e\to
0^+}\int\limits_{-1}^{1}\Bigg\{
\int\limits_{I^{\tau_0}}F\bigg(m_\e(x)\,,\,
%\e\big(\vec P\cdot\nabla\{\vec A\cdot \nabla u_{\e,|x_j-\e s|}\}\big)(x)\\
\e\vec Q\cdot\nabla\psi_{\e,|x_j-\e s|}(x)\,,\,
%\big(\vec A\cdot\nabla u_{\e,|x_j-\e s|}\big)(x)+
\psi_{\e,|x_j-\e s|}(x)\bigg)dx\Bigg\}ds=0\,.
%\end{multline}
\end{equation}
Then plugging \er{L2009small1hjhjjhhjhhjjjkljkljklhjhjhihjnew} into
\er{L2009small1hjhjjhhjhhjjjkljkljklnew} we deduce
\er{L2009small1new}.

 Next let $\e_n\downarrow 0$ be such that
\begin{multline}\label{ggfghjjhfhfjfhjhjhjhjfjkdghfdnmbbihhbbmnew}
\lim_{n\to +\infty}\int\limits_{I_{\vec \nu}}\frac{1}{\e_n}
F\bigg(m_{\e_n}(x)\,,\,
%\e_n\big(\vec P\cdot\nabla \{\vec A\cdot \nabla v_{\e_n}\}\big)(x)+
\e_n\vec Q\cdot\nabla\varphi_{\e_n}(x)\,,\,
%\big(\vec A\cdot\nabla v_{\e_n}\big)(x)+
\varphi_{\e_n}(x)\bigg)dx=\liminf_{\e\to 0^+}\int\limits_{I_{\vec
\nu}}\frac{1}{\e} F\bigg(m_\e(x)\,,\,
%\e\big(\vec P\cdot\nabla \{\vec A\cdot \nabla v_{\e}\}\big)(x)+
\e\vec Q\cdot\nabla\varphi_{\e}(x)\,,\,
%\big(\vec A\cdot\nabla v_{\e}\big)(x)+
\varphi_{\e}(x)\bigg)dx\,.
\end{multline}
Then, since \er{L2009small1new} is valid for every $\tau\in(0,1/2)$,
we can pass to a subsequence, still denoted by $\e_n\downarrow 0$,
so that for a.e. $t\in(0,1/2)$ we will have
%\begin{multline}
\begin{equation}
\label{L2009small1hjkhhjhjnew}
\lim_{n\to+\infty}\int_{\cup_{j=1}^{N}\{ x\in
I^{t+\e_n}:\,|x_j|>t\}}\frac{1}{\e_n}
%\times\\ \times
F\bigg(m_{\e_n}(x)\,,\,
%\e_n\big(\vec P\cdot\nabla \{\vec A\cdot\nabla u_{\e_n,t}\}\big)(x)+
\e_n\vec Q\cdot\nabla\psi_{\e_n,t}(x)\,,\,
%\big(\vec A\cdot\nabla u_{\e_n,t}\big)(x)+
\psi_{\e_n,t}(x)\bigg)dx=0\,.
\end{equation}
%\end{multline}
Therefore, by \er{fughfighdfighfihklhh}, \er{L2009eqgl1new} and
\er{L2009small1hjkhhjhjnew}, for a.e. $t\in(0,1/2)$ we have
\begin{multline}\label{ggfghjjhfhfjfhjhjhjhjfjkdghfdnmbguiyuihyuyhukjnew}
\lim_{n\to +\infty}\int\limits_{I_{\vec \nu}}\frac{1}{\e_n}
%\times\\ \times
F\bigg(m_{\e_n}(x)\,,\,
%\e_n\big(\vec P\cdot\nabla \{\vec A\cdot\nabla v_{\e_n}\}\big)(x)+
\e_n\vec Q\cdot\nabla\varphi_{\e_n}(x)\,,\,
%\big(\vec A\cdot\nabla v_{\e_n}\big)(x)+
\varphi_{\e_n}(x)\bigg)dx \geq \limsup_{n\to
+\infty}\int_{I^t}\frac{1}{\e_n}
%\times\\ \times
F\bigg(m_{\e_n}(x)\,,\,
%\e_n\big(\vec P\cdot\nabla \{\vec A\cdot \nabla v_{\e_n}\}\big)(x)+
\e_n\vec Q\cdot\nabla\varphi_{\e_n}(x)\,,\,
%\big(\vec A\cdot\nabla v_{\e_n}\big)(x)+
\varphi_{\e_n}(x)\bigg)dx\\= \limsup_{n\to
+\infty}\int_{I^t}\frac{1}{\e_n}
%\times\\ \times
F\bigg(m_{\e_n}(x)\,,\,
%\e_n\big(\vec P\cdot\nabla \{\vec A\cdot\nabla u_{\e_n,t}\}\big)(x)+
\e_n\vec Q\cdot\nabla\psi_{\e_n,t}(x)\,,\,
%\big(\vec A\cdot\nabla u_{\e_n,t}\big)(x)+
\psi_{\e_n,t}(x)\bigg)dx\\= \limsup_{n\to
+\infty}\int_{I^{t+\e_n}}\frac{1}{\e_n}
%\times\\ \times
F\bigg(m_{\e_n}(x)\,,\,
%\e_n\big(\vec P\cdot\nabla \{\vec A\cdot\nabla u_{\e_n,t}\}\big)(x)+
\e_n\vec Q\cdot\nabla\psi_{\e_n,t}(x)\,,\,
%\big(\vec A\cdot\nabla u_{\e_n,t}\big)(x)+
\psi_{\e_n,t}(x)\bigg)dx\,.
\end{multline}
Thus, by \er{ggfghjjhfhfjfhjhjhjhjfjkdghfdnmbbihhbbmnew},
\er{L2009eqgl1new} and
\er{ggfghjjhfhfjfhjhjhjhjfjkdghfdnmbguiyuihyuyhukjnew} for a.e.
$t\in(0,1/2)$ we have
\begin{multline}\label{ggfghjjhfhfjfhjhjhjhjfjkdghfdnmbguiyuihyuyhukjhhhjhjjnew}
\liminf_{\e\to 0^+}\int\limits_{I_{\vec \nu}}\frac{1}{\e}
F\bigg(m_\e(x),
%\e\big(\vec P\cdot\nabla \{\vec A\cdot \nabla v_{\e}\}\big)(x)+
\e\vec Q\cdot\nabla\varphi_{\e}(x),
%\big(\vec A\cdot\nabla v_{\e}\big)(x)+
\varphi_{\e}(x)\bigg)dx\geq\limsup_{n\to
+\infty}\int\limits_{I^{t+\e_n}}\frac{1}{\e_n}
%\times\\ \times
F\bigg(m_{\e_n}(x),
%\e_n\big(\vec P\cdot\nabla \{\vec A\cdot\nabla u_{\e_n,t}\}\big)(x)+
\e_n\vec Q\cdot\nabla\psi_{\e_n,t}(x),
%\big(\vec A\cdot\nabla u_{\e_n,t}\big)(x)+
\psi_{\e_n,t}(x)\bigg)dx.
\end{multline}
Since the last inequality is valid for a.e. $t\in(0,1/2)$, by
diagonal arguments we deduce that there exists a new sequences
$t_n\uparrow (1/2)$ and $\e_n\downarrow 0$ as $n\to +\infty$ such
that $\e_n+t_n<1/2$,
\begin{equation}\label{vjhghihikhklhjkg}
\lim_{n\to+\infty}\int_{I^{t_n+\e_n}}\Big|m_{\e_n}(x)-m_0(x)\Big|^pdx=0\,,
\end{equation}
%\begin{multline}
\begin{equation}
\label{hdfghfighfigh} \lim_{n\to+\infty}\int_{I^{t_n+\e_n}}\bigg(
%\Big|u_{\e_n,t_n}/\e_n\Big|^{\bar p} +
\Big|\psi_{\e_n,t_n}\Big|^{q}+
%\Big|\vec A\cdot\nabla u_{\e_n,t_n}\Big|^p+
%\Big|\e_n\vec P\cdot\nabla\{\vec A\cdot\nabla u_{\e_n,t_n}\}\Big|^p+\\
\Big|\e_n\vec Q\cdot\nabla\psi_{\e_n,t_n}\Big|^q
%+\|\vec P\|\Big|\nabla u_{\e_n,t_n}\Big|^{\bar p}
\bigg)dx=0,
\end{equation}
%\end{multline}
%$\lim_{n\to+\infty}u_{\e_n,t_n}/\e_n=0$ in $L^{p}_{loc}(I_{\vec
%\nu},\R^d)$, $\,\lim_{n\to+\infty}\vec B\cdot\psi_{\e_n,t_n}=0$ in
%$L^{p}_{loc}(I_{\vec \nu},\R^k)$, $\lim_{n\to+\infty}\vec A\cdot
%\nabla u_{\e_n,t_n}=0$ in $L^{p}_{loc}(I_{\vec \nu},\R^k)$,
%$\lim_{n\to+\infty}\big(\e_n\vec P\cdot\nabla\{\vec A\cdot\nabla
%u_{\e_n,t_n}\}\big)=0$ in $L^p_{loc}(I_{\vec \nu},\R^{k})$,
%$\lim_{n\to+\infty}\big(\e_n\vec Q\cdot\nabla\{\vec B\cdot
%\psi_{\e_n,t_n}\}\big)=0$, $\lim_{n\to+\infty}\|\vec P\|\,\nabla
%u_{\e_n,t_n}=0$ in $L^{p}_{loc}(I_{\vec \nu},\R^{d\times N})$
and
\begin{multline}\label{ggfghjjhfhfjfhjhjhjhjfjkdghfdnmbguiyuihyuyhukjnewgffhjhj}
\liminf_{\e\to 0^+}\int\limits_{I_{\vec \nu}}\frac{1}{\e}
F\bigg(m_\e(x),
%\e\big(\vec P\cdot\nabla \{\vec A\cdot \nabla v_{\e}\}\big)(x)+
\e\vec Q\cdot\nabla\varphi_{\e}(x),
%\big(\vec A\cdot\nabla v_{\e}\big)(x)
\varphi_{\e}(x)\bigg)dx \geq \lim_{n\to
+\infty}\int\limits_{I^{t_n+\e_n}}\frac{1}{\e_n}
%\times\\ \times
F\bigg(m_{\e_n}(x),
%\e_n\big(\vec P\cdot\nabla \{\vec A\cdot\nabla u_{\e_n,t_n}\}\big)(x)+
\e_n\vec Q\cdot\nabla\psi_{\e_n,t_n}(x),
%\big(\vec A\cdot\nabla u_{\e_n,t_n}\big)(x)+
\psi_{\e_n,t_n}(x)\bigg)dx.
\end{multline}
On the other hand defining
%$\bar u_{n}(x):=u_{\e_n,t_n}\big((2t_n+2\e_n)x\big)/(2t_n+2\e_n)$ and
$\bar\psi_{n}(x):=\psi_{\e_n,t_n}\big((2t_n+2\e_n)x\big)$ we clearly
have $\psi_{\e}\in C^\infty_c(I^1,\R^m)$.
%, $u_{\e}\in C^\infty_c(I^1,\R^d)$.
Moreover, changing variables of integration
$z=x/(2t_n+2\e_n)$ in
\er{ggfghjjhfhfjfhjhjhjhjfjkdghfdnmbguiyuihyuyhukjnewgffhjhj} we
finally deduce,
\begin{multline}\label{ggfghjjhfhfjfhjhjhjhjfjkdghfdnmbguiyuihyuyhukjnewgffhjhjglgjklg}
\liminf_{\e\to 0^+}\int\limits_{I_{\vec \nu}}\frac{1}{\e}
F\bigg(m_\e(x)\,,\,
%\e\big(\vec P\cdot\nabla \{\vec A\cdot \nabla v_{\e}\}\big)(x)+
\e\vec Q\cdot\nabla\varphi_{\e}(x)\,,\,
%\big(\vec A\cdot\nabla v_{\e}\big)(x)+
\vec B\cdot\varphi_{\e}(x)\bigg)dx\geq\\
\lim_{n\to +\infty}\int\limits_{I^{1}}\frac{2(t_n+\e_n)}{\e_n}
F\bigg(m_{\e_n}\big(2(t_n+\e_n)z\big)\,,\,
%\\
%\frac{\e_n}{2(t_n+\e_n)}\big(\vec P\cdot\nabla \{\vec A\cdot \nabla\bar u_{n}\}\big)\big(z\big)+
\frac{\e_n}{2(t_n+\e_n)}\vec Q\cdot\nabla\bar\psi_{n}(z)\,,\,
%\big(\vec A\cdot\nabla \bar u_{n}\big)\big(z\big)+
\bar\psi_{n}(z)\bigg)dz,
\end{multline}
and \er{fvyjhfyffhjfghgjkghfff}
%{fughfighdfighfihklhh}
follows. Finally, since $m_0\in L^{p}(I_{\vec \nu},\R^k)$,  by
\er{vjhghihikhklhjkg} we deduce $\,\lim_{n\to
+\infty}m_{(r_{\e_n}\e_n)}\big(r_{\e_n} x\big)=m_0(x)$ in
$L^{p}(I_{\vec \nu},\R^k)$, and by \er{hdfghfighfigh} we obtain
%$\lim_{n\to +\infty}\bar u_{n}/{\e_n}=0$ in $L^{\bar p}(I_{\vec\nu},\R^d)$,
$\,\lim_{n\to +\infty}\bar\psi_{n}=0$ in $L^{q}(I_{\vec \nu},\R^m)$
and
%$\lim_{n\to +\infty}\vec A\cdot\nabla \bar u_{n}=0$ in $L^{p}(I_{\vec \nu},\R^k)$,
%$\lim_{n\to+\infty}\big(\e_n\vec P\cdot\nabla\{\vec A\cdot\nabla \bar
%u_{n}\}\big)=0$ in $L^p(I_{\vec \nu},\R^{k})$,
$\lim_{n\to +\infty}\big(\e_n\vec Q\cdot\nabla \bar\psi_{n}\big)=0$
in $L^q(I_{\vec \nu},\R^{d})$
%and $\lim_{n\to+\infty}\|\vec P\|\,\nabla \bar u_{n}=0$ in $L^{\bar p}(I_{\vec\nu},\R^{d\times N})$.
This completes the proof.
\end{proof}

By the same method we can prove the following more general result.
%%%%%%%%%%%%%%%%%%%%%%%%%%%%%%%%%%%%%%%%%%%%%%%%%%%%%%%%%%%%%%%%%%%%%%%%%%%%%%%%%%%%%%%%%%%%%%%%%%%%%%%%%%%%%%%%%%%%%%%%%%%%%%%%%%%%%%%%%%%%%%%%%%%%%%%%%%%%%%%%%%%%%%%%%%%%%%%%%%%%%%%%%%%%%%%%%%%%%%%%%%%%%%%%%%%%%%%%%%%%%%%%%%%%%%%%%%%%%%%%%%%%%%%%%%%%%%%%%%%%%%%%%%%%%%%%%%%%%%%%%%%%%%%%%%%%%%%%%%%%%%%%%%%%%%%%%%%%%%%%%%%%%%%%%%%%%%%%%%%%%%%%%%%%%%%%%%%%%%%%%%%%%%%%%%
\begin{lemma}\label{L2009.02newgen}
%and $\vec R_1\in\mathcal{L}(\R^{k\times
%N},\R^k)$,...,\\ $\vec R_{n_3}\in\mathcal{L}(\R^{k\times N^{n_3}},\R^k)$
%for some $n_1,n_2,n_3\in \mathbb{N}$
Let $n\geq 1$ be a natural number and
%$\Omega$ be a bounded domain
%order %(in particular a bounded $BVG$-domain)
%$F\in C^1(\R^{k\times N\times N}\times\R^{m\times
%N}\times\R^{k\times N}\times\R^m\,,\R)$
let $F\in C^0(\R^{k}\times\R^d\times\R^{m\times
N^{n-1}}\times\ldots\times \R^{m\times N}\times\R^m,\R)$ be such
that $F\geq 0$ and there exist $C>0$, $q\geq 1$ and
$p=(p_1,p_2,\ldots, p_k)\in\R^k$ such that $p_j\geq 1$ for every $j$
and
\begin{multline*}
%\label{RsTTkgkgk}
0\leq F\Big(a,b,c_1,\ldots,c_{n-1},d\Big)\leq
C\Big(|a|^{p}+|b|^q+\sum_{j=1}^{n-1}|c_j|^q+|d|^q+1\Big)\\ \forall\,
(a,b,c_1,\ldots,c_{n-1},d)\in\R^{k}\times\R^d\times\R^{m\times
N^{n-1}}\times\ldots\times \R^{m\times N}\times\R^m
\end{multline*}
(see Definition \ref{gdhgvdgjkdfgjkhdd}).
%$0\leq F(z)\leq C\big(|z|^{p}+1\big)$ for every $z\in\R^k$.
Next let
%$\vec A\in\mathcal{L}(\R^{d\times N},\R^k)$,
%$\vec B\in\mathcal{L}(\R^{m},\R^k)$,
%$\vec D\in\mathcal{L}(\R^{r},\R^k)$,
%$\vec P_1\in\mathcal{L}(\R^{k\times N},\R^k)$,...,$\vec P_{n_1}\in\mathcal{L}(\R^{k\times N^{n_1}},\R^k)$,
%$\vec Q_1\in\mathcal{L}(\R^{k\times N},\R^k)$,...,
$\vec Q_{n}\in\mathcal{L}(\R^{m\times N^{n}},\R^d)$ be a linear
operator, $\vec\nu\in S^{N-1}$ and let
$\big\{m_\e(x)\big\}_{0<\e<1}\subset L^{p}_{loc}(I_{\vec \nu},\R^k)$
and $m_0(x)\in L^{p}(I_{\vec \nu},\R^k)$ be such that
$F\big(m_0(x),0,0,\ldots,0\big)= 0$ for a.e. $x\in I_{\vec \nu}$ and
$\,\lim_{\e\to 0^+}m_\e=m_0$ in $L^{p}_{loc}(I_{\vec \nu},\R^k)$,
where, as before, $I_{\vec \nu}:=\{y\in\R^N:\;|y\cdot
\vec\nu_j|<1/2\;\;\;\forall j=1,
%2,
\ldots, N\}$ and
$\{\vec\nu_1,
%\vec\nu_2,
\ldots,\vec\nu_N\}\subset\R^N$ is an
orthonormal base in $\R^N$, such that $\vec\nu_1:=\vec \nu$.
%assume $p\geq\max{(p_1,p_2)}$ and
Furthermore, let
%$\bar p:=\max\limits_{1\leq j\leq k}p_j$ and
%$\big\{v_\e(x)\big\}_{0<\e<1}\subset L^{\bar p}_{loc}(I_{\vec\nu},\R^d)$,
$\big\{\varphi_\e(x)\big\}_{0<\e<1}\subset L^{q}_{loc}(I_{\vec
\nu},\R^m)$
%and $\big\{\gamma_\e(t)\big\}_{0<\e<1}\subset L^{\bar p}_{loc}\big((-1/2,1/2),\R^r\big)$
be such that
%$\lim_{\e\to 0^+}v_\e/\e=0$ in $L^{\bar p}_{loc}(I_{\vec \nu},\R^d)$,
$\,\lim_{\e\to 0^+}\varphi_\e=0$ in $L^{q}_{loc}(I_{\vec
\nu},\R^m)$,
%$\lim_{\e\to 0^+}\vec A\cdot \nabla v_\e=0$ in $L^{p}_{loc}(I_{\vec\nu},\R^k)$,
%$\lim_{\e\to 0^+}\vec D\cdot\gamma_\e(\vec\nu\cdot x)=0$ in $L^{p}_{loc}(I_{\vec\nu},\R^k)$,
%$\lim_{\e\to 0^+}\big(\e^{n_1}\vec P_{n_1}\cdot\nabla^{n_1}\{\vec A\cdot\nabla v_\e\}\big)=0$ in
%$L^p_{loc}(I_{\vec \nu},\R^{k})$,
$\lim_{\e\to 0^+}\big(\e^{n}\vec Q_{n}\cdot\nabla^{n}
\varphi_\e\big)=0$ in $L^q_{loc}(I_{\vec \nu},\R^{d})$,
%$\lim_{\e\to 0^+}\big(\e^{n_3}\vec R_{n_2}\cdot\nabla^{n_3}\{\vec D\cdot \gamma_\e(\vec\nu\cdot
%x)\}\big)=0$ in $L^p_{loc}(I_{\vec \nu},\R^{k})$,
and we have
%$\lim_{\e\to 0^+}\big(\sum_{s=j}^{n_1}\|\vec P_s\|\big)\big(\e^{j-1}\nabla^{j}v_\e\big)=0$ in $L^{\bar p
%}_{loc}(I_{\vec \nu},\R^{d\times N^{j}})$, for every $j=0,1,\ldots,(n_2-1)$ we have
$\,\lim_{\e\to 0^+}\big(\e^{j}\nabla^{j}\varphi_\e\big)=0$ in
$L^{q}_{loc}(I_{\vec \nu},\R^{m\times N^{j}})$ for every
$j\in\{1,\ldots,n-1\}$.
Then there exist $\big\{r_\e\big\}_{0<\e<1}\subset(0,1)$ and
%$\big\{u_\e(x)\big\}_{0<\e<1}\subset C^{\infty}_{c}(I_{\vec\nu},\R^d)$,
$\big\{\psi_\e(x)\big\}_{0<\e<1}\subset C^{\infty}_{c}(I_{\vec
\nu},\R^m)$
%and $\big\{\lambda_\e(t)\big\}_{0<\e<1}\subset C^\infty_c\big((-1/2,1/2),\R^r\big)$
such that $\lim_{\e\to
0^+}r_\e=1$,
%$\lim_{\e\to 0^+}u_\e/\e=0$ in $L^{\bar p}(I_{\vec\nu},\R^d)$,
$\,\lim_{\e\to 0^+}\psi_\e=0$ in $L^{q}(I_{\vec \nu},\R^m)$,
%$\lim_{\e\to 0^+}\vec A\cdot \nabla u_\e=0$ in $L^{p}(I_{\vec \nu},\R^k)$, $\lim_{\e\to 0^+}\vec
%D\cdot\lambda_\e(\vec\nu\cdot x)=0$ in $L^{p}(I_{\vec \nu},\R^k)$,
%$\lim_{\e\to 0^+}\big(\e^n\vec P_{n_1}\cdot\nabla^{n_1}\{\vec
%A\cdot\nabla u_\e\}\big)=0$ in $L^p(I_{\vec \nu},\R^{k})$,
$\lim_{\e\to 0^+}\big(\e^{n}\vec Q_{n}\cdot\nabla^{n}
\psi_\e\big)=0$ in $L^q(I_{\vec \nu},\R^{d})$,
%$\lim_{\e\to 0^+}\big(\e^{n_3}\vec R_{n_2}\cdot\nabla^{n_3}\{\vec D\cdot
%\lambda_\e(\vec\nu\cdot x)\}\big)=0$ in $L^p(I_{\vec \nu},\R^{k})$,
%for every $j=1,2,\ldots,n_1$ we have $\lim_{\e\to
%0^+}\big(\sum_{s=j}^{n_1}\|\vec
%P_s\|\big)\big(\e^{j-1}\nabla^{j}u_\e\big)=0$ in $L^{\bar p}(I_{\vec\nu},\R^{d\times N^{j}})$,
for every $j=1,\ldots,(n-1)$ we have $\lim_{\e\to
0^+}\big(\e^{j}\nabla^{j}\psi_\e\big)=0$ in $L^{q}(I_{\vec
\nu},\R^{m\times N^{j}})$
%, for every $j=0,1,\ldots,(n_3-1)$ we have $\lim_{\e\to
%0^+}\big(\sum_{s=j+1}^{n_3}\|\vec
%R_s\|\big)\big(\e^{j}\vec\nabla^{j}\{\vec D\cdot
%\lambda_\e(\vec\nu\cdot x)\}\big)=0$ in $L^{\bar p}(I_{\vec\nu},\R^{k\times N^{j}})$,
$\,\lim_{\e\to 0^+}m_{(r_\e\e)}\big(r_\e
x\big)=m_0(x)$ in $L^{p}(I_{\vec \nu},\R^k)$
%%%%$\lim_{\e\to 0^+}\big(\sum_{j=1}^{n}\|\vec P_j\|\big)\nabla u_\e=0$
%%%%in $L^{p}(I_{\vec \nu},\R^{d\times N})$, $\lim_{\e\to
%%%%0^+}\big(\sum_{j=2}^{n}\|\vec P_j\|\big)\big(\e\nabla\{\vec
%%%%A\cdot\nabla u_\e\}\big)=0$ in $L^p(I_{\vec \nu},\R^{k\times N})$,
%%%%$\lim_{\e\to 0^+}\big(\sum_{j=3}^{n}\|\vec
%%%%P_j\|\big)\big(\e^2\nabla^2\{\vec A\cdot\nabla u_\e\}\big)=0$ in
%%%%$L^p(I_{\vec \nu},\R^{k\times N^2})$,..., $\lim_{\e\to 0^+}\|\vec
%%%%P_n\|\big(\e^{n-1}\nabla^{n-1}\{\vec A\cdot\nabla u_\e\}\big)=0$ in
%%%%$L^p(I_{\vec \nu},\R^{k\times N^{n-1}})$, $\lim_{\e\to
%%%%0^+}\big(\sum_{j=2}^{n}\|\vec Q_j\|\big)\big(\e\nabla\{\vec B\cdot
%%%%\psi_\e\}\big)=0$ in $L^p(I_{\vec \nu},\R^{k\times N})$,\\
%%%%$\lim_{\e\to 0^+}\big(\sum_{j=3}^{n}\|\vec
%%%%Q_j\|\big)\big(\e^2\nabla^2\{\vec B\cdot \psi_\e\}\big)=0$ in
%%%%$L^p(I_{\vec \nu},\R^{k\times N^2})$,..., $\lim_{\e\to 0^+}\|\vec
%%%%Q_n\|\big(\e^{n-1}\vec\nabla^{n-1}\{\vec B\cdot \psi_\e\}\big)=0$ in
%%%%$L^p(I_{\vec \nu},\R^{k\times N^{n-1}})$
%$\lim_{\e\to 0^+}\big(\e\vec P\cdot\nabla\{\vec A\cdot\nabla
%u_\e\}\big)=0$ in $L^p_{loc}(I_{\vec \nu},\R^{k})$ and $\lim_{\e\to
%0^+}\big(\e\vec Q\cdot\nabla\{\vec B\cdot \psi_\e\}\big)=0$ in
%$L^p_{loc}(I_{\vec \nu},\R^{k})$
and
\begin{multline}\label{fvyjhfyffhjfghgjkghfffgen}
\liminf_{\e\to 0^+}\int\limits_{I_{\vec \nu}}\frac{1}{\e}
F\Bigg(m_\e(x)\,,\,\e^n\vec Q_n\cdot\nabla^n\varphi_{\e}(x)\,,\,
%\sum_{j=1}^{n_1}\e^j\big(\vec P_j\cdot\nabla^j \{\vec A\cdot \nabla v_{\e}\}\big)(x)+
\e^{n-1}\nabla^{n-1}\varphi_{\e}(x)\,,\,\ldots\,,\,
\e\nabla\varphi_{\e}(x)\,,\,
%\\+\sum_{j=1}^{n_3}\e^j\big(\vec
%R_j\cdot\nabla^j\{\vec D\cdot\gamma_{\e}\}\big)(\vec\nu\cdot x)+\big(\vec A\cdot\nabla v_{\e}\big)(x)
\varphi_{\e}(x)
%+\vec D\cdot\gamma_\e(\vec\nu\cdot x)
\Bigg)dx\geq\\
 \liminf_{\e\to 0^+}\int\limits_{I_{\vec
\nu}}\frac{1}{\e}F\Bigg(m_{r_\e\e}\big(r_\e x\big)\,,\,\e^n\vec
Q_n\cdot\nabla^n\psi_{\e}(x)\,,\,
%\sum_{j=1}^{n_1}\e^j\big(\vec P_j\cdot\nabla^j \{\vec A\cdot\nabla u_{\e}\}\big)(x)+
\e^{n-1}\nabla^{n-1}\psi_{\e}(x)\,,\,\ldots\,,\,
\e\nabla\psi_{\e}(x)\,,\,
%\\+\sum_{j=1}^{n_3}\e^j\big(\vec R_j\cdot\nabla^j\{\vec D\cdot\lambda_{\e}\}\big)(\vec\nu\cdot
%x)+\big(\vec A\cdot\nabla u_{\e}\big)(x)
\psi_{\e}(x)
%+\vec D\cdot\lambda_\e(\vec\nu\cdot x)
\Bigg)dx\,.
\end{multline}
\end{lemma}

As a consequence of Lemma \ref{L2009.02newgen} we have the following
Proposition.
\begin{proposition}\label{L2009.02kkknew}
Let $n_1,n_2\in \mathbb{N}$. Consider the linear operators
% $\vec A\in\mathcal{L}(\R^{k\times N},\R^d)$,
$\vec B\in\mathcal{L}(\R^{d\times N},\R^m)$
%$n\in\mathbb{N}$,
%$\vec R\in\mathcal{L}(\R^{m\times N^{n_1}},\R^{r})$,
%for all $j=1,2,\ldots n_1$,
%$\vec P_j\in\mathcal{L}(\R^{d\times N^j},\R^{d_j})$ for all $j=1,2,\ldots n_2$
and $\vec Q\in\mathcal{L}(\R^{k\times N^{n_2}},\R^{l})$, and
%for all $j=1,2,\ldots n_3$, and
%$\Omega$ be a bounded domain
%order %(in particular a bounded $BVG$-domain)
%$F\in C^1(\R^{k\times N\times N}\times\R^{m\times
%N}\times\R^{k\times N}\times\R^m\,,\R)$
let $F\in C^0\big(\{\R^{m\times
N^{n_1}}\times\ldots\times\R^{m\times N}\times\R^{m}\}\times
%\{\R^{d_{n_2}}\times\ldots\times\times\R^{d_{2}}\times\R^{d_1}\times\R^{d}\}\times
\{\R^{l}\times\R^{k\times N^{n_2-1}}\times\ldots\times\R^{k\times
N}\times\R^k\}\,,\R\big)$ be such that $F\geq 0$ and there exist
$C>0$ and $p_1\geq 1$, $p_2\geq 1$ satisfying
%$0\leq F(a,b,c,d)-F(0,0,c,d)\leq C|a|^{p_1}+C|b|^{p_1}$ and
%$F(0,0,c,d)\leq C (|c|^{p_2}+|d|^{p_2}+1)$ for every $(a,b,c,d)$.
\begin{multline*}
0\leq F\Big(\{a_1,\ldots, a_{n_1+1}\},\{b_1,\ldots,b_{n_2+1}\}
%,\ldots,\{c_1,c_2,\ldots,c_{n_3},c\}
\Big)\leq
C\bigg(\sum_{j=1}^{n_1+1}|a_j|^{p_1}+\sum_{j=1}^{n_2+1}|b_j|^{p_2}
%\sum_{j=1}^{n_3}|c_j|^{p_3}+
%|a|^{p_1}+
%|b|^{p_2}
%+|c|^{p_3}
+1\bigg)\\ \text{for every}\;\; \Big(\{a_1,\ldots,
a_{n_1+1}\},\{b_1,\ldots,b_{n_2+1}\}
%,\ldots,\{c_1,c_2,\ldots,c_{n_3},c\}
\Big)
%\Big(\{a_1,b_1,c_1\},\{a_2,b_2,c_2\},\ldots,\{a_n,b_n,c_n\},a,b,c\Big)
.
\end{multline*}
Furthermore, let
%$\vec k\in\R^k$,
$\vec\nu\in S^{N-1}$, $\varphi^+,\varphi^-\in\R^k$
%, $V^+,V^-\in\R^{d}$
and $W^+,W^-\in \R^{m}$ be such that if we set
\begin{equation}\label{vfyguiguhikjnklklhkukuytou}
\varphi_0(x): =\begin{cases}
\varphi^+\;\;\text{if}\;\;x\cdot\vec\nu>0,\\
\varphi^-\;\;\text{if}\;\;x\cdot\vec\nu<0,
\end{cases}
%\; V_0(x): =\begin{cases}
%V^+\;\;\text{if}\;\;x\cdot\vec\nu>0,\\
%V^-\;\;\text{if}\;\;x\cdot\vec\nu<0,
%\end{cases}
\;\text{and}\quad\; W_0(x): =\begin{cases}
W^+\;\;\text{if}\;\;x\cdot\vec\nu>0,\\
W^-\;\;\text{if}\;\;x\cdot\vec\nu<0,
\end{cases}
\end{equation}
%there exists $\vec k\in\R^k$ which satisfy
then
%\begin{equation}\label{vfyguiguhikjnklklhbjkbbjk}
%F\big(0,0,\ldots,0,W^+,V^+,\varphi^+\big)=F\big(0,0,\ldots,0,W^-,V^-,\varphi^-\big)=0\,,
%\end{equation}
\begin{equation}\label{vfyguiguhikjnklklhbjkbbjk}
F\Big(\big\{0,0,\ldots,W_0(x)\big\},
%\big\{0,0,\ldots,V_0(x)\big\},
\big\{0,0,\ldots,\varphi_0(x)\big\}\Big)=0\quad\text{for
a.e.}\;x\in\R^N\,.
\end{equation}
%for some $w\in\mathcal{D}'(\R^N,\R^q)$.
%Set $\varphi(x)\in L^\infty(\R^N,\R^m)$ and $v(x):Lip(\R^N,\R^k)$ by
%\begin{equation}\label{ghgghjhjkdfhgkkknew}
%\varphi(x):=\begin{cases}
%\varphi^+\quad\text{if}\;\;x\cdot\vec\nu>0\,,\\
%\varphi^-\quad\text{if}\;\;x\cdot\vec\nu<0\,,
%\end{cases}\quad\quad
%v(x):=\begin{cases}
%V^-\cdot x+(x\cdot\vec\nu)\vec k\quad\text{if}\;\;x\cdot\vec\nu\geq 0\,,\\
%V^-\cdot x\quad\quad\quad\text{if}\;\;x\cdot\vec\nu<0\,.
%\end{cases}
%\end{equation}
Next
%assume $p\geq\max{(p_1,p_2)}$ and
let $\big\{w_\e(x)\big\}_{0<\e<1}\subset L^{p_1}_{loc}(I_{\vec
\nu},\R^d)$
%W^{(n+1),p}
%$\big\{v_\e(x)\big\}_{0<\e<1}\subset L^{p_2}_{loc}(I_{\vec \nu},\R^k)$
and
$\big\{\varphi_\e(x)\big\}_{0<\e<1}\subset
%W^{n,p}
W^{n_2,p_2}_{loc}(I_{\vec \nu},\R^k)$ be such that $\{\vec B\cdot
\nabla w_\e\}\in W^{n_1,p_1}_{loc}(I_{\vec \nu},\R^m)$, $\lim_{\e\to
0^+}\varphi_\e=\varphi_0$ in $L^{p_2}_{loc}(I_{\vec \nu},\R^k)$,
$\lim_{\e\to 0^+}\{\vec B\cdot \nabla w_\e\}=W_0$ in
$L^{p_1}_{loc}(I_{\vec \nu},\R^m)$,
%$\lim_{\e\to 0^+}\{\vec A\cdot\nabla v_\e\}=V_0$ in $L^{p_2}_{loc}(I_{\vec \nu},\R^d)$,
%$\lim_{\e\to 0^+}\big(\e^{n_1}\,\vec R\cdot\nabla^{n_1}\big\{\vec
%B\cdot\nabla w_\e\big\}\big)=0$ in $L^{p_1}_{loc}(I_{\vec\nu},\R^{r})$,
$\lim_{\e\to 0^+}\big(\e^{n_2}\,\vec Q\cdot\{\nabla^{n_2}
\varphi_\e\}\big)=0$ in $L^{p_2}_{loc}(I_{\vec \nu},\R^{l})$
%for every $j=1,2,\ldots, (n_2-1)$ we have $\lim_{\e\to
%0^+}\big(\e^{j}\nabla^{j}\big\{\vec A\cdot\nabla v_\e\big\}\big)=0$
%in $L^{p_2}_{loc}(I_{\vec \nu},\R^{d_{j}})$,
%for every $j=1,2,\ldots, (n_1-1)$ we have $\lim_{\e\to
%0^+}\big(\e^j\nabla^{j}\big\{\vec B\cdot\nabla w_\e\big\}\big)=0$ in
%$L^{p_1}_{loc}(I_{\vec \nu},\R^{m\times N^j})$
%for
%every $j=2,\ldots, n_2$ we have $\lim_{\e\to
%0^+}\big(\sum_{s=j}^{n_2}\|\vec
%P_s\|\big)\big(\e^{j-1}\,\nabla^{j}v_\e\big)=0$ in $L^{\bar
%p}_{loc}(I_{\vec \nu},\R^{k\times N^{j}})$,
and for every $j=1,2,\ldots, (n_2-1)$ we have
%$\lim_{\e\to 0^+}\big(\sum_{s=j+1}^{n}\|\vec
%R_s\|\big)\big(\e^j\,\nabla^{j} \{\vec B\cdot\nabla w_\e\}\big)=0$
%in $L^p_{loc}(I_{\vec \nu},\R^{r\times N^{j}})$,
$\lim_{\e\to 0^+}\big(\e^j\,\nabla^j \varphi_\e\big)=0$ in $L^{
p_2}_{loc}(I_{\vec \nu},\R^{k\times N^j})$.
%, where $\bar p:=\max\{p_1,p_2,p_3\}$
%and
%and $\lim_{\e\to 0^+}\big(\sum_{s=1}^{n}\|\vec R_s\|\big)\nabla w_\e=0$
%in $L^p_{loc}(I_{\vec \nu},\R^{q\times N})$,
%$\lim_{\e\to
%0^+}\big(\sum_{s=1}^{n}\|\vec P_s\|\big)\nabla v_\e=0$ in
%$L^p_{loc}(I_{\vec \nu},\R^{k\times N})$.
%and $\lim_{\e\to 0^+}\big\{(v_\e-v)/\e\big\}=0$ in $L^{p}_{loc}(I_{\vec\nu},\R^k)$,
Here, as before, $I_{\vec \nu}:=\{y\in\R^N:\;|y\cdot
\vec\nu_j|<1/2\;\;\;\forall j=1,\ldots, N\}$ where
$\{\vec\nu_1,\ldots,\vec\nu_N\}\subset\R^N$ is an orthonormal base
in $\R^N$ such that $\vec\nu_1:=\vec \nu$.
Then, there exist $\big\{\psi_\e(x)\big\}_{0<\e<1}\subset
\mathcal{S}^{(n_2)}\big(\varphi^+,\varphi^-,I_{\vec\nu}\big)$
%, $\big\{u_\e(x)\big\}_{0<\e<1}\subset\mathcal{S}^{(n_2)}_1\big(V^+,V^-,I_{\vec\nu}\big)$
and
$\big\{f_\e(x)\big\}_{0<\e<1}\subset
%W^{(n+1),p}
L^{p_1}(I_{\vec \nu},\R^d)$,
%$\big\{v_\e(x)\big\}_{0<\e<1}\subset
%W^{(n+1),p}
%L^p_{loc}(I_{\vec \nu},\R^k)$
where
%\begin{multline}\label{L2009Ddef2hhhjjjj77788hhhkkkkllkjjjjkkkhhhhffggdddkkkgjhikhhhjjfgnhfkkknew}
%\mathcal{S}^{(n)}_1\big(V^+,V^-,I_{\vec\nu}\big):=
%%\mathcal{D}\Big(\vec A,(\vec A\cdot\nabla v)^-(x),(\vec A\cdot\nabla v)^+(x),\vec\nu(x),\vec k_2(x),\vec k_3(x),\ldots,\vec k_{N}(x)\Big):=\\
%\bigg\{\xi\in C^{n+1}(\R^N,\R^k):\;\;\vec A\cdot\nabla \xi(y)=V^-\;\text{ if }\;y\cdot\vec\nu\leq-1/2,\\
%\vec A\cdot\nabla \xi(y)=V^+\;\text{ if }\; y\cdot\vec\nu\geq
%1/2\;\text{ and }\;\vec A\cdot\nabla \xi\big(y+\vec\nu_j\big)=\vec
%A\cdot\nabla \xi(y)\;\;\forall j=2,3,\ldots, N\bigg\}\,,
%\end{multline}
%and
\begin{multline}\label{L2009Ddef2hhhjjjj77788hhhkkkkllkjjjjkkkhhhhffggdddkkkgjhikhhhjjddddhdkgkkknew}
\mathcal{S}^{(n)}\big(\varphi^+,\varphi^-,I_{\vec\nu}\big):=
%\mathcal{D}\Big(\vec A,(\vec A\cdot\nabla v)^-(x),(\vec A\cdot\nabla v)^+(x),\vec\nu(x),\vec k_2(x),\vec k_3(x),\ldots,\vec k_{N}(x)\Big):=\\
\bigg\{\zeta\in
C^n(\R^N,\R^k):\;\zeta(y)=\varphi_0(y)\;\,\text{if}\;\,|y\cdot\vec\nu|\geq
1/2,\;\,\text{and}\;\,\zeta\big(y+\vec \nu_j\big)=\zeta(y)\;\forall
j=2,\ldots, N\bigg\},
\end{multline}
such that $\{\vec B\cdot \nabla f_\e\}\in W^{n_1,p_1}(I_{\vec
\nu},\R^m)$, $\,\lim_{\e\to 0^+}\psi_\e=\varphi_0$ in
$L^{p_2}(I_{\vec \nu},\R^k)$,
%$\lim_{\e\to 0^+}\{\vec A\cdot \nabla u_\e\}=V_0$ in $L^{p_2}(I_{\vec \nu},\R^d)$,
$\lim_{\e\to 0^+}\{\vec B\cdot \nabla f_\e\}=W_0$ in
$L^{p_1}(I_{\vec \nu},\R^m)$,
%$\lim_{\e\to 0^+}\big(\e^{n_1}\,\vec
%R\cdot\nabla^{n_1}\big\{\vec B\cdot\nabla f_\e\big\}\big)=0$ in
%$L^{p_1}_{loc}(I_{\vec \nu},\R^{r})$,
$\lim_{\e\to 0^+}\big(\e^{n_2}\,\vec Q\cdot\{\nabla^{n_2}
\psi_\e\}\big)=0$ in $L^{p_2}(I_{\vec \nu},\R^{l})$,
%for every $j=1,2,\ldots, (n_1-1)$ we have $\lim_{\e\to
%0^+}\big(\e^{j}\nabla^{j}\big\{\vec B\cdot\nabla f_\e\big\}\big)=0$
%in $L^{p_1}(I_{\vec \nu},\R^{m\times N^j})$,
%$\lim_{\e\to 0^+}\big(\sum_{s=1}^{n}\|\vec P_s\|\big)\nabla u_\e=0$ in
%$L^p_{loc}(I_{\vec \nu},\R^{k\times N})$,
%for every $j=2,\ldots, n_2$ we have $\lim_{\e\to
%0^+}\big(\sum_{s=j}^{n_2}\|\vec P_s\|\big)\big(\e^{j-1}\,\nabla^{j}
%u_\e\big)=0$ in $L^{\bar p}_{loc}(I_{\vec \nu},\R^{k\times N^{j}})$,
for every $j=1,
%2,
\ldots, (n_2-1)$ we have $\lim_{\e\to
0^+}\big(\e^j\,\nabla^j \psi_\e\big)=0$ in $L^{p_2}(I_{\vec
\nu},\R^{k\times N^j})$,
%for every $j=1,2,\ldots, (n_1-1)$ we have $\lim_{\e\to
%0^+}\big(\e^j\,\vec R_j\cdot\nabla^{j}\big\{\vec B\cdot\nabla
%f_\e\big\}\big)=0$ in $L^{\bar p}_{loc}(I_{\vec \nu},\R^{r_j})$,
%%%%%
% and $\lim_{\e\to 0^+}\big(\sum_{s=j+1}^{n}\|\vec
%R_s\|\big)\big(\e^j\,\nabla^{j} \{\vec B\cdot\nabla f_\e\}\big)=0$
%in $L^p_{loc}(I_{\vec \nu},\R^{r\times N^{j}})$,
and
\begin{multline}\label{ggfghjjhfhfjfhjkkknew}
\liminf_{\e\to 0^+}\int_{I_{\vec \nu}}\frac{1}{\e}
F\Bigg(\Big\{\e^{n_1}\nabla^{n_1} \{\vec B\cdot\nabla
w_{\e}\},\ldots,\e\nabla\{\vec B\cdot\nabla w_{\e}\},\{\vec
B\cdot\nabla w_{\e}\}\Big\},
%,\,\Big\{\e^{n-1}\,\vec P_{n-1}\cdot\{\nabla^{n} v_{\e}\}\,,\e^{n-1}\,\vec Q_{n-1}\cdot\{\nabla^{n-1}\varphi_{\e}\}\Big\}
%\Big\{\e^n\vec P_n\cdot\nabla^{n} \{\vec A\cdot\nabla
%v_{\e}\},\ldots,\e\vec P_1\cdot\nabla\{\vec A\cdot\nabla
%v_{\e}\},\{\vec A\cdot\nabla v_{\e}\}\Big\},\,
\Big\{\e^{n_2}\vec Q\cdot\nabla^{n_2}
\varphi_{\e},\e^{n_2-1}\nabla^{n_2-1}
\varphi_{\e},\ldots,\varphi_{\e}\Big\}\Bigg)dx \geq\\ \liminf_{\e\to
0^+}\int_{I_{\vec \nu}}\frac{1}{\e}
F\Bigg(\Big\{\e^{n_1}\nabla^{n_1} \{\vec B\cdot\nabla
f_{\e}\},\ldots,\e\nabla\{\vec B\cdot\nabla f_{\e}\},\{\vec
B\cdot\nabla f_{\e}\}\Big\},
%,\,\Big\{\e^{n-1}\,\vec P_{n-1}\cdot\{\nabla^{n} v_{\e}\}\,,\e^{n-1}\,\vec Q_{n-1}\cdot\{\nabla^{n-1}\varphi_{\e}\}\Big\}
%\Big\{\e^n\vec P_n\cdot\nabla^{n} \{\vec A\cdot\nabla
%u_{\e}\},\ldots,\e\vec P_1\cdot\nabla\{\vec A\cdot\nabla
%u_{\e}\},\{\vec A\cdot\nabla u_{\e}\}\Big\},\,
\Big\{\e^{n_2}\vec Q\cdot\nabla^{n_2}
\psi_{\e},\e^{n_2-1}\nabla^{n_2-1}
\psi_{\e},\ldots,\psi_{\e}\Big\}\Bigg)dx.
\end{multline}
%%%%%Moreover, we have
%\begin{itemize}
%\item either
%$\lim_{\e\to 0^+}\frac{1}{\e}\Big(u_\e(x)-\e v_0(x/\e)-c_\e\Big)=0$
%in $L^{p}_{loc}(I_{\vec \nu},\R^k)$ in the case {\bf (a)},
%\item or
%%%%%%%%%%%%%%%%%%%%%%%%%%%%%%%%%%%%%%%%%%%%%%%%%%%%%%%%%
%$h_0(x)\subset C^{n+1}(\R^N,\R^k)$
%family
%%$\big\{h_\e(x)\big\}_{0<\e<1}\subset C^{n+1}(\R^N,\R^k)$,
%such that $\vec A\cdot\nabla h_0(x)\equiv z_0(\vec\nu\cdot x)$ for
%some function $z_0$ (i.e. $\vec A\cdot\nabla h_0(x)$ depends
%actually only on the first real variable in the base
%$\{\vec\nu_1,\vec\nu_2,\ldots,\vec\nu_N\}$), $\vec A\cdot\nabla
%h_0(x)=V_0(x)$ if $|\vec\nu\cdot x|>c_0$, where $0<c_0<1/2$ is a
%constant, and
%$\lim_{\e\to 0^+}u_\e=v$ in $W^{1,p}_{loc}(I_{\vec
%\nu},\R^k)$, $\lim_{\e\to 0^+}\big(\e^n\,\vec P_n\cdot\{\nabla^{n+1}
%u_\e\}\big)=0$ in $L^p_{loc}(I_{\vec \nu},\R^{d_n})$, for every
%$j=1,2,\ldots, (n-1)$ we have $\lim_{\e\to 0^+}\big(\e^j\,\nabla^{j}
%\{\vec A\cdot\nabla u_\e\}\big)=0$ in $L^p_{loc}(I_{\vec
%\nu},\R^{k\times N^{j+1}})$ and
%%%%%$\lim_{\e\to 0^+}\frac{1}{\e}\big(u_\e(x)-h_\e(x)\big)=0$ in
%%%%%$L^{p}_{loc}(I_{\vec \nu},\R^k)$.
%, in the case {\bf (b)}.
%\end{itemize}
\end{proposition}
\begin{proof}
Consider a function
$\lambda(x)\in\mathcal{S}^{(n_2)}\big(\varphi^+,\varphi^-,I_{\vec\nu}\big)$,
such that $\lambda(x)\equiv l_0(\vec\nu\cdot x)$ for some function
$l_0$ (i.e. $\lambda(x)$ depends actually only on the first real
variable in the base $\{\vec\nu_1,\ldots,\vec\nu_N\}$), and set
$\lambda_\e(x):=\lambda\big(x/\e\big)$. Then clearly
$\lambda_\e(x)\to \varphi_0(x)$ in $L^{p_2}_{loc}(\R^N,\R^k)$.
Moreover clearly $\e^j\nabla^j\lambda_\e(x)\to 0$ as $\e\to 0^+$ in
$L^{p_2}_{loc}(\R^N,\R^{k\times N^j})$ for every
$j\in\{1,\ldots,n_2\}$ and $\e^j\nabla^j\lambda_\e(x)$ is bounded in
$L^\infty$ for every $j\in\{0,1,\ldots,n_2\}$. Next define
$\theta_\e(x):=\varphi_\e(x)-\lambda_\e(x)$
%and $\eta_\e(x):=v_\e(x)-h_\e(x)-\gamma_\e(\vec\nu\cdot x)$.
Then
clearly
%$\,\lim_{\e\to 0^+}\eta_\e/\e=0$ in $L^{p_2}_{loc}(I_{\vec\nu},\R^k)$,
$\,\lim_{\e\to 0^+}\theta_\e=0$ in $L^{p_2}_{loc}(I_{\vec
\nu},\R^k)$,
%$\lim_{\e\to 0^+}\{\vec A\cdot \nabla \eta_\e\}=0$ in $L^{p_2}_{loc}(I_{\vec \nu},\R^d)$,
%$\lim_{\e\to 0^+}\big(\e^{n_2}\,\vec P_{n_2}\cdot\nabla^{n_2}\big\{\vec
%A\cdot\nabla \eta_\e\big\}\big)=0$ in $L^{p_2}_{loc}(I_{\vec\nu},\R^{d_{n_2}})$,
$\lim_{\e\to 0^+}\big(\e^{n_2}\,\vec Q\cdot\{\nabla^{n_2}
\theta_\e\}\big)=0$ in $L^{p_2}_{loc}(I_{\vec \nu},\R^{l})$,
%, for every $j=1,2,\ldots, n_2$ we have $\lim_{\e\to
%0^+}\big(\sum_{s=j}^{n_2}\|\vec
%P_s\|\big)\big(\e^{j-1}\,\nabla^{j}\eta_\e\big)=0$ in
%$L^{p_2}_{loc}(I_{\vec \nu},\R^{k\times N^{j}})$ and for every
and for every $j=1,\ldots, (n_2-1)$ we have $\lim_{\e\to
0^+}\big(\e^j\,\nabla^j \theta_\e\big)=0$ in $L^{p_2}_{loc}(I_{\vec
\nu},\R^{k\times N^j})$.
%and $\lim_{\e\to
%0^+}\big(\sum_{s=1}^{n_2}\|\vec P_s\|\big)\big(\nabla\eta_\e\big)=0$
%in $L^{p_2}_{loc}(I_{\vec \nu},\R^{k\times N})$.
%%%%%%%%%%%%
%and
%and $\lim_{\e\to 0^+}\big(\sum_{s=1}^{n}\|\vec R_s\|\big)\nabla w_\e=0$
%in $L^p_{loc}(I_{\vec \nu},\R^{q\times N})$,
%%%$\lim_{\e\to 0^+}\big(\sum_{s=1}^{n}\|\vec P_s\|\big)\nabla
%%%\eta_\e=0$ in $L^p_{loc}(I_{\vec \nu},\R^{k\times N})$.
Then by Lemma \ref{L2009.02newgen} there exist
$\big\{r_\e\big\}_{0<\e<1}\subset(0,1)$ and $\bar\theta_\e(x)\in
C^\infty_c\big(I_{\vec\nu},\R^m\big)$,
%$\bar\eta_\e(x)\in C^\infty_c\big(I_{\vec\nu},\R^k\big)$ and $\bar\delta_\e(t)\in
%C^\infty_c\big((-1/2,1/2),\R^k\big)$,
such that $\lim_{\e\to
0^+}r_\e=1$,
%$\lim_{\e\to 0^+}\bar\eta_\e/\e=0$ in $L^{p_2}(I_{\vec\nu},\R^k)$,
$\,\lim_{\e\to 0^+}\bar\theta_\e=0$ in $L^{p_2}(I_{\vec \nu},\R^k)$,
%$\lim_{\e\to 0^+}\{\vec A\cdot \nabla \bar\eta_\e\}=0$ in $L^{p_2}(I_{\vec \nu},\R^d)$,
%$\lim_{\e\to 0^+}\vec A\cdot\big\{\bar\delta_\e(\vec\nu\cdot
%x)\otimes\vec\nu\big\}=0$ in $L^{p_2}(I_{\vec \nu},\R^d)$,
$\lim_{\e\to 0^+}\big(\e^{n_2}\,\vec Q\cdot\{\nabla^{n_2}
\bar\theta_\e\}\big)=0$ in $L^{p_2}(I_{\vec \nu},\R^{l})$,
%$\lim_{\e\to 0^+}\big(\e^{n_2}\,\vec
%P_{n_2}\cdot\nabla^{n_2}\big\{\vec A\cdot\nabla
%\bar\eta_\e\big\}\big)=0$ in $L^{p_2}(I_{\vec \nu},\R^{d_{n_2}})$,
%$\lim_{\e\to 0^+}\e^{n_2}\,\vec P_{n_2}\cdot\nabla^{n_2}_x\big(\vec
%A\cdot\big\{\bar\delta_\e(\vec\nu\cdot
%x)\otimes\vec\nu\big\}\big)=0$ in $L^{p_2}(I_{\vec\nu},\R^{d_{n_2}})$,
for every $j=1,\ldots, (n_2-1)$ we have $\lim_{\e\to
0^+}\big(\e^j\,\nabla^j \bar\theta_\e\big)=0$ in $L^{p_2}(I_{\vec
\nu},\R^{k\times N^j})$,
%for every $j=1,2,\ldots, n_2$ we have $\lim_{\e\to 0^+}\big(\sum_{s=j}^{n_2}\|\vec
%P_s\|\big)\big(\e^{j-1}\,\nabla^{j}\bar\eta_\e\big)=0$ in
%$L^{p_2}(I_{\vec \nu},\R^{k\times N^{j}})$, for every $j=1,2,\ldots,
%n_2-1$ we have $\lim_{\e\to 0^+}\big(\sum_{s=j}^{n_2}\|\vec
%P_s\|\big)\e^{j}\nabla^{j}_x\big(\vec
%A\cdot\big\{\bar\delta_\e(\vec\nu\cdot
%x)\otimes\vec\nu\big\}\big)=0$ in $L^{p_2}(I_{\vec \nu},\R^{d\times N^{j}})$,
and
\begin{multline}\label{ggfghjjhfhfjfhjkkknewggkgkggghjjjh}
\liminf_{\e\to 0^+}\int_{I_{\vec \nu}}\frac{1}{\e}
F\Bigg(\Big\{\e^{n_1}\nabla^{n_1} \{\vec B\cdot\nabla
w_{\e}\},\ldots,\e\nabla\{\vec B\cdot\nabla w_{\e}\},\{\vec
B\cdot\nabla w_{\e}\}\Big\},
%,\,\Big\{\e^{n-1}\,\vec P_{n-1}\cdot\{\nabla^{n} v_{\e}\}\,,\e^{n-1}\,\vec Q_{n-1}\cdot\{\nabla^{n-1}\varphi_{\e}\}\Big\}
%\\ \Big\{\e^n\vec P_n\cdot\nabla^{n} \{\vec A\cdot\nabla
%v_{\e}\},\ldots,\e\vec P_1\cdot\nabla\{\vec A\cdot\nabla
%v_{\e}\},\{\vec A\cdot\nabla v_{\e}\}\Big\},\,
\Big\{\e^{n_2}\vec Q\cdot\nabla^{n_2}
\varphi_{\e},\e^{n_2-1}\nabla^{n_2-1}
\varphi_{\e},\ldots,\varphi_{\e}\Big\}\Bigg)dx=\\
\liminf_{\e\to 0^+}\int_{I_{\vec \nu}}\frac{1}{\e}
F\Bigg(\Big\{\e^{n_1}\nabla^{n_1} \{\vec B\cdot\nabla
w_{\e}\},\ldots,\e\nabla\{\vec B\cdot\nabla
w_{\e}\},\{\vec B\cdot\nabla w_{\e}\}\Big\}\,,\\
%,\,\Big\{\e^{n-1}\,\vec P_{n-1}\cdot\{\nabla^{n} v_{\e}\}\,,\e^{n-1}\,\vec Q_{n-1}\cdot\{\nabla^{n-1}\varphi_{\e}\}\Big\}
%\Big\{\e^n\vec P_n\cdot\nabla^{n}_x \big\{\vec A\cdot\nabla_x
%\big(h_\e+\eta_\e+\gamma_\e(\vec\nu\cdot x)\big)\big\},\ldots,
%\e\vec P_1\cdot\nabla_x\big\{\vec A\cdot\nabla_x
%\big(h_\e+\eta_\e+\gamma_\e(\vec\nu\cdot x)\big)\big\},\big\{\vec
%A\cdot\nabla_x \big(h_\e+\eta_\e+\gamma_\e(\vec\nu\cdot
%x)\big)\big\}\Big\},\,\\
\Big\{\e^{n_2}\vec Q\cdot\nabla^{n_2}
\big(\lambda_{\e}+\theta_\e\big),\e^{n_2-1}\nabla^{n_2-1}
\big(\lambda_{\e}+\theta_\e\big),\ldots,\big(\lambda_{\e}+\theta_\e\big)\Big\}\Bigg)dx\geq
\\
\liminf_{\e\to 0^+}\int_{I_{\vec \nu}}\frac{1}{\e}
F\Bigg(\bigg\{\e^{n_1}\nabla^{n_1} \{\vec B\cdot\nabla
f_{\e}\},\ldots,\e\nabla\{\vec B\cdot\nabla
f_{\e}\},\{\vec B\cdot\nabla f_{\e}\}\bigg\}\,,\\
%%,\,\Big\{\e^{n-1}\,\vec P_{n-1}\cdot\{\nabla^{n} v_{\e}\}\,,\e^{n-1}\,\vec Q_{n-1}\cdot\{\nabla^{n-1}\varphi_{\e}\}\Big\}
%\bigg\{\e^n\vec P_n\cdot\nabla^{n}_x \big\{\vec A\cdot\big(\nabla
%\bar h_\e+\nabla\bar\eta_\e+\bar\delta_\e(\vec\nu\cdot
%x)\otimes\vec\nu\big)\big\}\\,\ldots, \e\vec
%P_1\cdot\nabla_x\big\{\vec A\cdot\big(\nabla \bar
%h_\e+\nabla\bar\eta_\e+\bar\delta_\e(\vec\nu\cdot
%x)\otimes\vec\nu\big)\big\},\big\{\vec A\cdot\big(\nabla \bar
%h_\e+\nabla\bar\eta_\e+\bar\delta_\e(\vec\nu\cdot
%x)\otimes\vec\nu\big)\big\}\bigg\},\,\\
\bigg\{\e^{n_2}\vec Q\cdot\nabla^{n_2}
\big(\bar\lambda_{\e}+\bar\theta_\e\big),\e^{n_2-1}\nabla^{n_2-1}
\big(\bar\lambda_{\e}+\bar\theta_\e\big),\ldots,\big(\bar\lambda_{\e}+\bar\theta_\e\big)\bigg\}\Bigg)dx,
\end{multline}
where $f_\e(x):=w_{r_\e\e}(r_\e x)/r_\e$
%, $\bar h_\e(x):=h_{r_\e\e}(r_\e x)/r_\e$
and $\bar\lambda_\e(x):=\lambda_{r_\e\e}(r_\e x)$. Moreover, by the
same Lemma, $\lim_{\e\to 0^+}\{\vec B\cdot \nabla f_\e\}=W_0$ in
$L^{p_1}(I_{\vec \nu},\R^m)$.
%%%%%%%%%%%%%%%%%%%%%%%%%%%%%%%%%%%%%%%%%%%%%%%%%%%%%%%%%%%%%%%%%%%%%%%%%%%%%%%%%%%%%%%%%%%%%%%%%%%%%%%%%%%%%%%%%%%%%%%%%%%%%%%%%%%%%%%%%%%%%%%%%%%%%%%%%%%%%%%%%%%%%%%%%%%%%%%%%%%%%%%%%%%%%%%%%%%%%%%%%%%%%%%%%%%%%%%%%%%%%%%%%%%%%%%%%%%%%%%%%%%%%%%%%%%%%%%%%%%%%%%%%%%%ZZZ
%, $\lim_{\e\to 0^+}\big(\e^{n_1}\,\vec
%R\cdot\nabla^{n_1}\big\{\vec B\cdot\nabla f_\e\big\}\big)=0$ in
%$L^{p_1}(I_{\vec \nu},\R^{r})$ and for every $j=1,\ldots, (n_1-1)$
%we have $\lim_{\e\to 0^+}\big(\e^j\nabla^{j}\big\{\vec B\cdot\nabla
%f_\e\big\}\big)=0$ in $L^{p_1}(I_{\vec \nu},\R^{m\times N^j})$.
On
the other hand clearly
$\bar\lambda_\e(x)\in\mathcal{S}^{(n_2)}\big(\varphi^+,\varphi^-,I_{\vec\nu}\big)$
%and $\bar h_\e(x)\in\mathcal{S}^{(n_2)}_1\big(V^+,V^-,I_{\vec\nu}\big)$
for
$\e>0$ sufficiently small.
%Moreover, clearly there exists $\bar\gamma_\e(x)\in\mathcal{S}^{(n_2)}_1\big(0,0,I_{\vec\nu}\big)$,
%such that $\nabla\bar\gamma_\e(x)\equiv \bar\delta_\e(\vec\nu\cdot x)\otimes\vec\nu$.
Thus, since $\bar\theta_\e(x)\in
C^\infty_c\big(I_{\vec\nu},\R^m\big)$,
%and $\bar\eta_\e(x)\in
%C^\infty_c\big(I_{\vec\nu},\R^k\big)$
we have
$\psi_\e(x)\in\mathcal{S}^{(n_2)}\big(\varphi^+,\varphi^-,I_{\vec\nu}\big)$
%and $u_\e(x)\in\mathcal{S}^{(n_2)}_1\big(V^+,V^-,I_{\vec\nu}\big)$,
%such that
where $\psi_\e(x)\equiv \bar\lambda_{\e}(x)+\bar\theta_\e(x)$ for
every $x\in I_{\vec\nu}$.
%and $u_\e(x)\equiv \bar h_{\e}(x)+\bar\eta_\e(x)+\bar\gamma_\e(x)$ for every $x\in I_{\vec\nu}$.
So by \er{ggfghjjhfhfjfhjkkknewggkgkggghjjjh} we deduce
\er{ggfghjjhfhfjfhjkkknew}. On the other hand since $r_\e\to 1^-$ we
easily obtain $\lim_{\e\to 0^+}\psi_\e=\varphi_0$ in
$L^{p_2}(I_{\vec \nu},\R^k)$,
%$\lim_{\e\to 0^+}\{\vec A\cdot \nabla
%u_\e\}=V_0$ in $L^{p_2}(I_{\vec \nu},\R^d)$,
$\lim_{\e\to 0^+}\big(\e^{n_2}\,\vec Q\cdot\{\nabla^{n_2}
\psi_\e\}\big)=0$ in $L^{p_2}(I_{\vec \nu},\R^{l})$,
%for every $j=1,\ldots, n_2$ we have $\lim_{\e\to 0^+}\big(\e^{j}\,\vec
%P_{j}\cdot\nabla^{j}\big\{\vec A\cdot\nabla u_\e\big\}\big)=0$ in
%$L^{p_2}(I_{\vec \nu},\R^{d_{j}})$,
%
%
%
%$\lim_{\e\to 0^+}\big(\e^{n_1}\,\vec R_{n_1}\cdot\nabla^{n_1}\big\{\vec
%B\cdot\nabla f_\e\big\}\big)=0$ in $L^{p_1}_{loc}(I_{\vec\nu},\R^{r_{n_1}})$,
%%%%$
%$\lim_{\e\to 0^+}\big(\sum_{s=1}^{n}\|\vec P_s\|\big)\nabla u_\e=0$ in
%$L^p_{loc}(I_{\vec \nu},\R^{k\times N})$,
%%%%for every $j=2,\ldots, n_2$ we have $\lim_{\e\to
%0^+}\big(\sum_{s=j}^{n_2}\|\vec P_s\|\big)\big(\e^{j-1}\,\nabla^{j}
%u_\e\big)=0$ in $L^{\bar p}_{loc}(I_{\vec \nu},\R^{k\times N^{j}})$,
and for every $j=1,\ldots, (n_2-1)$ we have $\lim_{\e\to
0^+}\big(\e^j\,\nabla^j \psi_\e\big)=0$ in $L^{p_3}(I_{\vec
\nu},\R^{k\times N^j})$.
%Moreover, $\lim_{\e\to
%0^+}\frac{1}{\e}\big(u_\e(x)-h_\e(x)\big)=0$ in $L^{p}_{loc}(I_{\vec\nu},\R^k)$.
This completes the proof.
\end{proof}

Next we have the following simple Lemma.
\begin{lemma}\label{nfjghfighfihjtfohjt}
Let $n_0\in\mathbb{N}$ and $\big\{\varphi_\e(x)\big\}_{\e>0}\subset
W^{n_0,p}_{loc}\big(I_{\vec\nu},\R^m\big)$ be such that
$\e^{n_0}\nabla^{n_0} \varphi_\e\to 0$ in
$L^p_{loc}\big(I_{\vec\nu},\R^{m\times N^{n_0}}\big)$ and the
sequence $\big\{\varphi_\e(x)\big\}_{\e>0}$ is bounded in
$L^p_{loc}\big(I_{\vec\nu},\R^{m}\big)$ i.e. for every compactly
embedded open set $G\subset\subset I_{\vec\nu}$ there exists a
constant $\bar C:=\bar C(G)>0$ such that $\int_{G}|\varphi_\e|^p
dx\leq \bar C$. Then for every $j\in\{1,
%2,
\ldots,n_0\}$ we have
$\e^{j}\nabla^{j} \varphi_\e\to 0$ in
$L^p_{loc}\big(I_{\vec\nu},\R^{m\times N^{j}}\big)$.
\end{lemma}
\begin{proof}
Indeed fix an arbitrary domain $U\subset\subset I_{\vec \nu}$ with a
smooth boundary. Then clearly
\begin{equation}\label{gfjfhjfgjhfhkdrydsgbnvfjggyhggghfgfgdfdddrrdnewllkk}
d_e:=
%\int_U\bigg(\big|\e^{n_0}\,\nabla^{n_0}\varphi_\e\big|^p\bigg)
\big\|\e^{n_0}\,\nabla^{n_0} \varphi_\e\big\|_{L^p(U)} \to
0\quad\text{as}\;\;\e\to 0^+\,.
\end{equation}
Moreover, clearly there exists $\bar C>0$ such that
$\int_{U}|\varphi_\e|^p dx\leq \bar C$. On the other hand, by
Theorem 7.28 in \cite{gt} there exists $C_0>0$, which depends only
on $U$ $p$ and $n_0$, such that for every $\sigma(x)\in
W^{n_0,p}(U,\R^m)$ and every $\tau>0$ we have
\begin{equation}\label{gfjfhjfgjhfhkdrydsgnewllkk}
\big\|\nabla^j\sigma(x)\big\|_{L^p(U)}\leq \tau
\big\|\sigma(x)\big\|_{W^{n,p}(U)}+C_0\tau^{-j/(n-j)}\big\|\sigma(x)\big\|_{L^p(U)}\quad\quad\forall\;
2\leq n\leq n_0\,,\;\;1\leq j<n\,.
\end{equation}
Thus in particular we deduce from \er{gfjfhjfgjhfhkdrydsgnewllkk}
that there exists $C_1>0$, which depends only on $U$ $p$ and $n_0$,
such that for every $\sigma(x)\in W^{n_0,p}(U,\R^m)$ and every
$\tau\in (0,1)$ we have
\begin{equation}\label{gfjfhjfgjhfhkdrydsghgghgfgffgfggnewllkk}
\big\|\tau^j\nabla^j\sigma(x)\big\|_{L^p(U)}\leq
\big\|\tau^{n_0}\nabla^{n_0}\sigma(x)\big\|_{L^p(U)}+C_1\big\|\sigma(x)\big\|_{L^p(U)}\quad\quad\forall\;
1\leq j<n_0\,.
\end{equation}
Then setting $\tau:=\e\cdot(d_\e)^{-1/n_0}$, where $d_\e$ is defined
by \er{gfjfhjfgjhfhkdrydsgbnvfjggyhggghfgfgdfdddrrdnewllkk}, using
\er{gfjfhjfgjhfhkdrydsghgghgfgffgfggnewllkk} and the fact that
$\int_{U}|\varphi_\e|^pdx\leq \bar C$ we obtain
\begin{equation}\label{gfjfhjfgjhfhkdrydsghgghgfgffgfggjhhgkhhhlllhhljjggjkgkjknewllkk}
\big\|\e^j\nabla^j\varphi_\e(x)\big\|_{L^p(U)}\leq \hat C
d_\e^{j/n_0}\quad\quad\forall\; 1\leq j<n_0\,,
\end{equation}
where $\hat C>0$ dose not depend on $\e$. Thus using
\er{gfjfhjfgjhfhkdrydsgbnvfjggyhggghfgfgdfdddrrdnewllkk} we deduce
\begin{equation*}
\big\|\e^j\nabla^j\varphi_\e(x)\big\|_{L^p(U)}\to
0\quad\text{as}\;\;\e\to 0^+\quad\forall\; 1\leq j<n_0\,,
\end{equation*}
Therefore, since the domain with a smooth boundary $U\subset\subset
I_{\vec \nu}$ was chosen arbitrary, we  finally deduce that
$\lim_{\e\to 0^+}\big(\e^{j}\,\nabla^{j} \varphi_\e\big)=0$ in
$L^p_{loc}(I_{\vec \nu},\R^{m\times N^{j}})$ for every
$j\in\{1,
%2,
\ldots,n_0\}$. This completes the proof.
\end{proof}

Plugging Lemma \ref{nfjghfighfihjtfohjt} with the  particular case
of Proposition \ref{L2009.02kkknew} we get the following Theorem.
\begin{theorem}\label{L2009.02kkkjkhkjhnew}
Let $n_1,n_2\in \mathbb{N}$. Consider the linear operator
% $\vec A\in\mathcal{L}(\R^{k\times N},\R^d)$,
%$\vec B\in\mathcal{L}(\R^{d\times N},\R^m)$
%$n\in\mathbb{N}$,
%$\vec R\in\mathcal{L}(\R^{m\times N^{n_1}},\R^{r})$,
%for all $j=1,2,\ldots n_1$,
%$\vec P_j\in\mathcal{L}(\R^{d\times N^j},\R^{d_j})$ for all $j=1,2,\ldots n_2$
%and
$\vec Q\in\mathcal{L}(\R^{m\times N^{n_2}},\R^{l})$, which satisfies
\begin{itemize}
\item
either $\vec Q=Id$
\item
or $n_2=1$.
\end{itemize}
Next
%for all $j=1,2,\ldots n_3$, and
%$\Omega$ be a bounded domain
%order %(in particular a bounded $BVG$-domain)
%$F\in C^1(\R^{k\times N\times N}\times\R^{m\times
%N}\times\R^{k\times N}\times\R^m\,,\R)$
assume that $F$ is a continuous function, defined on
$$\Big\{\R^{k\times N^{n_1+1}}\times\ldots\times\R^{k\times
N^2}\times\R^{k\times N}\Big\}\times\Big\{\R^{d\times
N^{n_1+1}}\times\ldots\times\R^{d\times N^2}\times\R^{d\times
N}\Big\}\times
%\{\R^{d_{n_2}}\times\ldots\times\times\R^{d_{2}}\times\R^{d_1}\times\R^{d}\}\times
\Big\{\R^{l}\times\R^{m\times
N^{n_2-1}}\times\ldots\times\R^{m\times N}\times\R^m\Big\},$$ taking
values in $\R$ and satisfying $F\geq 0$. Moreover assume that there
exist $C>0$ and $p_1\geq 1$, $p_2\geq 1$ satisfying
%$0\leq F(a,b,c,d)-F(0,0,c,d)\leq C|a|^{p_1}+C|b|^{p_1}$ and
%$F(0,0,c,d)\leq C (|c|^{p_2}+|d|^{p_2}+1)$ for every $(a,b,c,d)$.
\begin{multline}\label{hgdfvdhvdhfvnew}
\frac{|c_1|^{p_2}}{C}\leq F\Big(\{a_1,\ldots,
a_{n_1+1}\},\{b_1,\ldots, b_{n_1+1}\},\{c_1,\ldots,c_{n_2+1}\}
%,\ldots,\{c_1,c_2,\ldots,c_{n_3},c\}
\Big)\leq
C\bigg(\sum_{j=1}^{n_1+1}|a_j|^{p_1}+\sum_{j=1}^{n_1+1}|b_j|^{p_1}+\sum_{j=1}^{n_2+1}|c_j|^{p_2}
%\sum_{j=1}^{n_3}|c_j|^{p_3}+
%|a|^{p_1}+
%|b|^{p_2}
%+|c|^{p_3}
+1\bigg)\\ \text{for every}\;\; \Big(\{a_1,\ldots,
a_{n_1+1}\},\{b_1,\ldots,b_{n_1+1}\},\{c_1,\ldots,c_{n_2+1}\}
%,\ldots,\{c_1,c_2,\ldots,c_{n_3},c\}
\Big).
%\Big(\{a_1,b_1,c_1\},\{a_2,b_2,c_2\},\ldots,\{a_n,b_n,c_n\},a,b,c\Big)
\end{multline}
Furthermore let
%$\vec k\in\R^k$,
$\vec\nu\in S^{N-1}$, $\varphi^+,\varphi^-\in\R^m$,
$V^+,V^-\in\R^{k\times N}$ and $W^+,W^-\in \R^{d\times N}$ be such
that if we set
\begin{equation*}
%\label{vfyguiguhikjnklklhkukuytou}
\varphi_0(x): =\begin{cases}
\varphi^+\;\;\text{if}\;\;x\cdot\vec\nu>0,\\
\varphi^-\;\;\text{if}\;\;x\cdot\vec\nu<0,
\end{cases}
\; V_0(x): =\begin{cases}
V^+\;\;\text{if}\;\;x\cdot\vec\nu>0,\\
V^-\;\;\text{if}\;\;x\cdot\vec\nu<0,
\end{cases}
\;\text{and}\quad\; W_0(x): =\begin{cases}
W^+\;\;\text{if}\;\;x\cdot\vec\nu>0,\\
W^-\;\;\text{if}\;\;x\cdot\vec\nu<0,
\end{cases}
\end{equation*}
%there exists $\vec k\in\R^k$ which satisfy
then
%\begin{equation}\label{vfyguiguhikjnklklhbjkbbjk}
%F\big(0,0,\ldots,0,W^+,V^+,\varphi^+\big)=F\big(0,0,\ldots,0,W^-,V^-,\varphi^-\big)=0\,,
%\end{equation}
\begin{equation*}
%\label{vfyguiguhikjnklklhbjkbbjk}
F\Big(\big\{0,0,\ldots,V_0(x)\big\}, \big\{0,0,\ldots,W_0(x)\big\},
\big\{0,0,\ldots,\varphi_0(x)\big\}\Big)=0\quad\text{for
a.e.}\;x\in\R^N\,.
\end{equation*}
%for some $w\in\mathcal{D}'(\R^N,\R^q)$.
%Set $\varphi(x)\in L^\infty(\R^N,\R^m)$ and $v(x):Lip(\R^N,\R^k)$ by
%\begin{equation}\label{ghgghjhjkdfhgkkknew}
%\varphi(x):=\begin{cases}
%\varphi^+\quad\text{if}\;\;x\cdot\vec\nu>0\,,\\
%\varphi^-\quad\text{if}\;\;x\cdot\vec\nu<0\,,
%\end{cases}\quad\quad
%v(x):=\begin{cases}
%V^-\cdot x+(x\cdot\vec\nu)\vec k\quad\text{if}\;\;x\cdot\vec\nu\geq 0\,,\\
%V^-\cdot x\quad\quad\quad\text{if}\;\;x\cdot\vec\nu<0\,.
%\end{cases}
%\end{equation}
Next
%assume $p\geq\max{(p_1,p_2)}$ and
consider $\big\{v_\e(x)\big\}_{0<\e<1}\subset
%W^{(n+1),p}
W^{n_1+1,p_1}_{loc}(I_{\vec \nu},\R^k)$,
$\big\{m_\e(x)\big\}_{0<\e<1}\subset
%W^{(n+1),p}
W^{n_1,p_1}_{loc}(I_{\vec \nu},\R^{d\times N})$
%$\big\{v_\e(x)\big\}_{0<\e<1}\subset L^{p_2}_{loc}(I_{\vec \nu},\R^k)$
and $\big\{\varphi_\e(x)\big\}_{0<\e<1}\subset
%W^{n,p}
W^{n_2,p_2}_{loc}(I_{\vec \nu},\R^m)$, such that $\,\Div m_\e\equiv
0$, $\,\lim_{\e\to 0^+}\varphi_\e=\varphi_0$ in
$L^{p_2}_{loc}(I_{\vec \nu},\R^k)$, $\,\lim_{\e\to 0^+}m_\e=W_0$ in
$L^{p_1}_{loc}(I_{\vec \nu},\R^{d\times N})$ and $\lim_{\e\to
0^+}\nabla v_\e=V_0$ in $W^{1,p_1}_{loc}(I_{\vec \nu},\R^k)$.
%$\lim_{\e\to 0^+}\{\vec A\cdot\nabla v_\e\}=V_0$ in $L^{p_2}_{loc}(I_{\vec \nu},\R^d)$,
%$\lim_{\e\to 0^+}\big(\e^{n_1}\,\vec R\cdot\nabla^{n_1}\big\{\vec
%B\cdot\nabla w_\e\big\}\big)=0$ in $L^{p_1}_{loc}(I_{\vec\nu},\R^{r})$,
%$\lim_{\e\to 0^+}\big(\e^{n_2}\,\vec Q\cdot\{\nabla^{n_2}
%\varphi_\e\}\big)=0$ in $L^{p_2}_{loc}(I_{\vec \nu},\R^{l})$
%for every $j=1,2,\ldots, (n_2-1)$ we have $\lim_{\e\to
%0^+}\big(\e^{j}\nabla^{j}\big\{\vec A\cdot\nabla v_\e\big\}\big)=0$
%in $L^{p_2}_{loc}(I_{\vec \nu},\R^{d_{j}})$,
%for every $j=1,2,\ldots, (n_1-1)$ we have $\lim_{\e\to
%0^+}\big(\e^j\nabla^{j}\big\{\vec B\cdot\nabla w_\e\big\}\big)=0$ in
%$L^{p_1}_{loc}(I_{\vec \nu},\R^{m\times N^j})$
%for
%every $j=2,\ldots, n_2$ we have $\lim_{\e\to
%0^+}\big(\sum_{s=j}^{n_2}\|\vec
%P_s\|\big)\big(\e^{j-1}\,\nabla^{j}v_\e\big)=0$ in $L^{\bar
%p}_{loc}(I_{\vec \nu},\R^{k\times N^{j}})$,
%and for every $j=1,2,\ldots, (n_2-1)$ we have
%$\lim_{\e\to 0^+}\big(\sum_{s=j+1}^{n}\|\vec
%R_s\|\big)\big(\e^j\,\nabla^{j} \{\vec B\cdot\nabla w_\e\}\big)=0$
%in $L^p_{loc}(I_{\vec \nu},\R^{r\times N^{j}})$,
%$\lim_{\e\to 0^+}\big(\e^j\,\nabla^j \varphi_\e\big)=0$ in $L^{
%p_2}_{loc}(I_{\vec \nu},\R^{m\times N^j})$.
%, where $\bar p:=\max\{p_1,p_2,p_3\}$
%and
%and $\lim_{\e\to 0^+}\big(\sum_{s=1}^{n}\|\vec R_s\|\big)\nabla w_\e=0$
%in $L^p_{loc}(I_{\vec \nu},\R^{q\times N})$,
%$\lim_{\e\to
%0^+}\big(\sum_{s=1}^{n}\|\vec P_s\|\big)\nabla v_\e=0$ in
%$L^p_{loc}(I_{\vec \nu},\R^{k\times N})$.
%and $\lim_{\e\to 0^+}\big\{(v_\e-v)/\e\big\}=0$ in $L^{p}_{loc}(I_{\vec\nu},\R^k)$,
Here, as before, $I_{\vec \nu}:=\{y\in\R^N:\;|y\cdot
\vec\nu_j|<1/2\;\;\;\forall j=1,\ldots, N\}$ where
$\{\vec\nu_1,\ldots,\vec\nu_N\}\subset\R^N$ is an orthonormal base
in $\R^N$ such that $\vec\nu_1:=\vec \nu$.
Then, there exist $\big\{\psi_\e(x)\big\}_{0<\e<1}\subset
\mathcal{S}^{(n_2)}\big(\varphi^+,\varphi^-,I_{\vec\nu}\big)$,
$\big\{h_\e(x)\big\}_{0<\e<1}\subset
%W^{(n+1),p}
W^{n_1,p_1}(I_{\vec \nu},\R^{d\times N})$
%, $\big\{u_\e(x)\big\}_{0<\e<1}\subset\mathcal{S}^{(n_2)}_1\big(V^+,V^-,I_{\vec\nu}\big)$
and $\big\{u_\e(x)\big\}_{0<\e<1}\subset
%W^{(n+1),p}
W^{n_1+1,p_1}(I_{\vec \nu},\R^k)$,
%$\big\{v_\e(x)\big\}_{0<\e<1}\subset
%W^{(n+1),p}
%L^p_{loc}(I_{\vec \nu},\R^k)$
where
%\begin{multline}\label{L2009Ddef2hhhjjjj77788hhhkkkkllkjjjjkkkhhhhffggdddkkkgjhikhhhjjfgnhfkkknew}
%\mathcal{S}^{(n)}_1\big(V^+,V^-,I_{\vec\nu}\big):=
%%\mathcal{D}\Big(\vec A,(\vec A\cdot\nabla v)^-(x),(\vec A\cdot\nabla v)^+(x),\vec\nu(x),\vec k_2(x),\vec k_3(x),\ldots,\vec k_{N}(x)\Big):=\\
%\bigg\{\xi\in C^{n+1}(\R^N,\R^k):\;\;\vec A\cdot\nabla \xi(y)=V^-\;\text{ if }\;y\cdot\vec\nu\leq-1/2,\\
%\vec A\cdot\nabla \xi(y)=V^+\;\text{ if }\; y\cdot\vec\nu\geq
%1/2\;\text{ and }\;\vec A\cdot\nabla \xi\big(y+\vec\nu_j\big)=\vec
%A\cdot\nabla \xi(y)\;\;\forall j=2,3,\ldots, N\bigg\}\,,
%\end{multline}
%and
\begin{multline*}
%\label{L2009Ddef2hhhjjjj77788hhhkkkkllkjjjjkkkhhhhffggdddkkkgjhikhhhjjddddhdkgkkknew}
\mathcal{S}^{(n)}\big(\varphi^+,\varphi^-,I_{\vec\nu}\big):=
%\mathcal{D}\Big(\vec A,(\vec A\cdot\nabla v)^-(x),(\vec A\cdot\nabla v)^+(x),\vec\nu(x),\vec k_2(x),\vec k_3(x),\ldots,\vec k_{N}(x)\Big):=\\
\bigg\{\zeta\in
C^n(\R^N,\R^m):\;\zeta(y)=\varphi_0(y)\;\,\text{if}\;\,|y\cdot\vec\nu|\geq
1/2,\;\,\text{and}\;\,\zeta\big(y+\vec \nu_j\big)=\zeta(y)\;\forall
j=2,\ldots, N\bigg\},
\end{multline*}
such that such that $\,\Div h_\e\equiv 0$, $\,\lim_{\e\to
0^+}\psi_\e=\varphi_0$ in $L^{p_2}(I_{\vec \nu},\R^m)$,
$\,\lim_{\e\to 0^+}h_\e=W_0$ in $L^{p_1}(I_{\vec \nu},\R^{d\times
N})$,
%$\lim_{\e\to 0^+}\{\vec A\cdot \nabla u_\e\}=V_0$ in $L^{p_2}(I_{\vec \nu},\R^d)$,
$\lim_{\e\to 0^+}\nabla u_\e=V_0$ in $W^{1,p_1}(I_{\vec \nu},\R^k)$,
%$\lim_{\e\to 0^+}\big(\e^{n_1}\,\vec
%R\cdot\nabla^{n_1}\big\{\vec B\cdot\nabla f_\e\big\}\big)=0$ in
%$L^{p_1}_{loc}(I_{\vec \nu},\R^{r})$,
%$\lim_{\e\to 0^+}\big(\e^{n_2}\,\vec Q\cdot\{\nabla^{n_2}
%\psi_\e\}\big)=0$ in $L^{p_2}(I_{\vec \nu},\R^{l})$,
%for every $j=1,2,\ldots, (n_1-1)$ we have $\lim_{\e\to
%0^+}\big(\e^{j}\nabla^{j}\big\{\vec B\cdot\nabla f_\e\big\}\big)=0$
%in $L^{p_1}(I_{\vec \nu},\R^{m\times N^j})$,
%$\lim_{\e\to 0^+}\big(\sum_{s=1}^{n}\|\vec P_s\|\big)\nabla u_\e=0$ in
%$L^p_{loc}(I_{\vec \nu},\R^{k\times N})$,
%for every $j=2,\ldots, n_2$ we have $\lim_{\e\to
%0^+}\big(\sum_{s=j}^{n_2}\|\vec P_s\|\big)\big(\e^{j-1}\,\nabla^{j}
%u_\e\big)=0$ in $L^{\bar p}_{loc}(I_{\vec \nu},\R^{k\times N^{j}})$,
%for every $j=1,2,\ldots, (n_3-1)$ we have $\lim_{\e\to
%0^+}\big(\e^j\,\nabla^j \psi_\e\big)=0$ in $L^{p_2}(I_{\vec
%\nu},\R^{m\times N^j})$,
%for every $j=1,2,\ldots, (n_1-1)$ we have $\lim_{\e\to
%0^+}\big(\e^j\,\vec R_j\cdot\nabla^{j}\big\{\vec B\cdot\nabla
%f_\e\big\}\big)=0$ in $L^{\bar p}_{loc}(I_{\vec \nu},\R^{r_j})$,
%%%%%
% and $\lim_{\e\to 0^+}\big(\sum_{s=j+1}^{n}\|\vec
%R_s\|\big)\big(\e^j\,\nabla^{j} \{\vec B\cdot\nabla f_\e\}\big)=0$
%in $L^p_{loc}(I_{\vec \nu},\R^{r\times N^{j}})$,
and
\begin{multline*}
%\label{ggfghjjhfhfjfhjkkknew}
\liminf_{\e\to 0^+}\int_{I_{\vec \nu}}\frac{1}{\e}
F\Bigg(\Big\{\e^{n_1}\nabla^{n_1+1}v_{\e},\ldots,\e\nabla^2
v_{\e},\nabla v_{\e}\Big\},\,
\Big\{\e^{n_1}\nabla^{n_1}m_{\e},\ldots,\e\nabla m_{\e},
m_{\e}\Big\},\\
%,\,\Big\{\e^{n-1}\,\vec P_{n-1}\cdot\{\nabla^{n} v_{\e}\}\,,\e^{n-1}\,\vec Q_{n-1}\cdot\{\nabla^{n-1}\varphi_{\e}\}\Big\}
%\Big\{\e^n\vec P_n\cdot\nabla^{n} \{\vec A\cdot\nabla
%v_{\e}\},\ldots,\e\vec P_1\cdot\nabla\{\vec A\cdot\nabla
%v_{\e}\},\{\vec A\cdot\nabla v_{\e}\}\Big\},\,
\Big\{\e^{n_2}\vec Q\cdot\nabla^{n_2}
\varphi_{\e},\e^{n_2-1}\nabla^{n_2-1}
\varphi_{\e},\ldots,\varphi_{\e}\Big\}\Bigg)dx \geq\\ \liminf_{\e\to
0^+}\int_{I_{\vec \nu}}\frac{1}{\e}
F\Bigg(\Big\{\e^{n_1}\nabla^{n_1+1}u_{\e},\ldots,\e\nabla^2
u_{\e},\nabla
u_{\e}\Big\},\,\Big\{\e^{n_1}\nabla^{n_1}h_{\e},\ldots,\e\nabla
h_{\e}, h_{\e}\Big\},\\
%,\,\Big\{\e^{n-1}\,\vec P_{n-1}\cdot\{\nabla^{n} v_{\e}\}\,,\e^{n-1}\,\vec Q_{n-1}\cdot\{\nabla^{n-1}\varphi_{\e}\}\Big\}
%\Big\{\e^n\vec P_n\cdot\nabla^{n} \{\vec A\cdot\nabla
%u_{\e}\},\ldots,\e\vec P_1\cdot\nabla\{\vec A\cdot\nabla
%u_{\e}\},\{\vec A\cdot\nabla u_{\e}\}\Big\},\,
\Big\{\e^{n_2}\vec Q\cdot\nabla^{n_2}
\psi_{\e},\e^{n_2-1}\nabla^{n_2-1}
\psi_{\e},\ldots,\psi_{\e}\Big\}\Bigg)dx.
\end{multline*}
\end{theorem}
\begin{proof}
It is clear that without any loss of generality we may assume
\begin{multline}\label{ggfghjjhfhfjfhjkkkhjkgghghhjhhbbhjjhbnew}
\liminf_{\e\to 0^+}\int_{I_{\vec \nu}}\frac{1}{\e}
F\Bigg(\Big\{\e^{n_1}\nabla^{n_1+1}v_{\e},\ldots,\e\nabla^2
v_{\e},\nabla v_{\e}\Big\},\,
\Big\{\e^{n_1}\nabla^{n_1}m_{\e},\ldots,\e\nabla m_{\e},
m_{\e}\Big\},\\
%,\,\Big\{\e^{n-1}\,\vec P_{n-1}\cdot\{\nabla^{n} v_{\e}\}\,,\e^{n-1}\,\vec Q_{n-1}\cdot\{\nabla^{n-1}\varphi_{\e}\}\Big\}
%\Big\{\e^n\vec P_n\cdot\nabla^{n} \{\vec A\cdot\nabla
%v_{\e}\},\ldots,\e\vec P_1\cdot\nabla\{\vec A\cdot\nabla
%v_{\e}\},\{\vec A\cdot\nabla v_{\e}\}\Big\},\,
\Big\{\e^{n_2}\vec Q\cdot\nabla^{n_2}
\varphi_{\e},\e^{n_2-1}\nabla^{n_2-1}
\varphi_{\e},\ldots,\varphi_{\e}\Big\}\Bigg)dx<+\infty.
\end{multline}
Then by \er{hgdfvdhvdhfvnew} and
\er{ggfghjjhfhfjfhjkkkhjkgghghhjhhbbhjjhbnew}
%and the fact that
%$\,\lim_{\e\to 0^+}\varphi_\e=\varphi$ in $L^{p}_{loc}(I_{\vec
%\nu},\R^l)$ and $\lim_{\e\to 0^+}v_\e=v$ in $W^{1,p}_{loc}(I_{\vec
%\nu},\R^k)$
we deduce that $\lim_{\e\to 0^+}\big(\e^{n_2}\,\vec
Q\cdot\{\nabla^{n_2} \varphi_\e\}\big)=0$ in $L^{p_2}_{loc}(I_{\vec
\nu},\R^{l})$. Next remember we assumed that
\begin{itemize}
\item
either $\vec Q=Id$
\item
or $n_2=1$.
\end{itemize}
So in any case, by Lemma \ref{nfjghfighfihjtfohjt} we deduce that
for every $1\leq j<n_2$ we have $\lim_{\e\to
0^+}\big(\e^{j}\,\nabla^{j} \varphi_\e\big)=0$ in
$L^{p_2}_{loc}(I_{\vec \nu},\R^{m\times N^{j}})$.
% for every $1\leq j\leq (n_2-1)$.
Thus, applying Proposition \ref{L2009.02kkknew}
completes the proof.
\end{proof}
Next by the composition of Theorem \ref{L2009.02kkkjkhkjhnew} and
Theorem \ref{dehgfrygfrgygenjklhhj}
%{dehgfrygfrgygen}
we obtain the
following result describing the lower bound for the first order
problems.
\begin{proposition}\label{dehgfrygfrgygenbgggggggggggggkgkgnew}
Let $\O\subset\R^N$ be an open set. Furthermore, let $F\in
C^0\big(\R^{m\times N^n}\times\R^{m\times
N^{(n-1)}}\times\ldots\times\R^{m\times N}\times
\R^m\times\R^N,\R\big)$, be such that $F\geq 0$ and there exist
$g(x)\in C^0\big(\R^N,(0,+\infty)\big)$ and $p\geq 1$ satisfying
\begin{multline}\label{hgdfvdhvdhfvjjjjiiiuyyyjinew}
\frac{1}{g(x)}|A|^p
%-C\Big(|b|^p+1\Big)
\leq F\Big(A,a_1,\ldots,a_{n-1},b,x\Big) \leq
g(x)\bigg(|A|^p+\sum_{j=1}^{n-1}|a_j|^{p}+|b|^p+1\bigg)\quad
\text{for every}\;\;\big(A,a_1,a_2,\ldots,a_{n-1},b,x\big) \,.
\end{multline}
Assume also that for every $x_0\in\O$ and every $\tau>0$ there
exists $\alpha>0$ satisfying
%for every $a\in \R^{m\times N}$, $b\in\R^m$ and $x\in\R^N$
%satisfying $|x-x_0|<\delta$ we have
\begin{multline}\label{vcjhfjhgjkgkgjgghjfhfhfmjghj}
F\big(a_1,a_2,\ldots, a_n,b,x\big)-F\big(a_1,a_2,\ldots,
a_n,b,x_0\big)\geq -\tau F\big(a_1,a_2,\ldots, a_n,b,x_0\big)\\
\forall\, a_1\in \times\R^{m\times N^n}\; \ldots\;\forall\,
a_n\in\R^{m\times N}\;\;\forall\, b\in\R^{m}\;\;\forall\,
x\in\R^N\;\;\text{such that}\;\;|x-x_0|<\alpha\,.
\end{multline}
%%%%%%%%%%%%%%%%%%%%%%%%%%%%%%%%%%%%%%%%%%%%%%%%%%%%%%%%%%%%%%%%%%%%%%%%%%%%%%%%%%%%%%%%%%%%%%%%%%%%%%%%%%%%%%%%%%%%%%%%%%%%%%%%%%%%%%%%%%%%%%%%%%%%%%%%%%%%%%%%%%%%%%%%%%%%%%%%%%%%%%%%%%%%%%%%%%%%%%%%%%%%%%%%%%%%%%%%%%%%%%%%%%%%%%%%%%%%%%%%%%%%%%%%%%%%%%%%%%%%%%%%%%%%%%%%%%%%%%%%%%%%%%%%%%%%%%%%%%%%%%%%%%%%%%%%%%%%%%%%%%%%%%%%%%%%%%%%%%%%%%%%%%%%%%%%%%%%%%%%%%%%%%%%%%%%%%%%%%%%%%%%%%%%%%%%%%%%%%%%%%%%%%%%%%%%%%%%
Next let $\varphi\in L^p(\O,\R^m)$ be such that $F\big(0,0,\ldots,0,
\varphi(x),x\big)=0$ for a.e. $x\in\O$. Assume also that there exist
a $\mathcal{H}^{N-1}$ $\sigma$-finite Borel set $D\subset\O$ and
three Borel mappings $\,\varphi^+(x):D\to\R^m$,
$\varphi^-(x):D\to\R^m$ and $\vec n(x):D\to S^{N-1}$ such that for
every $x\in D$ we have
\begin{equation}\label{L2009surfhh8128odno888jjjjjkkkkkkgenjnhhhnew}
\lim\limits_{\rho\to 0^+}\frac{\int_{B_\rho^+(x,\vec
n(x))}\big|\varphi(y)-\varphi^+(x)\big|^p\,dy}
{\mathcal{L}^N\big(B_\rho(x)\big)}=0\quad\text{and}\quad
\lim\limits_{\rho\to 0^+}\frac{\int_{B_\rho^-(x,\vec
n(x))}\big|\varphi(y)-\varphi^-(x)\big|^p\,dy}
{\mathcal{L}^N\big(B_\rho(x)\big)}=0\,.
\end{equation}
%Let $\O$, $D$, $\vec A$, $F$, $q$, $v$, $\vec n$, and $\{\vec
%A\cdot\nabla v\}^\pm$ as above.
Then for every
%sequence $\e_n\to 0^+$ as $n\to 0^+$ and every sequence
$\{\varphi_\e\}_{\e>0}\subset W^{n,p}_{loc}(\O,\R^m)$ such that
$\varphi_\e\to \varphi$ in $L^p_{loc}(\O,\R^m)$ as $\e\to 0^+$, we
will have
\begin{multline}\label{a1a2a3a4a5a6a7s8hhjhjjhjjjjjjkkkkgenhjhhhhjnew}
\liminf_{\e\to 0^+}\frac{1}{\e}\int_\O F\bigg(\,\e^n\nabla^n
\varphi_\e(x),\,\e^{n-1}\nabla^{n-1}\varphi_\e(x),\,\ldots,\,\e\nabla \varphi_\e(x),\, \varphi_\e(x),\,x\bigg)dx\\
\geq \int_{D}\bar E^{(n)}_{per}\Big(\varphi^+(x),\varphi^-(x),\vec
n(x),x\Big)d \mathcal H^{N-1}(x)\,,
\end{multline}
where
\begin{multline}\label{L2009hhffff12kkkhjhjghghgvgvggcjhggghnew}
\bar E^{(n)}_{per}\Big(\varphi^+,\varphi^-,\vec n,x\Big)\;:=\;\\
\inf\Bigg\{\int_{I_{\vec n}}\frac{1}{L} F\bigg(L^n\,\nabla^n
\zeta(y),\,L^{n-1}\,\nabla^{n-1} \zeta(y),\,\ldots,\,L\,\nabla
\zeta(y),\,\zeta(y),\,x\bigg)\,dy:\;\; L\in(0,+\infty)\,,\;\zeta\in
\mathcal{S}^{(n)}(\varphi^+,\varphi^-,I_{\vec n})\Bigg\}\,,
\end{multline}
with
\begin{multline}\label{L2009Ddef2hhhjjjj77788hhhkkkkllkjjjjkkkhhhhffggdddkkkgjhikhhhjjddddhdkgkkkhgghjhjhjkjnew}
\mathcal{S}^{(n)}\big(\varphi^+,\varphi^-,I_{\vec n}\big):=
%\mathcal{D}\Big(\vec A,(\vec A\cdot\nabla v)^-(x),(\vec A\cdot\nabla v)^+(x),\vec\nu(x),\vec k_2(x),\vec k_3(x),\ldots,\vec k_{N}(x)\Big):=\\
\bigg\{\zeta\in C^n(\R^N,\R^m):\;\;\zeta(y)=\varphi^-\;\text{ if }\;y\cdot\vec n\leq-1/2,\\
\zeta(y)=\varphi^+\;\text{ if }\; y\cdot\vec n\geq 1/2\;\text{ and
}\;\zeta\big(y+\vec \nu_j\big)=\zeta(y)\;\;\forall j=2,
%3,
\ldots,
N\bigg\}\,,
\end{multline}
Here $I_{\vec n}:=\{y\in\R^N:\;|y\cdot \vec\nu_j|<1/2\;\;\;\forall
j=1,2\ldots N\}$ where
$\{\vec\nu_1,
%\vec\nu_2,
\ldots,\vec\nu_N\}\subset\R^N$ is an
orthonormal base in $\R^N$ such that $\vec\nu_1:=\vec n$.
\end{proposition}
Thus by plugging Proposition
\ref{dehgfrygfrgygenbgggggggggggggkgkgnew} into Theorem
\ref{ffgvfgfhthjghgjhg}
%{ffgvfgfhthjghgjhg2}
we deduce the $\Gamma$-limit result for the
first order problem.
\begin{theorem}\label{dehgfrygfrgygenbgggggggggggggkgkgthtjtfnew}
Let $\O\subset\R^N$ be an open set. Furthermore, let $F\in
C^1\big(\R^{m\times N^n}\times\R^{m\times
N^{(n-1)}}\times\ldots\times\R^{m\times N}\times
\R^m\times\R^N,\R\big)$, be such that $F\geq 0$ and there exist
$g(x)\in C^0\big(\R^N,(0,+\infty)\big)$ and $p\geq 1$ satisfying
\begin{multline}\label{hgdfvdhvdhfvjjjjiiiuyyyjitghujtrnew}
\frac{1}{g(x)}|A|^p
%-C\Big(|b|^p+1\Big)
\leq F\Big(A,a_1,\ldots,a_{n-1},b,x\Big) \leq
g(x)\bigg(|A|^p+\sum_{j=1}^{n-1}|a_j|^{p}+|b|^p+1\bigg)\quad
\text{for every}\;\;\big(A,a_1,a_2,\ldots,a_{n-1},b,x\big) \,.
\end{multline}
Assume also that for every $x_0\in\O$ and every $\tau>0$ there
exists $\alpha>0$ satisfying
%for every $a\in \R^{m\times N}$, $b\in\R^m$ and $x\in\R^N$
%satisfying $|x-x_0|<\delta$ we have
\begin{multline}\label{vcjhfjhgjkgkgjgghjfhfhfmjghjvjhvjhgh}
F\big(a_1,a_2,\ldots, a_n,b,x\big)-F\big(a_1,a_2,\ldots,
a_n,b,x_0\big)\geq -\tau F\big(a_1,a_2,\ldots, a_n,b,x_0\big)\\
\forall\, a_1\in \times\R^{m\times N^n}\; \ldots\;\forall\,
a_n\in\R^{m\times N}\;\;\forall\, b\in\R^{m}\;\;\forall\,
x\in\R^N\;\;\text{such that}\;\;|x-x_0|<\alpha\,.
\end{multline}
%Furthermore, let $F\in C^1\big(\R^{m\times N^n}\times\R^{m\times
%N^{(n-1)}}\times\ldots\times\R^{m\times N}\times \R^m,\R\big)$, be
%such that $F\geq 0$ and there exist $C>0$ and $p\geq 1$ satisfying
%\begin{multline}\label{hgdfvdhvdhfvjjjjiiiuyyyjitghujtrnew}
%\frac{1}{C}|A|^p
%%%%-C\Big(|b|^p+1\Big)
%\leq F\Big(A,a_1,\ldots,a_{n-1},b\Big) \leq
%C\bigg(|A|^p+\sum_{j=1}^{n-1}|a_j|^{p}+|b|^p+1\bigg)\quad \text{for
%every}\;\;\big(A,a_1,a_2,\ldots,a_{n-1},b\big).
%\end{multline}
Next let $\varphi\in BV(\R^N,\R^{m})\cap L^\infty$ be such that $\|D
\varphi\|(\partial\Omega)=0$ and $F\big(0,0,\ldots,0,
\varphi(x),x\big)=0$ for a.e. $x\in\O$.
%Let $\O$, $D$, $\vec A$, $F$, $q$, $v$, $\vec n$, and $\{\vec
%A\cdot\nabla v\}^\pm$ as above.
Then for every
%sequence $\e_n\to 0^+$ as $n\to 0^+$ and every sequence
$\{\varphi_\e\}_{\e>0}\subset W^{n,p}_{loc}(\O,\R^m)$ such that
$\varphi_\e\to \varphi$ in $L^p_{loc}(\O,\R^m)$ as $\e\to 0^+$, we
will have
\begin{multline}\label{a1a2a3a4a5a6a7s8hhjhjjhjjjjjjkkkkgenhjhhhhjtjurtnew}
\liminf_{\e\to 0^+}\frac{1}{\e}\int_\O F\bigg(\,\e^n\nabla^n
\varphi_\e(x),\,\e^{n-1}\nabla^{n-1}\varphi_\e(x),\,\ldots,\,\e\nabla \varphi_\e(x),\, \varphi_\e(x),\,x\bigg)dx\\
\geq \int_{\O\cap J_\varphi}\bar
E^{(n)}_{per}\Big(\varphi^+(x),\varphi^-(x),\vec \nu(x),x\Big)d
\mathcal H^{N-1}(x)\,,
\end{multline}
where
\begin{multline}\label{L2009hhffff12kkkhjhjghghgvgvggcjhggghtgjutnew}
\bar E^{(n)}_{per}\Big(\varphi^+,\varphi^-,\vec \nu,x\Big)\;:=\;\\
\inf\Bigg\{\int_{I_{\vec \nu}}\frac{1}{L} F\bigg(L^n\,\nabla^n
\zeta(y),\,L^{n-1}\,\nabla^{n-1} \zeta(y),\,\ldots,\,L\,\nabla
\zeta(y),\,\zeta(y),\,x\bigg)\,dy:\;\; L\in(0,+\infty)\,,\;\zeta\in
\mathcal{S}^{(n)}(\varphi^+,\varphi^-,I_{\vec\nu})\Bigg\}\,,
\end{multline}
with
\begin{multline}\label{L2009Ddef2hhhjjjj77788hhhkkkkllkjjjjkkkhhhhffggdddkkkgjhikhhhjjddddhdkgkkkhgghjhjhjkjtjytrjnew}
\mathcal{S}^{(n)}\big(\varphi^+,\varphi^-,I_{\vec \nu}\big):=
%\mathcal{D}\Big(\vec A,(\vec A\cdot\nabla v)^-(x),(\vec A\cdot\nabla v)^+(x),\vec\nu(x),\vec k_2(x),\vec k_3(x),\ldots,\vec k_{N}(x)\Big):=\\
\bigg\{\zeta\in C^n(\R^N,\R^m):\;\;\zeta(y)=\varphi^-\;\text{ if }\;y\cdot\vec \nu\leq-1/2,\\
\zeta(y)=\varphi^+\;\text{ if }\; y\cdot\vec \nu\geq 1/2\;\text{ and
}\;\zeta\big(y+\vec \nu_j\big)=\zeta(y)\;\;\forall j=2,
%3,
\ldots,
N\bigg\}\,,
\end{multline}
Here $I_{\vec \nu}:=\{y\in\R^N:\;|y\cdot \vec\nu_j|<1/2\;\;\;\forall
j=1,
%2,
\ldots, N\}$, where
$\{\vec\nu_1,
%\vec\nu_2,
\ldots,\vec\nu_N\}\subset\R^N$ is an
orthonormal base in $\R^N$ such that $\vec\nu_1:=\vec \nu$.
Moreover, there exists e sequence $\{\psi_\e\}_{\e>0}\subset
C^\infty(\R^N,\R^m)$ such that
$\int_\O\psi_\e(x)dx=\int_\O\varphi(x)dx$, $\psi_\e\to \varphi$ in
$L^p(\O,\R^m)$ as $\e\to 0^+$ and we have
\begin{multline}\label{a1a2a3a4a5a6a7s8hhjhjjhjjjjjjkkkkgenhjhhhhjtjurtgfhfhfjfjfjnew}
\lim_{\e\to 0^+}\frac{1}{\e}\int_\O F\bigg(\,\e^n\nabla^n
\psi_\e(x),\,\e^{n-1}\nabla^{n-1}\psi_\e(x),\,\ldots,\,\e\nabla \psi_\e(x),\, \psi_\e(x),\,x\bigg)dx\\
= \int_{\O\cap J_\varphi}\bar
E^{(n)}_{per}\Big(\varphi^+(x),\varphi^-(x),\vec \nu(x),x\Big)d
\mathcal H^{N-1}(x)\,.
\end{multline}
\end{theorem}
\appendix
\section{Notations and basic results about $BV$-functions}
\label{sec:pre}

\noindent$\bullet$ For given a real topological linear space $X$ we
denote by $X^*$ the dual space (the space of continuous linear
functionals from $X$ to $\R$).

\noindent$\bullet$ For given $h\in X$ and $x^*\in X^*$ we denote by
$\big<h,x^*\big>_{X\times X^*}$ the value in $\R$ of the functional
$x^*$ on the vector $h$.

\noindent$\bullet$ For given two normed linear spaces $X$ and $Y$ we
denote by $\mathcal{L}(X;Y)$ the linear space of continuous
(bounded) linear operators from $X$ to $Y$.

\noindent$\bullet$ For given $A\in\mathcal{L}(X;Y)$  and $h\in X$ we
denote by $A\cdot h$ the value in $Y$ of the operator $A$ on the
vector $h$.

\noindent$\bullet$ For given two reflexive Banach spaces $X,Y$ and
$S\in\mathcal{L}(X;Y)$ we denote by $S^*\in \mathcal{L}(Y^*;X^*)$
the corresponding adjoint operator, which satisfies
\begin{equation*}\big<x,S^*\cdot y^*\big>_{X\times X^*}:=\big<S\cdot
x,y^*\big>_{Y\times Y^*}\quad\quad\text{for every}\; y^*\in
Y^*\;\text{and}\;x\in X\,.
\end{equation*}

\noindent$\bullet$ Given open set $G\subset\R^N$ we denote by
$\mathcal{D}(G,\R^d)$  the real topological linear space of
compactly supported $\R^d$-valued test functions i.e.
$C^\infty_c(G,\R^d)$ with the usual topology.

\noindent$\bullet$ Denote
$\mathcal{D}'(G,\R^d):=\big\{\mathcal{D}(G,\R^d)\big\}^*$ (the space
of $\R^d$-valued distributions in $G$).

\noindent$\bullet$ Given $h\in\mathcal{D}'(G,\R^d)$ and
$\delta\in\mathcal{D}(G,\R^d)$ we denote the value in $\R$ of the
distribution $h$ on the test function $\delta$ by
$<\delta,h>:=\big<\delta,h\big>_{\mathcal{D}(G,\R^d)\times
\mathcal{D}'(G,\R^d)}$.

\noindent$\bullet$ Given a linear operator $\vec
A\in\mathcal{L}(\R^d;\R^k)$ and a distribution
$h\in\mathcal{D}'(G,\R^d)$ we denote by $\vec A\cdot h$ the
distribution in $\mathcal{D}'(G,\R^k)$ defined by
\begin{equation*}
<\delta,\vec A \cdot h>:=<\vec A^*\cdot
\delta,h>\quad\quad\forall\delta\in\mathcal{D}(G,\R^k).
\end{equation*}

\noindent$\bullet$ Given $h\in\mathcal{D}'(G,\R^d)$ and
$\delta\in\mathcal{D}(G,\R)$ by $<\delta,h>$ we denote the vector in
$\R^d$ which satisfy $<\delta,h>\cdot \vec e:=<\delta\vec e,h>$ for
every $\vec e\in\R^d$.

\noindent$\bullet$ For a $p\times q$ matrix $A$ with $ij$-th entry
$a_{ij}$ and for a $q\times d$ matrix $B$ with $ij$-th entry
$b_{ij}$ we denote by $AB:=A\cdot B$ their product, i.e. the
$p\times d$ matrix, with $ij$-th entry
$\sum\limits_{k=1}^{q}a_{ik}b_{kj}$.

\noindent$\bullet$ We identify a $\vec u=(u_1,\ldots,u_q)\in\R^q$
with the $q\times 1$ matrix having $i1$-th entry $u_i$, so that for
a $p\times q$ matrix $A$ with $ij$-th entry $a_{ij}$ and for $\vec
v=(v_1,v_2,\ldots,v_q)\in\R^q$ we denote by $A\,\vec v :=A\cdot\vec
v$ the $p$-dimensional vector $\vec u=(u_1,\ldots,u_p)\in\R^p$,
given by $u_i=\sum\limits_{k=1}^{q}a_{ik}v_k$ for every $1\leq i\leq
p$.

\noindent$\bullet$ As usual $A^T$ denotes the transpose of the
matrix $A$.

\noindent$\bullet$ For $\vec u=(u_1,\ldots,u_p)\in\R^p$ and $\vec
v=(v_1,\ldots,v_p)\in\R^p$ we denote by $\vec u\vec v:=\vec
u\cdot\vec v:=\sum\limits_{k=1}^{p}u_k v_k$ the standard scalar
product. We also note that $\vec u\vec v=\vec u^T\vec v=\vec v^T\vec
u$ as products of matrices.

\noindent$\bullet$ For $\vec u=(u_1,\ldots,u_p)\in\R^p$ and $\vec
v=(v_1,\ldots,v_q)\in\R^q$ we denote by $\vec u\otimes\vec v$ the
$p\times q$ matrix with $ij$-th entry $u_i v_j$.

\noindent$\bullet$ For any $p\times q$ matrix $A$ with $ij$-th entry
$a_{ij}$ and $\vec v=(v_1,v_2,\ldots,v_d)\in\R^d$ we denote by
$A\otimes\vec v$ the $p\times q\times d$ tensor with $ijk$-th entry
$a_{ij}v_k$.

\noindent$\bullet$ Given a vector valued function
$f(x)=\big(f_1(x),\ldots,f_k(x)\big):\O\to\R^k$ ($\O\subset\R^N$) we
denote by $Df$ or by $\nabla_x f$ the $k\times N$ matrix with
$ij$-th entry $\frac{\partial f_i}{\partial x_j}$.

\noindent$\bullet$ Given a matrix valued function
$F(x):=\{F_{ij}(x)\}:\R^N\to\R^{k\times N}$ ($\O\subset\R^N$), we
denote $div\,F:=(l_1,\ldots,l_k)\in\R^k$, where
$l_i=\sum\limits_{j=1}^{N}\frac{\partial F_{ij}}{\partial x_j}$.

\noindent$\bullet$ Given a matrix valued function
$F(x)=\big\{f_{ij}(x)\big\}(1\leq i\leq p,\,1\leq j\leq
q):\O\to\R^{p\times q}$ ($\O\subset\R^N$) we denote by $DF$ or by
$\nabla_x F$ the $p\times q\times N$ tensor with $ijk$-th entry
$\frac{\partial f_{ij}}{\partial x_k}$.

\noindent$\bullet$ Given a vector measure $\mu=(\mu_1,\ldots,\mu_k)$
(where $\forall j=1,\ldots,k$ $\mu_j$ is a finite signed measure) we
denote by $\|\mu\|(E)$ the total variation of $\mu$ on the set $E$.

\noindent$\bullet$ For any $\mu$-measurable function $f$, we define
the product measure $f\cdot\mu$ by: $f\cdot\mu(E)=\int_E f\,d\mu$,
for every $\mu$-measurable set $E$.

 In what follows we present some known results on BV-spaces. We rely mainly on the book \cite{amb}
by Ambrosio, Fusco and Pallara. We begin by introducing some
notation. For every $\vec\nu\in S^{N-1}$ (the unit sphere in $\R^N$)
and $R>0$ we set
\begin{align}
B_{R}^+(x,\vec\nu)&=\{y\in\R^N\,:\,|y-x|<R,\,
(y-x)\cdot\vec\nu>0\}\,,\label{eq:B+}\\
B_{R}^-(x,\vec\nu)&=\{y\in\R^N:|y-x|<R,\,
(y-x)\cdot\vec\nu<0\}\,.\label{eq:B-}
%\\
%H_+(x,\vec\nu)&=\{y\in\R^N:(y-x)\cdot\vec\nu>0\}\,,\label{HN+}\\
%H_-(x,\vec\nu)&=\{y\in\R^N:(y-x)\cdot\vec\nu<0\}\label{HN-}\\
%\intertext{and} H^0_{\vec \nu}&=\{ y\in\R^N\,:\, y\cdot\vec
%\nu=0\}\label{HN}\,.
\end{align}
\begin{definition}
Let $\Omega$ be a domain in $\R^N$ and let $f\in L^1(\Omega,\R^m)$.
We say that $f\in BV(\Omega,\R^m)$ if the following quantity is
finite:
\begin{equation*}
\int_\Omega|Df|:= \sup\bigg\{\int_\Omega f\cdot\Div\f \,dx :\,\f\in
C^1_c(\Omega,\R^{m\times N}),\;|\f(x)|\leq 1\;\forall x \bigg\}.
\end{equation*}
\end{definition}
\begin{definition}\label{defjac889878}
Let $\Omega$ be a domain in $\R^N$. Consider a function
$f\in L^1_{loc}(\Omega,\R^m)$ and a point $x\in\Omega$.\\
i) We say that $x$ is an {\em approximate continuity point} of $f$
if there exists $z\in\R^m$ such that
$$\lim\limits_{\rho\to
0^+}\frac{\int_{B_\rho(x)}|f(y)-z|\,dy} {\rho^N}=0.$$
In this case
we denote $z$ by $\tilde{f}(x)$. The set of  approximate continuity
points of
$f$ is denoted by $G_f$.\\
ii) We say that $x$ is an {\em approximate jump point} of $f$ if
there exist $a,b\in\R^m$ and $\vec\nu\in S^{N-1}$ such that $a\neq
b$ and
\begin{equation*}
%\label{aprplmin}
\lim\limits_{\rho\to
0^+}\frac{\int_{B_\rho^+(x,\vec\nu)}|f(y)-a|\,dy}
{\mathcal{L}^N\big(B_\rho(x)\big)}=0,\quad \lim\limits_{\rho\to
0^+}\frac{\int_{B_\rho^-(x,\vec\nu)}|f(y)-b|\,dy}
{\mathcal{L}^N\big(B_\rho(x)\big)}=0.
\end{equation*}
%$$\lim\limits_{\rho\to
%0^+}\frac{\int_{B_\rho(x)}\big|\,f(y)-\chi(a,b,\vec\nu)(y)\,\big|\,dy}{\rho^N}=0,$$
%where $\chi(a,b,\vec\nu)$ is defined by
%\begin{equation*}
%%%\label{GYGYGY}
%\chi(a,b,\vec\nu)(y):=
%\begin{cases}
%b\quad\text{if }\vec\nu\cdot y<0,\\
%a\quad\text{if }\vec\nu\cdot y>0.
%\end{cases}
%\end{equation*}
The triple $(a,b,\vec\nu)$, uniquely determined, up to a permutation
of $(a,b)$ and a change of sign of $\vec\nu$, is denoted by
$(f^+(x),f^-(x),\vec\nu_f(x))$. We shall call $\vec\nu_f(x)$ the
{\em approximate jump vector} and we shall sometimes write simply
$\vec\nu(x)$ if the reference to the function $f$ is clear. The set
of approximate jump points is denoted by $J_f$. A choice of
$\vec\nu(x)$ for every $x\in J_f$ determines an orientation of
$J_f$. At an approximate continuity point $x$, we shall use the
convention $f^+(x)=f^-(x)=\tilde f(x)$.
\end{definition}
\begin{theorem}[Theorems 3.69 and 3.78 from \cite{amb}]\label{petTh}
Consider an open set $\Omega\subset\R^N$ and $f\in BV(\Omega,\R^m)$.
Then:\\
\noindent i) $\mathcal{H}^{N-1}$-a.e. point in
$\Omega\setminus J_f$ is a point of approximate continuity of $f$.\\
\noindent ii) The set $J_f$  is
$\sigma$-$\mathcal{H}^{N-1}$-rectifiable Borel set, oriented by
$\vec\nu(x)$. I.e. $J_f$ is $\sigma$-finite with respect to
$\mathcal{H}^{N-1}$, there exist countably many $C^1$ hypersurfaces
$\{S_k\}^{\infty}_{k=1}$ such that
$\mathcal{H}^{N-1}\Big(J_f\setminus\bigcup\limits_{k=1}^{\infty}S_k\Big)=0$,
and for $\mathcal{H}^{N-1}$-a.e. $x\in J_f\cap S_k$, the approximate
jump vector $\vec\nu(x)$ is  normal to $S_k$ at the point $x$.
\\ \noindent iii)
$\big[(f^+-f^-)\otimes\vec\nu_f\big](x)\in
L^1(J_f,d\mathcal{H}^{N-1})$.
\end{theorem}

\end{comment}

\end{document}